\documentclass{amsart}
\usepackage{amsmath,amsthm,amssymb,latexsym}
\usepackage{paralist}
\usepackage{graphicx}
\usepackage[colorlinks=true]{hyperref}
\hypersetup{urlcolor=blue, citecolor=red}

\newtheorem{theorem}{Theorem}[section]
\newtheorem{corollary}{Corollary}
\newtheorem{lemma}[theorem]{Lemma}
\newtheorem{proposition}{Proposition}

\newtheorem{MainTheorem}{Theorem}
\newtheorem{MainCorollary}[MainTheorem]{Corollary}
\newtheorem{MainConjecture}[MainTheorem]{Conjecture}
\theoremstyle{definition}
\newtheorem{definition}[theorem]{Definition}
\newtheorem{remark}{Remark}
\newtheorem{example}{Example}

\newcommand{\eps}{\varepsilon}
\usepackage{enumerate}
\makeatletter
\renewcommand\theenumi{\@alph\c@enumi}

\def\@ls#1#2{\@mathmeasure\z@\displaystyle{#2}%
             \@mathmeasure\@ne\displaystyle{#1}%
             \box\@ne\box\z@%
}
\def\gtso#1{\mathrel{\@ls{_{_{#1}}}{>}}}
\def\geso#1{\mathrel{\@ls{_{_{#1}}}{\geq}}}

\def\leso#1{\mathrel{\le_{_{#1}}}}
\def\Sho{\mbox{\tiny\textup{Sh}}}
\newcommand{\NTL}[1][t]{\N^{\scriptscriptstyle \vee}_{#1}}
\newcommand{\tinf}[1][\empty]{\ifx#1\empty\else #1\cdot \fi2^{\infty}}
\DeclareMathOperator{\Shs}{S\mbox{\tiny\textup{sh}}}
\newcommand{\IN}{\ensuremath{\mathbb{N}}}\let\N\IN
\newcommand{\IZ}{\ensuremath{\mathbb{Z}}}\let\Z\IZ
\newcommand{\IR}{\ensuremath{\mathbb{R}}}\let\R\IR
\newcommand{\IQ}{\ensuremath{\mathbb{Q}}}\let\Q\IQ
\newcommand{\IC}{\ensuremath{\mathbb{C}}}\let\C\IC
\newcommand{\SI}{\ensuremath{\mathbb{S}^1}}
\def\Re{\operatorname{Re}}
\def\Im{\operatorname{Im}}
\newcommand{\ret}{r}
\DeclareMathOperator{\Card}{Card}
\DeclareMathOperator{\Int}{Int}
\DeclareMathOperator{\Bd}{Bd}
\DeclareMathOperator{\diam}{diam}
\DeclareMathOperator{\Per}{Per}
\DeclareMathOperator{\TPer}{Per^{\circ}}
\DeclareMathOperator{\Orb}{Orb}
\DeclareMathOperator{\LOrb}{\mathsf{L}Orb}
\DeclareMathOperator{\Rot}{Rot}
\newcommand{\RotR}{\ensuremath{\Rot_{_{\IR}}}}
\newcommand{\CA}{\mathcal{A}}
\newcommand{\CB}{\mathcal{B}}
\newcommand{\CC}{\mathcal{C}}
\newcommand{\CG}{\mathcal{G}}
\newcommand{\CI}{\mathcal{I}}
\newcommand{\Bo}{\mathring{B}}
\newcommand{\Lifts}{\ensuremath{\mathcal{L}(S)}}
\newcommand{\Li}[1][1]{\ensuremath{\mathcal{L}_{#1}(S)}}
\newcommand{\LL} {\mathcal{L}}
\newcommand{\Cstar}[1][3]{\mathcal{X}_{#1}}
\newcommand{\modi}{\ensuremath{\kern -0.55em\pmod{1}}}
\newcommand{\rhos}[1][F]{\rho_{_{#1}}}
\newcommand{\set}[2]{\ensuremath{\{#1 \,\colon #2\}}}
\newcommand{\maparrow}[1][F]{\xrightarrow[#1]{\hspace*{1.35em}}}
\newcommand{\signedcover}[2][+]{%
  \nolinebreak[4]%
  \xrightarrow[#2]{\hspace*{.25em}#1\hspace*{.1em}}%
  \nolinebreak[4]%
}
\makeatletter
\newcommand{\dotsarrowto}{\longrightarrow\linebreak[0]\dotsm\linebreak[0]}
\newcommand{\arrowto}{\nolinebreak[4]\longrightarrow\nolinebreak[4]}
\def\CPath#1{#1\fletxa}
\def\fletxa>#1{\def\@parm{#1}\def\@dummy{\dots}%
               \ifx\@parm\@dummy\dotsarrowto\else%
               \arrowto #1\fi\@ifnextchar>{\fletxa}{}%
}
\makeatother
\newcommand{\evalat}[1]{\bigr\rvert_{#1}}
\makeatletter
\def\@map#1#2[#3]{\mbox{$#1 \colon #2 \longrightarrow #3$}}
\def\map#1#2{\@ifnextchar [{\@map{#1}{#2}}{\@map{#1}{#2}[#2]}}
\def\@Bchull#1{\langle #1 \rangle}
\def\@chull#1[#2]{\ensuremath{\@Bchull{#1}_{_{#2}}}}
\def\chull#1{\@ifnextchar [{\@chull{#1}}{\ensuremath{\@Bchull{#1}}}}
\newcommand{\case@separation}{\par\par\medskip\par\par}
\newenvironment{case}[2][Case]{\case@separation
   \trivlist
       \item[\hskip\labelsep{\bfseries #1 #2:}]\begin{em}}{\end{em}
   \endtrivlist
}
\newcommand{\titlecase}[1]{\case@separation\noindent{\bfseries #1:\ }}
\newcommand{\mathtitlecase}[1]{\titlecase{$\mathbf{#1}$}}%
\makeatother

\title[On the set of periods of sigma maps] 
      {On the set of periods of sigma maps of degree 1}

\author[Llu\'{\i}s Alsed\`{a} and Sylvie Ruette]{}


\thanks{The first author has been partially supported by the MINECO
grants number MTM2008-01486 and MTM2011-26995-C02-01.
Both authors thank kind invitations of the
\emph{Laboratoire de Math\'ematiques, Universit\'e Paris-Sud 11}
and
\emph{Departament de Matem\`{a}tiques, Universitat Aut\`{o}noma de Barcelona}
that made this paper possible.}

\begin{document}
\maketitle

\centerline{\scshape Llu\'{\i}s Alsed\`{a}}
\medskip
{\footnotesize
 \centerline{Departament de Matem\`{a}tiques,}
 \centerline{Edifici Cc,}
 \centerline{Universitat Aut\`{o}noma de Barcelona,}
 \centerline{08913 Cerdanyola del Vall\`es, Barcelona,}
 \centerline{Spain}
 \centerline{\texttt{alseda@mat.uab.cat}}
}

\medskip

\centerline{\scshape Sylvie Ruette}
\medskip
{\footnotesize
 \centerline{Laboratoire de Math\'ematiques,}
 \centerline{CNRS UMR 8628,}
 \centerline{B\^atiment 425,}
 \centerline{Universit\'e Paris-Sud 11,}
 \centerline{91405 Orsay cedex,}
 \centerline{France}
 \centerline{\texttt{sylvie.ruette@math.u-psud.fr}}
}

\bigskip

\begin{abstract}
We study the set of periods of degree 1 continuous maps from $\sigma$
into itself, where $\sigma$ denotes the space shaped like the letter $\sigma$
(i.e., a segment attached to a circle by one of its endpoints).
Since the maps under consideration have degree 1, the rotation theory can
be used.
We show that, when the interior of the rotation interval contains an
integer, then the set of periods (of periodic points of any rotation
number) is the set of all integers except maybe $1$ or $2$.
We exhibit degree 1 $\sigma$-maps $f$ whose set of periods is a combination
of the set of periods of a degree 1 circle map and the set of periods of
a $3$-star (that is, a space shaped like the letter $Y$).
Moreover, we study the set of periods forced by periodic orbits that do
not intersect the circuit of $\sigma$; in particular, when there exists
such a periodic orbit whose diameter (in the covering space) is at least $1$,
then there exist periodic points of all periods.
\end{abstract}

\section{Introduction}

In this paper we study the set of periods of continuous
maps from the space $\sigma$ to itself, where the space $\sigma$
consists of a circle with a segment attached to it at one of
the segment's endpoints. Our results continue the progression
of results which began with Sharkovskii's Theorem on the characterization
of the sets of periods of periodic points of continuous interval maps
\cite{SharOri,SharTrans}
and continued with the study of the periods of maps of the circle
\cite{BGMY, Block, Mis},
trees \cite{AJM1,AJM2,AJM3,AJM4,Bern,ALMY, BaldLli}
and other graphs \cite{LLl,LPR}.

A full characterization of the sets of periods for
continuous self maps of the graph $\sigma$ having the branching
fixed is given in \cite{LLl}. Our goal is to extend this
result to the general case. The most natural approach is to follow the
strategy used in the circle case which consists in dividing the
problem according to the degree of the map \cite{BGMY, Block, Mis}.
The cases considered for the circle are degree different from $\{-1,
0, 1\}$, and separately the cases of degree $0$, $-1$ and $1$.
A characterization of the set of periods of the class of continuous
maps from the space $\sigma$ to itself with degree different from
$\{-1, 0, 1\}$ can be found in \cite{Mal}. In this paper, we aim at
studying the set of periods of continuous $\sigma$-maps of degree 1.
Following again the strategy of the circle case, we shall work in the
covering space and we shall use rotation theory. This theory for graphs
with a single circuit was developed in \cite{AlsRue2008}; the
current paper is thus an application of the theory developed there.

We shall follow three main directions in studying the set of periods of
$\sigma$-maps. The first very natural one follows from the trivial
observation that the space $\sigma$ contains both a circle and a subset
homeomorphic to a $Y$ (also called a $3$-star). It is quite obvious
that there exist $\sigma$-maps of degree $1$ whose set of periods is equal
to the set of periods of any given degree $1$ circle map, as well as the set
of periods of any given $3$-star map. We shall show that there exist
$\sigma$-maps $f$ whose set of periods is any combination of both kinds
of sets, provided that $0$ is an endpoint of the rotation interval of $f$:
the whole rotation interval gives a set of periods as for circle maps
whereas the set of periods of a given $3$-star map appears with rotation
number $0$.

The second direction is the study of periodic orbits that do not
intersect the circuit of the space $\sigma$; this study is necessary
because the rotation interval does not capture well the behaviors of
such orbits. We shall show that the existence of such a periodic orbit
of period $n$ implies all periods less than $n$ for the Sharkovsky
ordering; this is quite natural because this ordering rules the sets
of periods of interval maps and the branch of $\sigma$ is an interval.
Moreover, we shall show that if, in the covering space, there exists a
periodic orbit living in the branches and with diameter greater than
or equal to $1$, then the set of periods contains necessarily all
integers.

The third direction focuses on the rotation number $0$.
For degree $1$ circle maps, the strategy is to characterize
the set of periods for a given rotation number $p/q$ in the interior
of the rotation interval, which comes down to do the same for the
rotation number $0$ for another map. Unfortunately, mimicking this
strategy fails for $\sigma$-maps because the set of periods of rotation
number $0$ can be complicated and we do not know how to describe it.
However, we shall characterize the set of periods (of any rotation number)
when $0$ in the interior of the rotation interval of a $\sigma$ map:
in this case, the set of periods is, either $\IN$, or $\IN\setminus\{1\}$,
or  $\IN\setminus\{2\}$.

Moreover, we shall stress some difficulties that appear when one tries to
follow the same strategy as for degree $1$ circle maps.

In the next section, we state and discuss the main results of the
paper, after introducing the necessary notation to formulate them.

\section{Definitions and statements of the main results}\label{sec:statements}
\subsection{Covering space, periodic (mod 1) points,
            rotation set}\label{ss:coveringS}
As it has been said, in this paper we want to study the set of periods of the
$\sigma$-maps. Given a map {\map{f}{X}},
we say that a point $x \in X$ is \emph{periodic of period $n$}
if $f^n(x) = x$ and $f^i(x) \ne x$ for all $i=1,2,\dots,n-1$.
Moreover, for every $x \in X$, the set
\[
   \Orb(x,f) := \set{f^{n}(x)}{ n \ge 0}
\]
is called the \emph{orbit of $x$}. Observe that if $x$ is periodic
with period $n$, then we have $\Card(\Orb(x,f)) = n$
(where $\Card(\cdot)$ denotes the cardinality of a finite set).
The set of periods of all periodic points of $f$ will be
denoted by $\TPer(f)$.

Following the strategy of the circle it is advisable to work
in the covering space  and we shall use the rotation theory developed in
\cite{AlsRue2008}.
We also shall consider periodic {\modi} points and orbits for liftings instead
of the true ones defined above.
The results obtained in this setting can be obviously pushed down
to the original map and space.

We start by introducing the framework to use the rotation theory
developed in \cite{AlsRue2008}.

We consider the universal covering of $\sigma$. More precisely, we take
the following realization of the covering space (see Figure~\ref{FigS}):
\[
 S = \IR \cup B,
\]
where
\[
B := \set{z \in \IC}{\Re(z) \in \IZ \text{ and }\Im(z) \in [0,1]},
\]
and $\Re(z)$ and $\Im(z)$ denote respectively the real and imaginary part
of a complex number $z$.
The set $B$ is called the \emph{set of branches of $S$}.
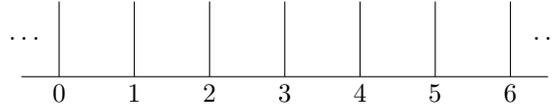
\begin{figure}[h]
\setlength{\unitlength}{1cm}
\begin{center}
\begin{picture}(7,2)(-0.5,-0.5)
\put(-0.5,0){\line(1,0){7}}
\multiput(0,0)(1,0){7}{\line(0,1){1}}
\put(0,-0.1){\makebox(0,0)[t]{0}}
\put(1,-0.1){\makebox(0,0)[t]{1}}
\put(2,-0.1){\makebox(0,0)[t]{2}}
\put(3,-0.1){\makebox(0,0)[t]{3}}
\put(4,-0.1){\makebox(0,0)[t]{4}}
\put(5,-0.1){\makebox(0,0)[t]{5}}
\put(6,-0.1){\makebox(0,0)[t]{6}}
\put(-0.2,0.5){\makebox(0,0)[r]{$\cdots$}}
\put(6.3,0.5){\makebox(0,0)[l]{$\cdots$}}
\end{picture}
\end{center}
\caption{The space $S$, universal covering of $\sigma$.}\label{FigS}
\end{figure}

Observe that $S \subset \C$ and that $\R$ actually means the copy of the
real line embedded in $\C$ as the real axis.
Also, the maps $z \mapsto z + n$ for $n \in \Z$
(since $S\subset \IC$, the operation $+$ is just the usual one in $\IC$)
are the covering (or deck) transformations. So, they leave $S$ invariant:
$S = S + \Z = \set{z+k}{z \in S \text{ and } k \in \Z}.$
Moreover, the real part function $\Re$ defines a retraction from $S$ to $\IR$.
That is, $\Re(z) = z$ for every $z \in \IR$ and, when $z \in S\setminus \IR$,
then $\Re(z)$ gives the integer in the base of the segment where $z$
lies.

For every $m \in \IZ$, we set
\begin{align*}
B_m   &:= \set{z \in S}{\Re(z) = m \text{ and }\Im(z) \in [0,1]}
          = S \cap \Re^{-1}(m), \text{ and}\\
\Bo_m &:= B_m \setminus \{m\}.
\end{align*}
Each of the sets $B_m$ is called \emph{a branch of $S$}.
Clearly,
$B = \cup_{m \in \IZ} B_m$, $B_m \cap \IR = \{m\}$ and
$\Bo_m \cap \IR = \emptyset$.
Each branch $B_m$ is endowed with a linear ordering $\le$ as follows.
If $x,y \in B_m$, we write $x < y$ if and only if $\Im(x) < \Im(y)$.

In what follows, $\Li[d]$ will denote the class of continuous
maps $F$ from $S$ into itself of degree $d \in \IZ$, that is,
$F(z+1)=F(z)+d$ for all $z \in S$.
We also set $\Lifts = \cup_{d \in \IZ} \Li[d]$.
Observe that $\Re \in \Li$ and thus, if $F\in \Li$, then
$\Re \circ F^n \in \Li$ for every $n \in \N$.

Let $F \in \Lifts$ and $z\in S$. The set
\[
   \set{F^{n}(z)+m}{ n \ge 0 \text{ and } m\in \IZ}
\]
is called the \emph{lifted orbit of $z$}, and denoted by
$\LOrb(z,F)$. The point $z$ is called \emph{periodic \modi}
if there exists $n \in \IN$ such that $F^n(z) \in z+\IZ$.
The \emph{period \modi} of $z$ is the least positive integer $n$
satisfying this property, that is,
$F^n(z) \in z+\IZ$ and
$F^i(z) \notin z+\IZ$ for all $1\leq i\leq n-1$.
When $z$ is periodic {\modi}, then
$\LOrb(z,F)$ is also called a \emph{lifted periodic orbit}.
It is not difficult to see that, for all $k \in \IZ$,
$\Card\left(\LOrb(z,F) \cap \Re^{-1}\bigl([k,k+1)\bigr)\right)$
coincides with the period {\modi} of $z$.
The set of all periods of the periodic {\modi} points of
$F \in \Lifts$ will be denoted by $\Per(F)$.

Wen talking about periodic points and periodic {\modi} points
we shall sometimes write \emph{true period} or
\emph{true periodic point} to emphasize that they are not {\modi}.

Let $\map{\pi}{S}[\sigma]$ be the standard projection from $S$ to $\sigma$,
that is, $\pi\evalat{\Re^{-1}([0,1))}$ is continuous onto and one-to-one
and $\pi(z) = \pi(z + k)$ for all $z \in S$ and all $k \in \Z$.
Clearly, for every $F\in \Lifts$, $\pi_{\star}F := \pi \circ F \circ \pi^{-1}$
is a well defined continuous self map of $\sigma$. Reciprocally, for every
continuous map $f$ from $\sigma$ into itself, there exists a lifting
$F\in\Lifts$ such that $\pi_{\star}F=f$, and this lifting is unique up to an
integer (that is, if $G$ is another lifting, there exists $k\in\IN$
such that $G=F+k$).
Moreover,
$\pi(\LOrb(z,F)) = \Orb(\pi(z),\pi_{\star}F),$ and
$z$ is a periodic {\modi} point of $F$ of period $n$ if and only
if $\pi(z)$ is a true periodic point of $\pi_{\star}F$ of (true) period $n.$
Consequently, $\Per(F) = \TPer(\pi_{\star}F)$ and \emph{characterizing the sets of
periods {\modi} of maps from $\Lifts$ is equivalent to characterizing
the sets of periods of continuous self maps of $\sigma$}.

This paper will deal with maps of degree $1$, for which rotation
numbers can be defined.
Next we recall the notion of rotation number in our setting and its
basic properties.

\begin{definition}
Let $F\in \Li$ and $z\in S$. We define the
\emph{rotation number of $z$} as
\[
   \rhos(z) :=  \lim_{n\to+\infty} \frac{\Re(F^n(z))-\Re(z)}{n}
\]
if the limit exists.
We also define the following \emph{rotation sets of $F$}:
\begin{align*}
\Rot(F)   &= \set{\rhos(z)}{z \in S},\\
\RotR(F)  &= \set{\rhos(z)}{z \in \R}.
\end{align*}
\end{definition}

For every $z \in S,$ $k\in\IZ$ and $n\in\IN$, it follows that
$\rhos(z+k)=\rhos(z),$
$\rhos[(F+k)](z)=\rhos(z)+k$ and
$\rhos[F^n](z)=n\rhos(z)$
(c.f \cite[Lemma~1.10]{AlsRue2008}). The second property implies that, if
$F$, $G$ are two liftings of the same continuous map from
$\sigma$ into itself, then their rotation sets differ from an integer
($\exists k\in\IZ$ such that $G=F+k$, and hence $\Rot(G)=\Rot(F)+k$).

Unfortunately, the set $\Rot(F)$ may not be connected as it
has been shown in \cite{AlsRue2008}. However, the set
$\RotR(F)$, which is a subset of $\Rot(F)$, has better properties.
Next result is \cite[Theorem~3.1]{AlsRue2008}.

\begin{theorem}\label{theo:RotR}
For every $F\in\Li,$ $\RotR(F)$ is a non empty compact interval.
Moreover, if $\alpha\in \RotR(F)$, then
there exists a point $x\in\IR$ such that $\rhos(x)=\alpha$ and
$F^n(x)\in\IR$ for infinitely many $n$.
If $p/q\in\RotR(F)$, then there exists a periodic {\modi} point
$x\in S$ with $\rhos(x)=p/q$.
\end{theorem}

\begin{definition}
Given $F \in \Li$ and $\alpha \in \IR$, let
$\Per(\alpha,F)$ denote the set of periods of all periodic
{\modi} points of $F$ whose rotation number is $\alpha$.
\end{definition}

It is easy to see that every periodic {\modi} point has a rational rotation
number (see also Lemma~\ref{lem:FF+k}(e)). Therefore, Theorem~\ref{theo:RotR}
implies that, when $\alpha\in\RotR(F)$,
$\Per(\alpha,F)$ is non-empty if and only if $\alpha\in\IQ$.

Observe that the class of maps $F \in \Li$ such that $F(\R) \subset
\R$ and $F(B_m) = F(m)$ for every $m \in \Z$ can be identified with
the class of liftings of continuous circle maps of degree $1$.
Therefore any possible set of periods of a continuous circle map of
degree $1$ can be a set of periods of a map in $\Li$.
On the other hand, set $Y_0:=B_0 \cup [-1/3,1/3]$ (this space is called
a \emph{$3$-star}) and  consider the class of maps $F \in \Li$ such
that $F(Y_0) \subset Y_0$,
$F(x) \in Y_0 \cup [1/3,x)$ for every $x \in [1/3,1/2)$ and
$F(x) \in (Y_0 + 1 )\cup (x,2/3]$ for every $x \in (1/2,2/3]$
(in particular $F(1/2) = 1/2$).
This implies that $\Per(F) = \TPer(F\evalat{Y_0})$ and thus,
every possible set of periods of a map from a $3$-star into itself can
be a set of periods of a map from $\Li$. Clearly, this includes the
sets of periods of interval maps.
Moreover, it might happen that this
phenomenon occurs for rotation numbers different from 0, that is,
there may exist a map from $\Cstar$ with set of periods $A\subset
\IN$,
$p \in \Z,$ $q \in \N$ and $\widetilde{S} \subset S$ such that
$\TPer((F^q -p)\evalat{\widetilde{S}}) = A$ and
$\Per(p/q,F) = q \cdot \TPer((F^q -p)\evalat{\widetilde{S}}).$
Therefore, a natural conjecture for the structure of the set of periods of
maps from $\Li$ could be that it is the union of the set of periods of a
circle map of degree $1$ with some sets of the form $q\cdot \TPer(f)$
with $q\in\N$ and $f\in \Cstar$ much in the spirit of the
characterization of the set of periods for circle maps of degree one.
We shall see that it is unclear that all possibilities can occur.

To explain these ideas in detail, and to state the main results of
the paper, we need to recall the characterization of the sets of
periods of circle maps of degree $1$ and of star maps. We are going to
do this in the next two subsections; we shall also introduce
the necessary notations.

\subsection{Tree maps}

A \emph{tree} is a compact uniquely arcwise connected space which is a
point or a union of a finite number of segments glued together at some
of their endpoints (by a \emph{segment} we mean any space homeomorphic
to $[0,1]$). Any continuous map $f$ from a tree into itself is called
a \emph{tree map}.
The space $S$ is often called an infinite tree by similarity.

Consider a tree $T$ or the space $S$.
For every $x$ in $T$ or $S$, the \emph{valence} of $x$ is the
number of connected components of $T\setminus\{x\}$. A point of
valence different from $2$ is called a \emph{vertex}.
A point of valence $1$ is called an \emph{endpoint}.
The points of valence greater than or equal to $3$ (that is, vertices
that are not endpoints) are called the \emph{branching points}.
If $K$ is a subset of $T$ or $S$, then $\chull{K}$ denotes the
\emph{convex hull} of $K$, that is, the smallest closed
connected set containing $K$ (which is well defined since the trees
and the space $S$ are uniquely arcwise connected). An \emph{interval}
in $T$ or $S$ is any subset homeomorphic to an interval of $\IR$. For
a compact interval $I$, it is equivalent to say that there exist two
points $a,b$ such that $I=\chull{a,b}$; in this case, $\{a,b\}=\Bd(I)$
(where $\Bd(\cdot)$ denotes the boundary of a set). When a distance is
needed in a tree or $S$, we use a taxicab metric, that is, a distance
$d$ such that, if $z\in\chull{x,y}$, then $d(x,y)=d(x,z)+d(z,y)$. In
$S$, the distance is simply defined by
\[
  d(x,y) =  \begin{cases}
      |x-y| & \text{if $x,y\in B_m;\ m\in\IZ$,}\\
      |x-\Re(x)| + |\Re(x)-\Re(y)| + |y-\Re(y)| & \text{otherwise}
\end{cases}
\]
for every $x,y\in S.$
Consider a compact interval $I$ in an tree $T$ or in $S,$ and a
continuous map {\map{f}{I}[S]}. We say that $f$ is \emph{monotone} if,
either $f(I)$ is reduced to one point, or $f(I)$ is a non degenerate
interval and, given any homeomorphisms {\map{h_1}{[0,1]}[I]},
{\map{h_2}{[0,1]}[f(I)]},
the map {\map{h_2^{-1}\circ f\circ h_1}{[0,1]}}
is monotone.
We say that $f$ is \emph{affine} if $f(I)$ is an interval and there
exists a constant $\lambda$ such that $\forall x,y \in I,$
$d(f(x),f(y))=\lambda d(x,y).$

A tree that is a union of $n \ge 2$ segments whose intersection is
a unique point $y$ of valence $n$ is called an \emph{$n$-star},
and $y$ is called its \emph{central point}. For a fixed $n$,
all $n$-stars are homeomorphic.
In what follows, $X_n$ will denote an $n$-star,
$\Cstar[n]$ the
class of all continuous maps from $X_n$ to itself and
$\Cstar[n]^{\circ}$ the class of all maps from $\Cstar[n]$ that leave
the unique branching point of $X_n$ fixed.

A crucial notion for periodic orbits of maps in $\Cstar[n]$ is the
\emph{type} of an orbit \cite{Bald}.
Let $f \in \Cstar[n]$ and let $P$ be a periodic orbit of $F$.
Let $y$ denote the branching point of $X_n$.
If $y \in P$, then we say that $P$ has \emph{type 1}.
Otherwise, let $\mathrm{Br}$ be the set of branches of $X_n$
that intersect $P$
(by a branch we mean a connected component of $X_n\setminus\{y\}$).
For each $b \in \mathrm{Br}$ we denote by $\mathrm{sm}_b$ the point of
$P \cap b$ closest to $y$ (that is, $\mathrm{sm}_b \in b$ and
$\chull{y,\mathrm{sm}_b} \cap P = \{\mathrm{sm}_b\}$). Then we define
a map {\map{\phi}{\mathrm{Br}}} by letting $\phi(b)$ be the branch of
$\mathrm{Br}$ containing $f(\mathrm{sm}_b).$
Since $\mathrm{Br}$ is a finite set, $\phi$ has periodic orbits.
Each period of a periodic orbit of $\phi$ is called a
\emph{type} of $P$.
Clearly the type may not be unique. However, it is clearly
unique in the case when $P$ has type $n$.

We shall also speak of the type of a (true) periodic orbit $P$
of a map $F\in\Li$ such that $\chull{P}$ is homeomorphic to
$X_n$ (indeed $X_3$). The definition of type extends straightforwardly
to this situation.

We now recall the Sharkovsky total ordering and Baldwin partial
orderings, which are needed to state the characterization of the sets of
periods of star maps.

The \emph{Sharkovsky ordering} $\leso{\Sho}$ is defined on
$\N_{\Sho} = \N \cup \{ \tinf \}$ by:
\begin{align*}
& 3 \gtso{\Sho} 5 \gtso{\Sho} 7 \gtso{\Sho} \dots
2 \cdot 3 \gtso{\Sho} 2 \cdot 5 \gtso{\Sho}
              2 \cdot 7 \gtso{\Sho} \dots\\
&2^2 \cdot 3 \gtso{\Sho} 2^2 \cdot 5 \gtso{\Sho}2^2 \cdot 7 \gtso{\Sho}
             \dots
         \gtso{\Sho} \dots\\
&  \tinf \gtso{\Sho} \dots
   2^n \gtso{\Sho} \dots \gtso{\Sho} 2^4 \gtso{\Sho}
   2^3 \gtso{\Sho} 2^2 \gtso{\Sho} 2 \gtso{\Sho} 1.
\end{align*}
That is, this ordering starts with all the odd numbers greater than 1,
in increasing order,
then $2$ times the odd numbers $>1$,
then $2^2$ times, $2^3$ times, \ldots $2^n$ times the odd
numbers $>1$;
finally the last part of the ordering consists of all powers of
$2$ in decreasing order; the symbol $2^\infty$ being greater than all
powers of $2$ and $1=2^0$ being the smallest element.


For every integer $t\ge 2$, let $\N_t$ denote the set
$(\N \cup \{ \tinf[t] \}) \setminus\{2,3,\dots,t-1\}$ and
$\NTL:=\set{mt}{m \in \N} \cup \{1,\tinf[t]\}.$
Then the \emph{Baldwin partial ordering} $\leso{t}$ is defined in
$\N_t$ as follows.
For all $k,m \in \N_t$, we write $k \leso{t} m$ if one of the following
cases holds:
\begin{enumerate}[(i)]
\item $k=1$ or $k=m$,
\item $k,m \in \NTL \setminus \{1\}$ and $m/t \gtso{\Sho} k/t $,
\item $k \in \NTL$ and $m \notin \NTL$,
\item $k,m \notin \NTL$ and $k = i m + j t$ with $i,j \in \N$,
\end{enumerate}
where in case~(ii) we use the following arithmetic rule for the
symbol $\tinf[t]$: $\tinf[t]/t = \tinf$.

There are two parts in the structure of the orderings $\leso{t}$.
The smallest part consists of all elements of $\NTL$ ordered
as follows. The smallest element is 1. Then all the multiples of $t$
(including $\tinf[t]$) come in the ordering induced by the
Sharkovsky ordering and the largest element of $\NTL$ is $3\cdot t$.
Then the ordering $\geso{t}$ divides $\N_t \setminus \NTL$ into $t-1$
``branches''. The $l$-th branch ($l \in \{1,2,\dots,t-1\}$) is formed
by all positive integers (except $l$) which are congruent to $l$
modulo $t$ in decreasing order. All elements of these branches are
larger than all elements of $\NTL$.

We note that, by means of the inclusion of the symbol $\tinf[t]$,
each subset of $\N_t$ has a maximal element with respect to the
ordering $\leso{t}$. We also note that the ordering $\leso{2}$ on
$\N_2$ coincides with the Sharkovsky ordering on $\N_{\Sho}$ (by
identifying the symbol $\tinf[2]$ with $\tinf$).

A non empty set $A \subset \N_t \cap \N$ is called a \emph{tail of the
ordering $\leso{t}$} if, for all $m \in A$,  we have
$\set{k \in \N}{k \leso{t} m} \subset A$.
Moreover, for all $s\in \N_{\Sho}$, $\Shs(s)$ denotes the initial
segment of the Sharkovsky ordering starting at $s$, that is,
$\Shs(s) = \set{k \in \N}{k \leso{\Sho} s}$.

The following result, due to Baldwin \cite{Bald}, characterizes the set
of periods of star maps.

\begin{theorem}\label{GMT1}
Let $f \in \Cstar[n]$. Then $\TPer(f)$ is a finite union of tails of
the orderings $\geso{t}$ for all $t \in \{2,\ldots, n\}$ (in
particular, $1\in\TPer(f)$). Conversely, if a non empty
set $A$ can be expressed as a finite union of tails of the
orderings $\geso{t}$ with $2 \le t \le n$, then there exists a map $f
\in \Cstar[n]^{\circ}$ such that $\TPer(f) = A$.
\end{theorem}

Note that the case $n=2$ in the above theorem is, indeed,
Sharkovsky's Theorem for interval maps \cite{SharOri}.
Moreover, since every tail of $\geso{t}$ contains $1 \in \TPer(f),$
then the order $\geso{t}$ does not contribute to $\TPer(f)$ if the
tail with respect to $\geso{t}$ in the above lemma is reduced to
$\{1\}$.

\subsection{Circle maps of degree 1}
Let $\SI$ be the unit circle in the complex plane, that is, $\SI =
\set{z \in \C}{|z| = 1}$, and let $\LL_1(\IR)$ denote the class of all
liftings of continuous circle maps of degree one. If $F\in
\LL_1(\IR)$, $\Rot(F)$ denotes the rotation set of $F$ and, by
\cite{Ito}, is a compact non empty interval.

To study the connection between the set of periods and the rotation
interval, we need some additional notation. For  all $c\le d$, we set
$M(c,d):=\set{n \in \N}{c < k/n < d \text{ for some integer $k$}}$.
Notice that we do not assume here that $k$ and $n$ are coprime.
Obviously, $M(c,d)= \emptyset$ if and only if $c=d$.
Given $\rho \in \R$ and $S \subset \N$, we set
\[
\Lambda(\rho, S) = \begin{cases}
    \emptyset        & \text{if $\rho \notin \Q$}, \\
    \set{nq}{q\in S} & \text{if $\rho=k/n$ with $k$ and $n$ coprime}.
\end{cases}
\]

The next theorem recalls Misiurewicz's characterization of the sets of
periods for degree $1$ circle maps (see \cite{Mis,ALM}).

\begin{theorem}\label{S9.5}
Let $F\in \LL_1(\IR)$, and let $\Rot(F) = [c,d]$.
Then there exist numbers $s_c,s_d \in \N_{\Sho}$ such that
$\Per(F)= \Lambda(c,\Shs(s_c)) \cup M(c,d) \cup \Lambda(d,\Shs(s_d))$.
Conversely, for all $c,d \in \R$ with $c \le d$
and all $s_c, s_d\in\N_{\Sho}$, there exists a map $F \in \LL_1(\IR)$ such that
$\Rot(F) = [c,d]$ and
$\Per(F) = \Lambda(c,\Shs(s_c)) \cup M(c,d) \cup \Lambda(d,\Shs(s_d))$.
\end{theorem}

\subsection{Statement of main results}\label{ss:main-statements}
In view of what we said at the end of Subsection~\ref{ss:coveringS},
a reasonable conjecture about
the set of periods for maps from $\Li$ could be the following:

\begin{MainConjecture}\label{WishList}
Let $F\in \Li$ be with $\RotR(F) = [c,d]$.
Then there exist sets $E_c, E_d \subset \IN$ which are finite unions of
of tails of the orderings $\leso{2}$  and $\leso{3}$ such that
\[
 \Per(F) = \Lambda(c,E_c) \cup M(c,d) \cup \Lambda(d,E_d).
\]

Conversely, given $c,d \in \R$ with $c \le d,$
and non empty sets $E_c, E_d \subset \IN$ which are finite union of
of tails of the orderings $\leso{2}$ and $\leso{3},$
there exists a map $F \in \Li$ such that
$\RotR(F) = [c,d]$ and
\[
 \Per(F) = \Lambda(c,E_c) \cup M(c,d) \cup \Lambda(d,E_d).
\]
\end{MainConjecture}

As we shall see, some facts seem to indicate that this conjecture is
not entirely true (though they do not disprove it).  However, we shall
use this conjecture as a guideline: on the one hand, we shall prove
that it is partly true; on the other hand, we shall stress some
difficulties.

We start by discussing the second statement of
Conjecture~\ref{WishList}. This statement holds in two
particular cases, stated in Corollary~\ref{ConverseSigmaCircleCase} and
Theorem~\ref{YinSigma} below. The first one is an easy corollary of
Theorem~\ref{S9.5} and the second one deals with the particular case
when $0$ is an endpoint of the rotation interval.
Recall that $\leso{2}$ coincide with $\leso{\Sho}$.

\begin{MainCorollary}\label{ConverseSigmaCircleCase}
Given $c,d \in \R$ with $c \le d$
and $s_c, s_d\in\N_{\Sho}$, there exists a map $F \in \Li$ such that
$\RotR(F) = \Rot(F) = [c,d]$ and
$\Per(F) = \Lambda(c,\Shs(s_c)) \cup M(c,d) \cup \Lambda(d,\Shs(s_d))$.
\end{MainCorollary}

Notice that, when both $c$ and $d$ are irrational,
Corollary~\ref{ConverseSigmaCircleCase} implies the second statement
of Conjecture~\ref{WishList}. Therefore it remains to consider the
cases when $c$ and/or $d$ are in $\IQ$ and when the order $\leso{3}$
is needed (or equivalently when one refers to the set of periods of
any $3$-star map). The next theorem deals with the case when $c$ (or
$d$) is equal to $0$ (or, equivalently, to an integer) and $\leso{3}$
is needed only for this endpoint.

\begin{MainTheorem}\label{YinSigma}
Let $d \ne 0$ be a real number, $s_d\in\N_{\Sho}$ and $f
\in \Cstar$. Then there exists a map $F \in \Li$ such that
$\RotR(F)=\Rot(F)$ is the closed interval with endpoints $0$ and $d$
(i.e., $[c,d]$ or $[d,c]$),
$\Per(0,F) = \TPer(f)$ and
$\Per(F) = \TPer(f) \cup M(0,d) \cup \Lambda(d,\Shs(s_d))$.
\end{MainTheorem}

A natural strategy to prove the second statement of
Conjecture~\ref{WishList} in the general case (i.e. when no endpoint
of the rotation interval is an integer) is to construct examples of
maps $F\in\Li$ with a \emph{block structure} over maps $f\in \Cstar$
in such a way that $p/q$ is an endpoint of the rotation interval
$\RotR(F)$ and $\Per(p/q,F) = q\cdot\TPer(f)$.
The next result shows that this is not possible.
Hence, if the second statement of Conjecture~\ref{WishList} holds,
the examples must be built by using some more complicated behavior
of the points of the orbit in $\R$ and on the branches than a block
structure.

Let $F \in \Li$ and let $P$ be a lifted periodic orbit of $F$ with
period $nq$ and rotation number $p/q$.
For every $x \in P$ and $i=0,1,\dots,q-1$, we set
\[
P_i(x):=\{ F^i(x), G(F^i(x)), G^2(F^i(x)), \dots, G^{n-1}(F^i(x)) \},
\]
where $G := F^q -p$.
By Lemma~\ref{blocksareperiodic}, every $P_i(x)$
is a (true) periodic orbit of $G$ of period $n$.

\begin{MainTheorem}\label{ConverseEndInteger}
Let $F \in \Li$ and let $P$ be a lifted periodic orbit of $F$ with
period $nq$ and rotation number $p/q$.
Assume that there exists $x \in P$ such that $\chull{P_0(x)}$ is
homeomorphic to a 3-star and
$\chull{P_1(x)} \subset [n,n+1] \subset \R$ for some $n \in \Z$.
Assume also that
$P_0(x)$ is a periodic orbit of type 3 of $G:=F^q - p,$
$F^i(m) \in \chull{P_i(x)}$ for $i=0,1,\dots,q-1$
and $G(m) = m$, where $m \in \Z \cap \chull{P_0(x)}$
denotes the branching point of $\chull{P_0(x)}$.
Then $\Per(p/q,F) = q\cdot\N$.
\end{MainTheorem}

Next we study the first statement of  Conjecture~\ref{WishList}.
It turns out that there are two completely different types of
lifted orbits according to the way that they force the existence
of other periods.
Namely, the lifted periodic orbits contained in $B$ (viewed at $\sigma$
level, this means that these periodic orbits do not intersect the
circuit of $\sigma$) or the ``rotational
orbits'' that visit the ground $\R$ of our space $S$.
We start by studying the periods forced by the lifted periodic orbits
contained in $B$.
We also consider the special case of large orbits (i.e., orbits of
large diameter) and show that any
orbit of this kind implies periodic {\modi} points of all periods.
To do this, we have to introduce some notation.

\begin{definition}
Let $F \in \Lifts$ and let $P$ be a lifted periodic orbit of $F$. We
say that $P$ \emph{lives in the branches} when $P \subset B$. Observe
that, since $P$ is a lifted orbit, for every $m \in \IZ$, $B_m \cap P
= (B_0 \cap P) + m$.
\end{definition}

The following result holds for any degree. It extends
\cite[Proposition~5.1]{LLl} (which deals
with $\sigma$ maps fixing the branching point of $\sigma$) to all
$\sigma$ maps.

\begin{MainTheorem}\label{TheoremSharkovskiiintheBranches}
Let $F \in \Lifts$ and let $P$ be a lifted periodic orbit of $F$ of
period $p$ that lives in the branches.
Then $\Per(F) \supset \Shs(p)$.
Moreover, for every $d \in \IZ$ and every $p \in \N_{\Sho}$,
there exists a map $F_p \in \Li[d]$ such that $\Per(F_p) = \Shs(p)$.
\end{MainTheorem}

\begin{definition}
Let $F \in \Lifts$ and let $Q$ be a (true) periodic orbit of $F$.
We say that $Q$ is a \emph{large orbit} if $\diam(\Re(Q)) \ge 1$,
where $\diam(\cdot)$ denotes the diameter of a set.
\end{definition}

If $F \in \Lifts$ and if $Q$ is a true periodic orbit of
$F$, then $Q + \IZ$ is a lifted periodic orbit of $F$ of period
$\Card(Q)$. Clearly, $Q \subset B$ if and only if $Q + \IZ \subset B$.
Therefore we shall also say that $Q$ \emph{lives in the branches} whenever $Q
\subset B$.
Moreover, when $F$ is of degree $1$, true periodic orbits correspond
to lifted periodic orbits of rotation number $0$.
Observe that a periodic orbit $Q$ living in the branches is large if and only
if $Q$ intersects two different branches.

In the case of large orbits living in the branches and degree $1$ maps,
we obtain the next result, much stronger than
Theorem~\ref{TheoremSharkovskiiintheBranches}

\begin{MainTheorem}\label{LargeOrbitsintheBranches}
Let $F \in \Li$ and let $Q$ be a large orbit of $F$ such that $Q$
lives in the branches. Then $\Per(F) = \IN$.
\end{MainTheorem}

\begin{remark}
Large orbits contained in $\IR$ work as in the circle case by using
$\Re \circ F$. More precisely, if $F\in\Li$ has a large orbit contained in
$\R$, then so does the map $\Re \circ F$.
Thus, by \cite[Theorem~2.2]{AlsRue2010},
there exists $n \in \N$ such that
\[
 \left[-\tfrac{1}{n}, \tfrac{1}{n} \right] \subset
      \Rot(\Re \circ F).
\]
In the proof of \cite[Theorem~4.17]{AlsRue2008}, it is shown that,
if $0\in\Int{\Rot(\Re \circ F)}$, then $F$ has a positive horseshoe
and $\Per(0,F)=\IN$.
Consequently, $\Per(F) \supset \Per(0, F) = \N$.
\end{remark}

The set of periods of maps from $\Li$ having a large orbit that
intersects both $\R$ and the branches remain unknown.
Example~\ref{ex:0inintRot-1} shows that the existence of a large orbit
does not ensure that $\Per(F)=\IN$.

Next we study the orbits forced by the existence of lifted periodic
orbits that intersect $\R.$ We obtain the following theorem, which is
the main result of this paper.

\begin{MainTheorem}\label{theo:0inInterior}
Let $F\in\Li$. If $\Int(\RotR(F)) \cap \Z \ne \emptyset$, then
$\Per(F)$ is equal to,
either $\IN$, or $\IN\setminus\{1\}$, or $\IN\setminus\{2\}$.
Moreover, there exist maps $F_0, F_1, F_2 \in \Li$
with $0 \in \Int(\RotR(F_i))$ for $i=0,1,2$ such that
$\Per(F_0) = \IN$,
$\Per(F_1) = \IN\setminus\{1\}$ and
$\Per(F_2) = \IN\setminus\{2\}$.
\end{MainTheorem}

The paper is organized as follows.
In Section~\ref{sec:covering}, we state some relations about periodic
points of different liftings, we recall the notions of covering and
positive covering and give some of their properties, which are key tools
for finding periodic points.
In Section~\ref{sec:Y}, we prove
Corollary~\ref{ConverseSigmaCircleCase} and
Theorems~\ref{YinSigma} and \ref{ConverseEndInteger}.
In Section~\ref{WeAreInTheBranches}, we prove
Theorems~\ref{TheoremSharkovskiiintheBranches} and
\ref{LargeOrbitsintheBranches}.
Section~\ref{sec:0inIntRotR}, devoted to
Theorem~\ref{theo:0inInterior}, starts with the construction of examples,
then states some more technical lemmas about the set of periods and finally
gives the proof of Theorem~\ref{theo:0inInterior}.
In the last section, we stress some difficulties in the characterization
of the set of periods: a first example shows that, in
Theorem~\ref{theo:0inInterior}, one cannot replace $\Per(F)$ by $\Per(0,F)$
(i.e., periods {\modi} by true periods), which is an obstacle to apply
to $\sigma$ maps the same method as for circle maps; two other examples
show that orderings $\leso{n}$ with $n>3$ may be needed to characterize
$\Per(0,F)$, which might let us think that, in the first statement of
Conjecture~\ref{WishList}, considering orderings $\leso{2}$ and
$\leso{3}$ may not be sufficient.

\section{Coverings and periodic points}\label{sec:covering}
\subsection{Relations between periodic points of \emph{F} and of \emph{F+k}}

Next easy lemma summarizes some basic properties of liftings; in particular,
periodic {\modi} points do not depend on the choice of the lifting of a
given $\sigma$-map.

\begin{lemma}\label{lem:FF+k}
Let $F\in\Li[d]$. The following statements hold for all $k,m\in\IZ$ and
all $n\ge 0$:
\begin{enumerate}
\item $F^n(x+m)=F^n(x)+md^n$; in particular, if $d=1$ then $F^n(x+m)=F^n(x)+m$,
\item $(F+k)^n(x)=F^n(x)+k(1+d+\cdots+d^{n-1})$; in particular, if $d=1$ then
$(F+k)^n(x)=F^n(x)+kn$ and $\rhos[F+k](x)=\rhos(x)+k$,
\item If $F'\in\Li[d']$, then $F'\circ F\in\Li[dd']$,
\item A point $x$ is periodic {\modi} of period
$n$ for $F$ if and only if $x+m$ is periodic {\modi} of period
$n$ for $F+k$. This implies in particular that $\Per(F)=\Per(F+k)$,
\item if $d=1$ and $F^n(x)=x+m$, then $\rhos(x)=m/n$; thus all
periodic {\modi} points have rational rotation numbers.
\end{enumerate}
\end{lemma}

\begin{proof}
Statements (a), (b) and (c) are \cite[Lemma~1.6]{AlsRue2008}
(see also \cite[Lemma~1.10(b)]{AlsRue2008}), and
(e) is \cite[Remark~1.14(ii)]{AlsRue2008}.

We set $G:=F+k$. By (a) and (b),
\[
   \forall x\in S,
   \forall i\in \IN,
   G^i(x+m) = F^i(x) + m d^i + k\sum_{j=0}^{i-1} d^j.
\]
Therefore $F^i(x)- x\in\IZ$
if and only if $G^i(x+m)- (x+m)\in\IZ$, which proves (d).
\end{proof}

The next lemma is implicitly contained in
\cite[Theorem~3.11]{AlsRue2008}. It is a tool to relate the
periods and rotation numbers of lifted periodic orbits with the
periods of true orbits of appropriate powers of the map.

\begin{lemma}\label{relationF_Fqmp}
Let $F\in\Li$, $p \in \Z$ and $q \in \N$ be such that $p,q$ are
relatively prime.
Then $x$ is a periodic {\modi} point of $F$ of period $mq$ and
rotation number $p/q$ if and only if $x$ is a (true) periodic
point of $F^q -p$ of period $m$.
\end{lemma}

\begin{proof}
Set $G:= F^q-p$.
Assume first that $x$ is a period {\modi} point of $F$ of period $mq$
and rotation number $p/q$.
From the definition of periodic {\modi} point, we have
$F^{mq}(x) = x+k$ for some $k\in \Z$.
Then $p/q = \rhos(x) = k/(mq)$ by Lemma~\ref{lem:FF+k}(e). Hence $k= mp$.

By Lemma~\ref{lem:FF+k}(b), $G^j(x) = F^{qj}(x) -jp$ for every $j \ge 0$.
Consequently, $G^m(x) = F^{qm}(x) - mp = x + k - mp = x$
and $x$ is a true periodic point of $G$ of period a divisor of $m$.
Now we have to prove that
$G^j(x) \ne x$ for $j=1,2,\dots,m-1$.
Assume on the contrary that
$G^d(x) = x$ for some $d\in \{1,2,\dots,m-1\}$.
From above, we have $x = G^d(x) = F^{qd}(x) - dp$.
Hence $F^{qd}(x) - x \in \Z;$ a contradiction with the fact that $x$
is a periodic {\modi} point of $F$ of period $mq$. We deduce that $x$ is of
period $m$ for $G$.

Assume now that $x$ is a (true) periodic point of $G$ of period $m$.
From above, $x = G^m(x) = F^{qm}(x) - mp$.
Thus, $F^{qm}(x) = x + mp$, $\rhos(x)=\tfrac{p}{q}$
and the period {\modi} of $x$ for $F$ is an integer $d$ that divides $qm$.
Let $l\in\IN$ and $a\in\IZ$ be such that $d=\tfrac{mq}{l}$
and $F^d(x)=x+a$. To end the proof, we have to show that $d=qm$, that is,
$l=1$. Assume  on the contrary that $l > 1$.
Then, by Lemma~\ref{lem:FF+k}(b),
\[
 x + mp = F^{mq}(x) = F^{ld}(x) = x + la = x + \frac{mq}{d} a.
\]
Consequently, $a = d\tfrac{p}{q} \in \Z$.
Thus $d$ must be a multiple of $q$ because $p,q$ are coprime.
Write $d = bq$. Since  $d=\tfrac{mq}{l}$, we obtain
 $b = \tfrac{m}{l} < m$. But, on the other hand,
$F^d(x) = x + a$ can be written as $F^{bq}(x) = x + bp$,
which is equivalent to
$
x = (F^{bq}-bp)(x) = G^b(x).
$
This contradicts the fact that $x$ is a periodic point of
$G$ of period $m$. We deduce that the period {\modi} of $x$ for $F$ is
$mq$.
\end{proof}

The following technical lemma will be useful to relate true periodic
orbits of maps from $\Lifts$ wit lifted periodic orbits.

\begin{lemma}\label{PeriodsAndPeriodsmodiAreFriends}
Let $F \in \Lifts$, $x \in  S$ and $m,k \in \IZ$. Set
$G := F+k$ and $\widetilde{x} := x + m$.
\begin{enumerate}
\item If $\widetilde{x}$ is a true periodic point of $G$ of period
$q$, then $x$ is a periodic {\modi} point of $F$ of period~$q$.
In particular, for $k=m=0$, it states that a true periodic point of
$F$ is also a periodic {\modi} point of $F$ of the same period.

\item If $x$ is a periodic {\modi} point of $F$ of period $q$ and
$
\diam(\Orb(\widetilde{x},G)) < 1,
$
then $\widetilde{x}$ is a true periodic point of $G$ of period $q$.
\end{enumerate}
\end{lemma}

\begin{proof}
Let $d$ denote the degree of $F$.
Suppose that $\widetilde x$ is a periodic point of $G$ of period $q$. Then
$\widetilde x$ is periodic {\modi} of period $p$ for $G$ with $p$ a divisor
of $q$. Let $n\in\IZ$ and $a\in\IN$ be such that
$G^p(\widetilde x)=\widetilde x+n$ and $q=ap$.
According to Lemma~\ref{lem:FF+k}(a,c), the map $G^p$ is of degree
$d^p$ and
\[
   G^q(\widetilde x)=
   G^{ap}(\widetilde x)=
   \widetilde{x} + n\sum_{i=0}^{a-1} d^{pi}.
\]
This equality is possible only if $n=0$. Thus $G^p(\widetilde x)=\widetilde
x$, which implies that $p=q$. Then (a) follows from Lemma~\ref{lem:FF+k}(d).

Let $x$ be a periodic {\modi} point of $F$ of period $q$. Then
$\widetilde x=x+m$ is periodic {\modi} of period $q$ for $G$ by
Lemma~\ref{lem:FF+k}(d). If
$\diam(\Orb(\widetilde{x},G)) < 1$, the fact that
$G^n(\widetilde{x}) - \widetilde{x} \in \IZ$ is equivalent to
$G^n(\widetilde{x}) = \widetilde{x}$. This implies that $\widetilde x$
is actually a true periodic point of period $q$ for $G$.
\end{proof}

\subsection{Coverings and periods}

\begin{definition}
Let $F \in \Lifts$ and let $I, J$ be compact non-degenerate
subintervals of $S$.
We say that $I$ \emph{$F$-covers} $J$ if there exists a subinterval
$I' \subset I$ such that $F(I') = J$.
If $I_1,\ldots, I_k$ are compact non-degenerate intervals, the
\emph{$F$-graph} of $I_1,\ldots, I_k$ is the directed graph
whose vertices are  $I_1,\ldots, I_k$ and there is an arrow
from $I_i$ to $I_j$ in the graph if and only if $I_i$
$F$-covers $I_j$.
Then we write $I_i \arrowto I_j$
(or $I_i \maparrow I_j$ if the map needs to be specified)
to mean that $I_i$ $F$-covers $I_j$.
A \emph{path of coverings of length $n$} is a sequence
\[
  J_0\maparrow[F_0]J_1\maparrow[F_1]\cdots \maparrow[F_{n-1}]J_n,
\]
where $J_0,\ldots, J_n$ are compact non-degenerate intervals and
{\map{F_i}{J_i}[S]} are continuous maps (generally of the form
$F^{n_i}-p_i$) for all $0\le i\le n-1$.
Such a path is called a \emph{loop} if $J_n=J_0$.
If all the maps $F_i$ are equal to $F$ and $J_0,\ldots, J_n\in
\{I_1,\ldots, I_k\}$, we speak about paths (resp. loops) in the
$F$-graph of $I_1,\ldots, I_k$.

Consider two paths of the form
\begin{gather*}
\CA=J_0\maparrow[F_0]J_1\maparrow[F_1]\cdots \maparrow[F_{n-1}]J_n,\\
\CB= J_n\maparrow[F_n]J_{n+1}\maparrow[F_{n+1}]\cdots
     \maparrow[F_{n+m-1}]J_{n+m}.
\end{gather*}
Then $\CA\CB$ will denote the concatenation of these two paths,
that is,
\[
\CA\CB = J_0\maparrow[F_0]J_1\maparrow[F_1] \cdots
   \maparrow[F_{n-1}]J_n \maparrow[F_n]\cdots
   \maparrow[F_{n+m-1}]J_{n+m}.
\]
If $J_n=J_0$, it is possible to concatenate $\CA$ with itself and, for
every $n\in\IN$, $\CA^n$ will denote the concatenation of $\CA$ with
itself $n$ times.
\end{definition}

When considering an $F$-graph, the intervals
are often defined from a finite collection of points.

\begin{definition}\label{def:basic-int}
Let $P$ be a finite subset of $S$. A
\emph{$P$-basic interval} is any set $\chull{a,b}$, where
$a,b$ are two distinct points in $P$ such that $\chull{a,b}\cap
\chull{P}=\{a,b\}$. Observe that, if $P$ contains all the branching points
$\IZ\cap\chull{P}$, then the $P$-basic intervals are
equal to the closure of the connected components of $\chull{P}\setminus P$.
\end{definition}

\begin{remark}
If $\Int(I)$ and $\Int(J)$ contain no branching point, the fact that
$F(I) \supset J$ implies $I \arrowto J.$
In what follows, we shall only use coverings with intervals
containing no branching point in their interior.
\end{remark}

The next result is the key property for finding periodic points with
coverings.
It is \cite[Lemma~1.2.7]{ALM} generalized to intervals in $S$.

\begin{proposition}\label{prop:covering}
Let $I_0,I_1,\ldots, I_n$ be compact subintervals of $S$ with $I_n=I_0$
and, for every $0\le i\le n-1$, let {\map{F_i}{I_i}[S]} be a
continuous map such that $I_i$ $F_i$-covers $I_{i+1}$. Then there
exist points $x_i\in I_i$, $i=0,\ldots, n$, such that
$F_i(x_i)=x_{i+1}$ for all $0\le i\le n-1$ and $x_n=x_0$.
In particular,
\begin{itemize}
\item if $F_i=F$ for all $0\le i\le n-1$ (that is,
$I_0 \arrowto I_1 \arrowto \cdots \arrowto I_{n-1} \arrowto I_0$
is a loop in the $F$-graph of $I_1,\ldots, I_{n-1}$),
then $F^n(x_0)=x_0$;

\item if $F_i=F+k_i$ with $k_i\in\IZ$ for
all $0\le i\le n-1$, then $F^n(x_0)\in x_0+\IZ$.
\end{itemize}
\end{proposition}

The next lemma shows that, under certain hypotheses (that is, in
presence of ``semi horseshoes''), we have periodic points of all
periods. It is a generalization of \cite[Proposition~1.2.9]{ALM} and
its proof is a variant of the proof of that result. However, we
include it for clarity.

\begin{proposition}\label{prop:SemiHorseshoe}
Let $F \in \Lifts$ and assume that there exist two compact non-degenerate
subintervals $K$ and $L$ of $S$ such that $K$ and $L$ do not contain
branching points in their interior, $\Int(K) \cap \Int(L) = \emptyset$ and
$F(K) \supset L$ and  $F(L) \supset K \cup L$. Then, for every $n \in
\IN,$ the map $F$ has a periodic orbit of period $n$ contained in
$K \cup L$.
\end{proposition}

\begin{proof}
By assumption, $K \arrowto L$ and $L \arrowto K, L$.
Since $K,L$ contain no branching point in their interior,
the set $J := \chull{K \cup L}$ is an interval (which may contain branching
points).
By continuity of $F$, there exist subintervals
$L' \subset L$ and $K' \subset K$ such that
$F(L') \supset J$, $F(\Bd(L')) = \Bd(J)$, $F(K') = L'$ and
$F(\Bd(K')) = \Bd(L')$.
Therefore, for every $n\in\IN$, there is a loop
\[
  \CPath{K'}>{L'}>{L'}>{\dots}>{L'}>{K'}
\]
of length $n$ in the $F$-graph of $K', L'$
(if $n = 1$, the loop we take is $L' \arrowto L'$).
By Proposition~\ref{prop:covering}, $F$ has a periodic point $x \in K'$
such that $F^i(x) \in L'$ for $i = 1,2,\dots,n-1$ and $F^n(x) = x$
(if $n = 1$, $F(x)=x\in L'$).
To prove that $x$ has period $n$, we have to show that
$F^i(x) \neq x$ for all $i = 1,2,\dots,n-1$.

Suppose now that $F^i(x) = x$ for some
$i \in \{1,2,\dots,n-1\}$  (in particular $n > 1$).
Then $x=F^i(x)$ belongs to $K'\cap L'$, and hence
\begin{equation}\label{eq:xinL}
x\in \Bd(L').
\end{equation}
Consequently, $F(x) = F^{i+1}(x) \in \Bd(J)$.
If $i+1 \le n-1$, then $F(x) = F^{i+1}(x)$ also belongs to $L'$ and,
hence, it is the unique point in $\Bd(L') \cap \Bd(J)$ and, again,
$F^2(x) = F^{i+2}(x) \in \Bd(J)$.
Iterating this argument, we see that
$F^l(x) = F^{i+l}(x) \in \Bd(J)$ for all $l=0,1,\dots, n-i$.
Then $x = F^n(x) = F^{n-i}(x) \in K' \cap \Bd(J)$, which implies
that $x$ is the endpoint of $J$ that does not belong to $L'$.
But this contradicts \eqref{eq:xinL}. We conclude that the period of $x$
is equal to $n$.
\end{proof}

The next lemma is similar to the previous one, except that the
coverings are {\modi}.

\begin{lemma}\label{lem:SemiHorseshoe-mod1}
Let $F\in\Li$. Let $I,J$ be two non empty compact intervals in $S$
such that $\Int(I),\Int(J)$ are disjoint and contain
no branching point. Suppose
that there exist $k_1,k_2,k_3\in\IZ$ such that
\[
 I \maparrow[F-k_1] I, \quad
 I \maparrow[F-k_2] J, \quad
 J \maparrow[F-k_3] I.
\]
Suppose in addition that
\begin{itemize}
\item either $I,J$ are disjoint {\modi}
(that is, $(I+\IZ) \cap (J+\IZ)=\emptyset$),

\item or $k_3=k_1$.
\end{itemize}
Then $\Per(F)=\IN$.
\end{lemma}

\begin{proof}
We fix $n\in\IN$.
For $n = 1$,  we consider the loop $I \maparrow[F-k_1] I$,
and there exists a fixed {\modi} point in $I$ by
Proposition~\ref{prop:covering}.
For $n\ge 2$, we consider the loop of length~$n$
\[
  J \maparrow[F-k_3] I \maparrow[F-k_1] I
    \maparrow[F-k_1] \cdots
    \maparrow[F-k_1] I \maparrow[F-k_2] J.
\]
By Proposition~\ref{prop:covering}, $F$ has a periodic {\modi} point $x\in J$
such that $F^n(x)=x+k_3+(n-2)k_1+k_2$ and
$F^i(x)\in I+k_3+(i-1)k_1$ for all $1\le i\le n-1$.
Let $d$ denote the period {\modi} of $x$.

If $I,J$ are disjoint {\modi}, then
$F^i(x)-x\notin \IZ$ for all $1\le i\le n-1$, and thus $d=n$.

Suppose now that $k_3=k_1\neq k_2$. Then
\[
\rhos(x) =
  \frac{k_3+(n-2)k_1+k_2}{n} =
  k_1+\frac{k_2-k_1}{n}.
\]
If $d<n$, then $F^d(x)=x+k_3+(d-1)k_1$ and hence
\[
\rhos(x)=\frac{k_3+(d-1)k_1}{d}=k_1.
\]
But this is impossible because $\frac{k_2-k_1}{n}\neq 0$. We deduce that,
if $k_3=k_1\neq k_2$, then $d=n$.

Finally, if $k_1=k_2=k_3$, then Proposition~\ref{prop:SemiHorseshoe} applies
to the map $G:=F-k_1$ and $\Per(G)=\IN$. Thus $\Per(F)=\IN$
by Lemma~\ref{lem:FF+k}(d). This concludes the proof.
\end{proof}

\subsection{Positive coverings}
The notion of positive covering for subintervals of $\IR$ was
introduced in \cite{AlsRue2008}. It can be extended to all
subintervals on which a retraction can be defined.
This is in particular the case of all intervals which have an
infinite tree as the ambient space.

If $I \subset S$ is an interval, it can be endowed with two
opposite linear orders; we denote them by $<_{_I}$ and $>_{_I}$.
When $I\subset \IR$, we choose $<_{_I}$ so that it
coincide with the order $<$ in $\IR$; when $I\subset B$, we choose
$<_{_I}$ so that $x<_{_I} y\Leftrightarrow \Im(x)<\Im(y)$. In the other
cases, $<_{_I}$ is chosen arbitrarily. The notations $\le_{_I}$ and $\ge_{_I}$
are defined consistently.

\begin{definition}\label{def:positivecover}
Let $F \in \Lifts$ and let $I,J$ be compact non-degenerate subintervals
of $S$, endowed with orders $<_{_I}, <_{_J}$.
We say that $(I, <_{_I})$ \emph{positively} (resp. \emph{negatively})
\emph{$F$-covers $(J,<_{_J})$} and we write
$(I, <_{_I}) \signedcover{F} (J,<_{_J})$
(resp. $(I, <_{_I}) \signedcover[-]{F} (J,<_{_J})$)
if there exist $x,y\in I$ such that $x \le_{_I} y$,
$F(x)=\min J$ and $F(y)=\max J$
(resp. $F(x)=\max J$ and $F(y)=\min J$).
When there is no ambiguity on the orders (or no need to precise them),
we simply write
$I \signedcover{F} J$ or $I \signedcover[-]{F} J$.
\end{definition}

We remark that the notion of positive or negative covering does not
imply (unlike the usual notion of $F$-covering) that there exists a
closed subinterval of $I' \subset I$ such that $F(I') = J$. However, it
does for the retracted map.

We recall that the retraction {\map{\ret_{_I}}{S}[I]} is defined as
follows:
\[
 \ret_{_I}(x) = \begin{cases}
           x   & \text{if $x \in I$}\\
           c_x & \text{if $x \notin I$,}
          \end{cases}
\]
where $c_x$ is the only point in $I$ such that
$\chull{c_x, x} \cap I = \{c_x\}$ (it exists since $S$ is uniquely
arcwise connected).

\begin{remark}
$(I, <_{_I})$ positively (resp. negatively) $F$-covers $(J,<_{_J})$
if and only if there exist $x,y\in I$, $x \le_{_I} y$, such that
$\ret_{_J} \circ F(x)=\min J$ and $\ret_{_J} \circ F(y)=\max J$
(resp. $\ret_{_J} \circ F(x)=\max J$ and $\ret_{_J} \circ F(y)=\min J$).
Moreover, if $I$ positively or negatively $F$-covers $J$, then there
exists a closed subinterval $I' \subset I$ such that
$\ret_{_J}(F(I')) = J$ and $F(\Bd(I'))=\Bd(J)$.
\end{remark}

If $\eps,\eps'\in\{+,-\}$, the product $\eps\eps'\in\{+,-\}$ denotes
the usual product of signs, and $-\eps$ denotes the opposite sign.

\begin{definition}
A \emph{loop of signed coverings of length $k$} is a sequence
\[
(I_0, <_{_0}) \signedcover[\eps_1]{F_1}
(I_1, <_{_1}) \signedcover[\eps_2]{F_2} \cdots
(I_{k-1}, <_{_{k-1}}) \signedcover[\eps_k]{F_k} (I_0, <_{_0}),
\]
where $(I_0,<_{_0}),(I_1, <_{_1}),\dots,(I_{k-1}, <_{_{k-1}})$ are
compact non-degenerate intervals of $S$ endowed with an order,
$\eps_i\in \{+,-\}$ and {\map{F_i}{I_i}[S]} are continuous maps
(usually of the form $F^{n_i}-p_i$) for all $1\le i\le k$.
The  \emph{sign} of the loop is defined to be the product
$\eps_1\eps_2\cdots\eps_k$.
The loop is said \emph{positive} (resp. \emph{negative}) depending on
its sign. We shall use the same notations for concatenations of paths
of signed coverings as for coverings. It is clear that the sign of the
concatenation is the product of the signs of the paths involved.
\end{definition}

The next lemma studies the dependence of the sign of a loop of
signed coverings on the chosen orderings.

\begin{lemma}\label{lem:make-coverings-positive}
Let
\[
(I_0, <_{_0}) \signedcover[\eps_1]{F_1}
(I_1, <_{_1})\signedcover[\eps_2]{F_2} \cdots
(I_{k-1}, <_{_{k-1}})\signedcover[\eps_k]{F_k}(I_0, <_{_0}),
\]
be a loop of signed coverings of sign $\eps$.
\begin{enumerate}
\item For every $0\le i\le k-1$, let
$\widetilde{<_{_i}}\in\{<_{i},>_{i}\}.$
Then, there exist
$\eps_1',\ldots,\eps_k'\in \{+,-\}$ such that
\[
(I_0, \widetilde{<_{_0}})\signedcover[\eps_1']{F_1}
(I_1, \widetilde{<_{_1}}) \signedcover[\eps_2']{F_2} \cdots
(I_{k-1}, \widetilde{<_{_{k-1}}})\signedcover[\eps_k']{F_k}
(I_0, \widetilde{<_{_0}}),
\]
and the sign of this loop is equal to $\eps$.
Consequently, the sign of a loop is independent of the orders.

\item For every $1\le i\le k-1$, there exists
$\widetilde{<_{_i}}\in\{<_{_i},>_{_i}\}$ such that
\[
(I_0, <_{_0})\signedcover{F_1}
(I_1, \widetilde{<_{_1}})\signedcover{F_2}
\cdots \signedcover{F_{k-1}}(I_{k-1},
\widetilde{<_{_{k-1}}})\signedcover[\eps]{F_k} (I_0, <_{_0}).
\]
\end{enumerate}
\end{lemma}

\begin{proof}
Consider a sequence of two signed coverings
$
(I,<_{_I})\signedcover[\eps]{F}
(J,<_{_J})\signedcover[\eps']{G}
(K, <_{_K}).
$
If we reverse the order on $J$, it is clear from the definition that
we reverse the signs of both coverings. That is,
\begin{equation}\label{eq:reverse-covering}
(I,<_{_I}) \signedcover[-\eps]{F}
(J,>_{_J}) \signedcover[-\eps']{G}
(K, <_{_K}).
\end{equation}
To prove (a), it is sufficient to show that reversing any
order gives a new loop of signed coverings with the same sign.
If $1 \le i \le k-1$, according to \eqref{eq:reverse-covering},
changing $<_{_i}$ into $>_{_i}$ changes $\eps_{i-1}$ and $\eps_i$ into
$-\eps_{i-1}$ and $-\eps_i$ respectively. Changing $<_{_0}$ into
$>_{_0}$ changes $\eps_1$ and $\eps_k$ into $-\eps_1$ and $-\eps_k$
respectively. In both cases, we obtain a new loop of signed coverings
with the same sign.

To prove (b), we define inductively $\widetilde{<_{_i}}$ for
$i=1,\ldots k-1$.

Let $i \in \{1,\ldots, k-1\}$ and suppose that
$\widetilde{<_{_1}},\dots, \widetilde{<_{i-1}}$
have already been chosen such that
\[
   (I_0, <_{_0}) \signedcover{F_1}
   (I_1, \widetilde{<_{_1}}) \signedcover{F_2}
   \cdots \signedcover{F_{i-1}}
   (I_{i-1}, \widetilde{<_{i-1}}) \signedcover[\eps_i']{F_i}
   (I_i, <_{_i}) \signedcover[\eps_{i+1}']{F_{i+1}}
   \cdots \signedcover[\eps_k']{F_k}
   (I_0, <_{_0}),
\]
for some $\eps_i',\ldots,\eps_k'\in\{+,-\}$.
If $\eps_i'=+$,
let $\widetilde{<_{_i}}$ be equal to $<_{_i}$
and $\eps_{i+1}'':=\eps_{i+1}'$.
Otherwise,
let $\widetilde{<_{_i}}$ be equal to $>_{_i}$
and $\eps_{i+1}'':=-\eps_{i+1}'$.
According to \eqref{eq:reverse-covering}, we obtain
\begin{multline*}
   (I_0, <_{_0}) \signedcover{F_1}
   \cdots \signedcover{F_{i-1}}
   (I_{i-1}, \widetilde{<_{i-1}}) \signedcover{F_i}
   (I_i, \widetilde{<_{_i}}) \signedcover[\eps_{i+1}'']{F_{i+1}}\\
   (I_i, <_{i+1}) \signedcover[\eps_{i+2}']{F_{i+2}}
   \cdots \signedcover[\eps_k']{F_k}
   (I_0, <_{_0}).
\end{multline*}
Then, when all orderings
$\widetilde{<_{_1}},\ldots, \widetilde{<_{_{k-1}}}$
are defined, we obtain
\[
   (I_0, <_{_0}) \signedcover{F_1}
   (I_1, \widetilde{<_{_1}}) \signedcover{F_2}
   \cdots \signedcover{F_{k-1}}
   (I_{k-1}, \widetilde{<_{_{k-1}}}) \signedcover[\eps']{F_k}
   (I_0, <_{_0})
\]
for some $\eps'\in\{+,-\}$. The sign of this loop is $\eps'$, which is
equal to $\eps$ according to~(a).
\end{proof}

The next result is the analogous of Proposition~\ref{prop:covering} for
signed coverings.

\begin{proposition}\label{prop:signedcover}
Let $F \in \Li$ and let
$(I_0,<_{_0}), (I_1,<_{_1}), \dots, (I_{k-1},<_{_{k-1}})$
be compact non degenerate intervals of $S$ endowed with an order such
that
\[
   (I_0,<_{_0}) \signedcover[\eps_1]{F^{n_1}-p_1}
   (I_1,<_{_1}) \signedcover[\eps_2]{F^{n_2}-p_2}
   \cdots \signedcover[\eps_{k-1}]{F^{n_{k-1}}-p_{k-1}}
   (I_{k-1},<_{_{k-1}}) \signedcover[\eps_k]{F^{n_k}-p_k}
   (I_0,<_{_0})
\]
is a positive loop of signed coverings, where $n_i\in\IN$ and $p_i\in\IZ$.
For every $i\in\{1,2,\ldots, k\}$, set $m_i:=\sum_{j=1}^i n_j$ and
$\widehat{p_i}:=\sum_{j=1}^i p_j$.
Then there exists $x_0 \in I_0$ such that
$F^{m_k}(x_0) = x_0 + \widehat{p_k}$ and
$F^{m_i}(x_0) \in I_i + \widehat{p_i}$ for all $1\le i\le k-1$.
\end{proposition}

\begin{proof}
According to Lemma~\ref{lem:make-coverings-positive},
for every $1\le i\le k-1$, there exists
$\widetilde{<_{_i}} \in \{<_{_i},>_{_i}\}$ such that
\[
   (I_0,<_{_0}) \signedcover{F^{n_1}-p_1}
   (I_1,\widetilde{<_{_1}}) \signedcover{F^{n_2}-p_2}
   \cdots \signedcover{F^{n_{k-1}}-p_{k-1}}
   (I_{k-1},\widetilde{<_{_{k-1}}}) \signedcover{F^{n_k}-p_k}
   (I_0,<_{_0}).
\]
Thus we can consider a loop in which all coverings are positive. In this
case, we have the same situation as
\cite[Proposition 2.3]{AlsRue2008} except that
\cite[Proposition 2.3]{AlsRue2008} is stated for subintervals of
$\IR$.
Actually this assumption plays no role (except simplifying the
notations), and the proof in our context works exactly the same by
using the map $F$ composed with appropriate retractions.
\end{proof}

The next result is analogous to Lemma~\ref{lem:SemiHorseshoe-mod1}
(indeed to a particular case of Lemma~\ref{lem:SemiHorseshoe-mod1})
with the semi horseshoe being made of positive coverings.

\begin{corollary}\label{cory:+horseshoeFF-1}
Let $F\in \Li$ and let $I \subset S$ be a compact interval such that
$(I+n)_{n\in\IZ}$ are pairwise disjoint.
If $I \signedcover{F} I$ and  $I \signedcover{F} I+k$
for some $k \in \IZ\setminus\{0\}$,
then $\Per(F) = \IN$.
\end{corollary}

\begin{proof}
Fix $n\in\IN$.
We consider the following loop of positive coverings of length $n$:
\[
I \signedcover{F} I \signedcover{F} I \cdots \signedcover{F} I
\signedcover{F} I+k.
\]
By Proposition~\ref{prop:signedcover}, there exists a point $x\in I$
such that $F^n(x) = x+k$ and $F^i(x) \in I$ for all $1 \le i \le n-1$.
In particular, $\rhos(x) = k/n \neq 0$.
Suppose that $F^i(x) \in x+\IZ$ for some $i \in \{1,2,\dots,n-1\}$.
Both $x$ and $F^i(x)$ belong to $I$, and thus $F^i(x) = x$
because $(I+n)_{n\in\IZ}$ are pairwise disjoint. But this implies that
$\rhos(x) = 0$, which is a contradiction.
Therefore the period $\modi$ of $x$ is equal to $n$.
Finally, $\Per(F)=\IN$.
\end{proof}

The next lemma is a technical result in the spirit of the previous
one. It shows that, when certain signed loops are available,
the set of periods contains $\IN\setminus\{2\}.$

\begin{lemma}\label{lem:+-loop}
Let $F \in \Li$. Let $K, L\subset S$ be two compact intervals in $S$
and let $e\in S$ be such that
$(K+\IZ) \cap (L+\IZ) \subset \{e\}+\IZ$ and
$F(e) \notin L+\IZ$.
Suppose that there exist $k_1,k_2,k_3,k_4\in\IZ$ such that
\[
L \signedcover{F-k_1} L,\quad L \signedcover{F-k_2} K,\quad
K \signedcover[-]{F-k_3} L,\quad K \signedcover[-]{F-k_4} K.
\]
Then $\Per(F)\supset \IN\setminus\{2\}.$
\end{lemma}

\begin{proof}
According to Proposition~\ref{prop:signedcover} applied to
the loop $L\signedcover{F-k_1} L$,
there exists a fixed point {\modi} of
$F$ in $L$. Hence $1 \in \Per(F)$.

We now fix $n \ge 3$ and  we consider the following
positive loop of length $n$:
\[
  (L \signedcover{F-k_2} K \signedcover[-]{F-k_4} K \signedcover[-]{F-k_3} L)
  (L \signedcover{F-k_1} L)^{n-3}.
\]
By Proposition~\ref{prop:signedcover},
there exists a point $x \in L$ such that
$F(x) \in K+\IZ$,
$F^2(x) \in K + \IZ$,
$F^i(x) \in L + \IZ$ for all $3\le i\le n$ and
$F^n(x) -x\in\IZ$.
Thus $x$ is a periodic {\modi} point for $F$ and its period $p$ divides $n$.
It remains to prove that the period {\modi} of $x$ is exactly $n$.
Suppose on the contrary that $p < n$.
Then $1 \le p \le n-2$ because $p$ divides $n \ge 3$.
Thus $F^2(x)\in K+\IZ$, $F^{2+p}(x)\in L+\IZ$ and $F^{2+p}(x)-F^2(x)\in\IZ$.
By assumption, this is possible only if  $F^2(x)\in e+ \IZ$.
This leads to a contradiction because $F^3(x)\in L+\IZ$ whereas
$F(e)\notin L+\IZ$. This proves that $p = n$, and hence, $\forall n\ge 3$,
$n\in \Per(F)$.
Consequently, $\Per(F)\supset \IN\setminus\{2\}.$
\end{proof}

\section{Sets of periods of 3-star and degree 1 circle maps occur for
degree 1 sigma maps}\label{sec:Y}

Misiurewicz's Theorem~\ref{S9.5} gives a characterization of
the sets of periods of circle maps of degree $1$. It is very easy to build a
map in $\Li$ whose set of periods {\modi} is equal to the set of
periods of a given degree $1$ circle maps. This leads to the following result
(see Section~\ref{sec:statements} for the notations).

\setcounter{MainTheorem}{1}
\begin{MainCorollary}
Given $c,d \in \R$ with $c \le d$
and $s_c, s_d\in\N_{\Sho}$, there exists a map $F \in \Li$ such that
$\RotR(F) = \Rot(F)=[c,d]$ and
$\Per(F) = \Lambda(c,\Shs(s_c)) \cup M(c,d) \cup
\Lambda(d,\Shs(s_d))$.
\end{MainCorollary}

\begin{proof}
By Theorem~\ref{S9.5}, there exists a map $\widetilde{F} \in \LL_1(\IR)$
such that
$\Rot(\widetilde{F}) = [c,d]$ and
$\Per(\widetilde{F}) =
   \Lambda(c,\Shs(s_c)) \cup M(c,d) \cup \Lambda(d,\Shs(s_d))$.
Then we define $F \in \Li$ by
\[
F(x) = \begin{cases}
        \widetilde{F}(x) & \text{if $x \in \R$},\\
        \widetilde{F}(m) & \text{if $x \in B_m$}.
       \end{cases}
\]
Clearly, $F$ is continuous, $\Rot(F)=\RotR(F) = \Rot(\widetilde{F})$ and
every periodic {\modi} point of $F$ is contained in $\R$. Hence,
$\Per(F) = \Per(\widetilde{F})$.
This ends the proof of the corollary.
\end{proof}

It is also easy to build a  map in $\Li$ whose set of periods is
equal to the set of periods of a given $3$-star map.
This construction can be done in such a way that the rotation interval is
any interval of the form $[0,d]$ or $[d,0]$. The set of periods {\modi} is then
a combination of a set of periods of a $3$-star map and a set of periods of
a degree $1$ circle map, as stated in the next result.

\begin{MainTheorem}
Let $d \ne 0$ be a real number, $s_d\in\N_{\Sho}$ and $f \in \Cstar$.
Then there exists a map $F \in \Li$ such that $\RotR(F)=\Rot(F)$ is
the closed interval with endpoints $0$ and $d$ (i.e., $[c,d]$ or
$[d,c]$), $\Per(0,F) = \TPer(f)$ and $\Per(F) = \TPer(f) \cup M(0,d)
\cup \Lambda(d,\Shs(s_d))$.
\end{MainTheorem}

\begin{proof}
We shall only consider the case $d > 0$. The case $d$ negative is
analogous.

From Theorem~\ref{S9.5}, it follows that there
exists a map $G \in \LL_1(\IR)$ such that
$\Rot(G) = [0,d]$ and
$\Per(G) = \{1\} \cup M(c,d) \cup \Lambda(d,\Shs(s_d))$.
Moreover, from the proof of Theorem~\ref{S9.5}  (see
\cite[Theorem~3.10.1]{ALM}), the map $G$ is constructed in such a way that
$G(0) = 0$,
there exist $u \le 1/2 \le v$ such that $G\evalat{[0,u]}$ and
$G\evalat{[v,1]}$ are affine and
$\rhos[G](x) = d$ for every $x \in [u,v]$, and
$\rhos[G](x) \ne 0$ for every $x \in \R \setminus \bigcup_{n\ge 0}G^{-n}(\Z)$.

To prove the theorem, we shall construct a map $F \in \Li$ such that
$\RotR(F) = \Rot(F)=\Rot(G) = [0,d],$
$\Per(0,F) = \TPer(f)$ and
$\Per(F) = \Per(0,F) \cup \Per(G)$.

Let $0 < b < a < 1/2$.
For every $m \in \Z$, let
$Y^a_m$ (resp. $Y^b_m$) denote the set $[m-a, m+a] \cup B_m$
(resp. $[m-b, m+b] \cup B_m$).
Observe that
$Y^a_m \cap Y^a_j = \emptyset$ whenever $m \ne j$,
$Y^b_m \subset Y^a_m$, and the set
$Y^a_m \setminus Y^b_m$ has two connected components:
$[m-a,m-b)$ and $(m+b,m+a]$.
Moreover, the sets $Y^a_m$ and $Y^b_m$ are homeomorphic to $X_3.$
Let $\beta_0$ denote a homeomorphism from $Y^b_0$ to $X_3$.

Set $Z:=\bigcup_{i=0}^{\infty} G^{-i}(\Z)$.
Since $G(m) = m$ for every $m \in \Z$, both sets $Z$ and $\R \setminus Z$ are
$G$-invariant and $\IZ\subset Z$.
Moreover, $\rhos[G](x) = 0$ for all $x\in Z$.
Thus, $Z \cap \left(G([u,v]) + \Z\right) = \emptyset$
and
\begin{equation}\label{eq:tt}
     Z \subset ([0,u) \cup (v,1]) + \Z
\end{equation}
because $d \ne 0.$
Since $G|_{[0,u]}$ and $G|_{[v,1]}$ are affine, this implies
that every point in $Z$ has finitely many preimages and,
hence, $Z$ is countable. Moreover, since $G$ has
degree one (Lemma~\ref{lem:FF+k}(a)), $Z + \Z = Z$.
Therefore, there exists a continuous map
$\map{\varphi}{S}[\R]$ such that
$\varphi(x+1) = \varphi(x) +1$ for all $x\in S$,
$\varphi^{-1}(m)=Y_m^a$ for every $m \in \Z,$
$\varphi\evalat{\R}$ is non-decreasing,
$\varphi^{-1}(x)$ is a point for every $x \notin Z$ and
$\varphi^{-1}(x)$ is a non-degenerate interval for every $x \in Z\setminus\Z$.
The idea is similar to Denjoy's construction: under the action of
$\varphi^{-1}$, every integer $m$ is blown up into the 3-star $Y^a_m$, then
the preimages of $m$ under $G$ are blown up too, in order to be able to
define a map $\map{F}{S}$ which is a semiconjugacy of $G$.

Now we define our map $F$ as follows:
\mathtitlecase{F\evalat{Y^a_m}}
we set
$F\evalat{Y^b_0} = \beta^{-1}_0 \circ f \circ \beta_0$,
$F(a) = a,$ $F(-a) = -a$ and we define
$F\evalat{[-a,-b]}$ and $F\evalat{[b,a]}$ affinely in such a way that
$F\evalat{Y^a_0}$ is continuous.
Then, for every $m\in\IZ$ and $x\in Y^a_m,$
we set $F(x):=F(x-m)+m.$
In particular, $F(Y^a_m) \subset Y^a_m$ for every $m \in \Z.$
\mathtitlecase{F\evalat{\varphi^{-1}(Z \setminus G^{-1}(\Z))}}
For every $y \in Z \setminus G^{-1}(\Z)$, the sets
$\varphi^{-1}(y)$ and $\varphi^{-1}(G(y))$ are intervals because
$y$ and $G(y)$ belong to $Z \setminus \Z$.
Moreover, by \eqref{eq:tt}, the map $G$ is, either increasing, or
decreasing at $y$.
We define $F\evalat{\varphi^{-1}(y)}$ to be the unique affine map
from $\varphi^{-1}(y)$ onto $\varphi^{-1}(G(y))$
which is increasing (respectively decreasing) when
$G$ is increasing (respectively decreasing) at $y$.
In particular $F(\Bd(\varphi^{-1}(y))) = \Bd(\varphi^{-1}(G(y))).$
\mathtitlecase{F\evalat{\varphi^{-1}(G^{-1}(\Z)\setminus \Z)}}
For every $y \in G^{-1}(\Z)\setminus \Z$, it follows that
$y\in Z\setminus \Z$ and $G(y) \in \Z$.
We define $F\evalat{\varphi^{-1}(y)}$ to be the unique
affine map from $\varphi^{-1}(y)$ onto $[G(y)-a, G(y)+a]$
which is increasing (respectively decreasing) when
$G$ is increasing (respectively decreasing) at $y$.
In this case we have
$F(\Bd(\varphi^{-1}(y))) = \{G(y)-a, G(y)+a\}.$
\mathtitlecase{F\evalat{\varphi^{-1}(\IR\setminus Z)}}
For every $y\in \IR\setminus Z$, $G(y) \notin Z$ and
$\varphi^{-1}(y)$ and $\varphi^{-1}(G(y))$
are single points.
We set $F(\varphi^{-1}(y))=\varphi^{-1}(G(y))$.

Observe that, by definition, $F$ is continuous in every connected
component of $\varphi^{-1}(Z)$. To see that $F$ it is globally
continuous, notice that,
for every $y \in Z$, $F(z)$ has one-sided limits as
$z \in \varphi^{-1}(\R \setminus Z)$
tends to the endpoints of $\varphi^{-1}(y)$, and
these limits are equal to the endpoints of $\varphi^{-1}(G(y))$.
Consequently, $F$ is continuous.
Moreover, from the definition of $F$ and the fact that
$\varphi(x+1) = \varphi(x) +1,$ $F$ has degree 1.
Hence, $F \in \Li$. Furthermore, the fact that $F(Y^a_m) \subset Y^a_m$
implies that $\forall m\in\IZ$, $\forall x\in Y^a_m$, $\rhos(x)=\rhos(m)=0$,
and hence $\Rot(F)=\RotR(F)$.

On the other hand, from the definition of $F$, it follows that $F$ is
semiconjugate with $G$ through $\varphi$,
that is, $G \circ \varphi = \varphi \circ F$. Hence,
\begin{equation}\label{FGsemiconj}
G^n \circ \varphi = \varphi \circ F^n
\quad
\text{for every $n \in \N.$}
\end{equation}
From \eqref{FGsemiconj}, it follows that
$\rhos(x) = \rhos[G](\varphi(x))$ for all $x\in S$.
Consequently, $\Rot(F)=\RotR(F) = \Rot(G) = [0,d]$ and
\begin{equation}\label{eq:rhoG0}
\rhos(x) = 0\quad\text{if and only if}\quad \exists\, i\ge 0, m\in\IZ
\text{ such that } F^i(x)\in Y^a_m,
\end{equation}
i.e., $F^i(x-m)\in Y^a_0$.
Thus, $\Per(0,F) = \TPer(F\evalat{Y^a_0})$.

Now we are going to prove that
$\TPer(F\evalat{Y^a_0}) = \TPer(f)$,
which implies $\Per(0,F) = \TPer(f).$
By definition, $\TPer(F\evalat{Y^b_0}) = \TPer(f) \ni 1$ (recall
that a star map always has a fixed point by Theorem~\ref{GMT1}).
So, we only have to prove that all periodic points of $F$ in
$[-a,-b] \cup [b,a]$ are fixed points.
Recall that we defined $F$ so that
$F(a) = a,$ $F(-a) = -a,$ $F(b),F(-b) \in Y^b_0$
and
$F\evalat{[-a,-b]}$ and $F\evalat{[b,a]}$ are affine.
Thus, either $F\evalat{[-a,-b]}$ is the identity map, or it is
expansive; and the same holds for $F\evalat{[b,a]}.$
Hence, the only periodic points of $F$ in $[-a,-b] \cup [b,a]$
are fixed points.

To end the proof of the theorem, we have to show that
$\Per(F) = \Per(0,F) \cup \Per(G)$.
Since $G(0) = 0$ and
$\rhos[G](x) \ne 0$ for every $x \in \R \setminus Z$,
it follows that
$
\Per(G) = \{1\} \cup
   \left(\bigcup_{\alpha \in (0,d]} \Per(\alpha,G) \right).
$
Consequently,
\[
\Per(0,F) \cup \Per(G) = \Per(0,F) \cup
   \left(\bigcup_{\alpha \in (0,d]} \Per(\alpha,G) \right)
\]
because $1 \in \Per(0,F).$
On the other hand, by definition,
$
\Per(F) = \Per(0,F) \cup
    \left(\bigcup_{\alpha \in (0,d]} \Per(\alpha,F)\right).
$
Therefore, we only have to show that
$\Per(\alpha,F) = \Per(\alpha,G)$ for every $\alpha \in (0,d]$.

Fix $\alpha \in (0,d]$ and let $x\in S$ be such that
$\rhos(x) = \alpha$.
Then $\rhos[G](\varphi(x)) = \rhos(x)$ by \eqref{FGsemiconj}.
We are going to prove that $x$ is a
periodic {\modi} point of $F$ of period $n$
if and only if $\varphi(x)$ is a
periodic {\modi} point of $G$ of period $n$.

Assume first that $x$ periodic {\modi} point of period $n$ for $F$, that
is, $F^n(x) = x + k$ for some $k \in \Z$
and
$F^j(x) - x \notin \Z$ for all $j = 1,2,\dots, n-1$.
From \eqref{FGsemiconj}, it follows that
\[
G^n(\varphi(x)) = \varphi(F^n(x)) = \varphi(x+k) = \varphi(x) + k.
\]
Therefore, $\varphi(x)$ is a periodic point {\modi} of $G$ with period, either
$n$, or a divisor of $n$. To see that $x$ has indeed $G$-period
{\modi} $n$, suppose by way of contradiction that
there exists  $j \in \{1,2,\dots,n-1\}$ such that
$G^j(\varphi(x)) = \varphi(x) + l$  for some $l \in \Z$. Then
$\varphi(F^j(x)) = G^j(\varphi(x)) = \varphi(x+l).$
Note that the fact that
$
\rhos[G](\varphi(x+l)) = \rhos[G](\varphi(x)) =
       \alpha \ne 0
$
implies that $\varphi(F^j(x)) = \varphi(x+l) \notin Z$ by \eqref{eq:rhoG0}.
Consequently, since $\varphi^{-1}(y)$ is a point for
every $y \notin Z$, $F^j(x) = x+l;$ a contradiction.

Now assume that $G^n(\varphi(x)) = \varphi(x) + k$ for some $k \in \Z$
and
$G^j(\varphi(x)) - \varphi(x) \notin \Z$ for all $j\in\{ 1,2,\dots, n-1\}$.
From \eqref{FGsemiconj}, it follows that
$
\varphi(F^n(x)) = G^n(\varphi(x)) = \varphi(x+k).
$
As above, $\rhos[G](\varphi(x)) = \alpha \ne 0$ implies that
$\varphi(F^n(x)) = \varphi(x+k) \notin Z$ and thus
$F^n(x) = x+k$. If there exists $j \in \{1,2,\dots,n-1\}$ such that
$F^j(x) \in x + \IZ$, then
$G^j(\varphi(x)) = \varphi(F^j(x)) \in \varphi(x) +\IZ$;
a contradiction. Thus $x$ is periodic {\modi} of period $n$ for $F$.
\end{proof}

\begin{remark}
Theorem~\ref{YinSigma} gives a map with a non-degenerate rotation interval.
It is even easier to obtain a degenerate interval (take $G={\rm Id}$ in
the proof), which shows that, for every $f
\in \Cstar$, there exists a map $F \in \Li$ such that
$\Rot(F)=\RotR(F)=\{0\}$ and $\Per(0,F) = \Per(F)=\TPer(f)$.
\end{remark}

One may wonder if Theorem~\ref{YinSigma} can be generalized in order
to obtain a map $F\in \Li$ such that $\RotR(F)= [c,d]$ and
$\Per(c,F)=q\cdot\TPer(f)$ for any $f\in \Cstar$ and any rational
number $c=p/q$ with $p,q$ relatively prime. As we said in
Subsection~\ref{ss:main-statements}, the natural strategy is to use a
block structure. The next result shows that this strategy fails.

\begin{MainTheorem}
Let $F \in \Li$ and let $P$ be a lifted periodic orbit of $F$ with
period $nq$ and rotation number $p/q$.
Assume that there exists $x \in P$ such that $\chull{P_0(x)}$ is
homeomorphic to a 3-star and
$\chull{P_1(x)} \subset [n,n+1] \subset \R$ for some $n \in \Z$.
Assume also that
$P_0(x)$ is a periodic orbit of type 3 of $G:=F^q - p,$
$F^i(m) \in \chull{P_i(x)}$ for $i=0,1,\dots,q-1$
and $G(m) = m$, where $m \in \Z \cap \chull{P_0(x)}$
denotes the branching point of $\chull{P_0(x)}$.
Then $\Per(p/q,F) = q\cdot\N$.
\end{MainTheorem}

Recall that, when $P$ and $G$ are as in
Theorem~\ref{ConverseEndInteger},
\[
  P_i(x) := \{ F^i(x), G(F^i(x)), G^2(F^i(x)), \dots,
                G^{n-1}(F^i(x)) \}
\]
for every $x \in P$ and $i=0,1,\dots,q-1.$
To simplify the notation, in what follows we shall set
$P_q(x) := P_0(x) + p$.

Before proving Theorem~\ref{ConverseEndInteger}, we are going to
develop the tools needed in its proof.

\begin{lemma}\label{blocksareperiodic}
For all $x\in P$ and all $0\le i\le q-1$,
$P_i(x)$ is a true periodic orbit of $G$ of period~$n$.
In particular, $P_i(x) = \set{G^s(F^i(x))}{s \ge 0}.$
\end{lemma}

\begin{proof}
Since the point $F^i(x)$ belongs to $P$,
it is periodic {\modi} of period $nq$ and rotation number
$p/q$ for $F$. Then the result follows from Lemma~\ref{relationF_Fqmp}.
\end{proof}

\begin{definition}
We say that $P$ has an \emph{increasing block structure} whenever,
for some $x \in P,$
\[
\max \Re(P_i(x)) < \min \Re(P_{i+1}(x))\qquad\forall i\in\{0,1,\dots,q-1\}
\]
(when $i=q-1$ this amounts to $\max \Re(P_{q-1}(x))<\min \Re(P_0(x))+p$).
\end{definition}

By the next lemma, the fact that a lifted periodic orbit has an
increasing block structure is independent on the point $x$ chosen to
build the blocks.
So, the notion of \emph{increasing block structure} is well defined.

\begin{lemma}
For every $z \in P$
there exist $k \in \Z$ and $j \in \{0,1,\dots,q-1\}$
such that $z \in P_j(x) + k,$
$P_i(z) = P_{i+j}(x)+k$ for all $0\le i \le q-1-j$ and
$P_i(z) = P_{i+j-q}(x)+k+p$ for all $q-j \le i \le q.$
\end{lemma}

\begin{proof}
By definition, for every $z \in P$ there exist
$k_1 \in \Z$ and $j_1 \in \N$
such that $z = F^{j_1}(x) + k_1.$
On the other hand, by Lemma~\ref{lem:FF+k}(b),
$
 G^n(x) = F^{nq}(x) - np,
$
for every $x \in S$ and $n \ge 0.$

We can write $j_1 = rq + j$ with $r \ge 0$ and $0 \le j < q.$
Hence, by Lemma~\ref{blocksareperiodic},
\[
z = F^{rq + j}(x) + k_1
  = F^{rq}(F^j(x)) + k_1
  = G^{r}(F^j(x)) + k
\in P_j(x) + k,
\]
where $k = k_1 + rp.$ This proves the first statement of the lemma.

By Lemma~\ref{blocksareperiodic},
$
P_i(z) = \set{G^s(F^i(z))}{s \ge 0}.
$
From above and Lemma~\ref{lem:FF+k}(a),
\[
 G^s(F^i(z)) = G^s(F^i(G^{r}(F^{j}(x)) + k)) = G^{r+s}(F^{i+j}(x)) + k
\]
for every $i,s \in \N.$
Consequently,
$
 P_i(z) = \set{G^{r+s}(F^{i+j}(x))}{s \ge 0} + k.
$
If $0\le i \le q-1-j,$ by Lemma~\ref{blocksareperiodic},
$
P_{i+j}(x) = \set{G^s(F^{i+j}(x))}{s \ge 0}
           = \set{G^{r+s}(F^{i+j}(x))}{s \ge 0},
$
which proves the second statement of the lemma.
In particular, $P_q(z) = P_0(z) + p = P_j(x) + k + p.$

If $q-j \le i < q$ then,
$
G^{r+s}(F^{i+j}(x)) = G^{r+s+1}(F^{i+j-q}(x)) + p
$
with $i+j-q \ge 0.$ Hence, as above, $P_i(z) = P_{i+j-q}(x) + k + p.$
\end{proof}

We are going to show that every lifted periodic orbit with
period $nq$ and rotation number $p/q$ will have an increasing block
structure by changing the lifting and the number $p$, if necessary.
To this end, we want to look at the lifted orbit $P$ under the
action of $\overline{F} := F + \ell$ with $\ell \in \Z$. By
Lemma~\ref{lem:FF+k}(b,d), the $\overline{F}-$rotation number
of $P$ is $\tfrac{p}{q} + \ell = \tfrac{p + q\ell}{q}$ while the
$\overline{F}-$period is still $nq$.
So, by using $\overline{F}$ instead of $F$, we can define
\[
\overline{P}_i(x) := \{
  \overline{F}^i(x),
  \overline{G}(\overline{F}^i(x)),
  \overline{G}^2(\overline{F}^i(x)), \dots,
  \overline{G}^{n-1}(\overline{F}^i(x))
\}
\]
for all $i\in\{0,1,\dots,q-1\}$,
where $\overline{G} :=\overline{F}^q - (p + q\ell)$.
We also set
$\overline{P}_q(x) := \overline{P}_0(x) + (p + q\ell).$

\begin{lemma}\label{sepblocks}
The following statements hold:
\begin{enumerate}
 \item $\overline{G} = G$.
 \item For every $i\in\{0,1,\dots,q\}$,
       $\overline{P}_i(x) = P_i(x) + i\ell$.
 \item Assume that
       $\ell > \max \Re(P_i(x)) - \min \Re(P_{i+1}(x))$
       for all $i\in\{0,1,\dots,q-1\}$. Then, the orbit $\overline{P}$ under
       $\overline{F}$ has an increasing block structure, that is,
       $\max \Re(\overline{P}_i(x)) < \min \Re(\overline{P}_{i+1}(x))$
       for all $i\in\{0,1,\dots,q-1\}$.
\end{enumerate}
\end{lemma}

\begin{proof}
For every $i \ge 0$, we have
\[
\overline{F}^i = (F+\ell)^i = F^i + i\ell
\]
by Lemma~\ref{lem:FF+k}(a-b). Hence,
\[
\overline{G} :=
  \overline{F}^q - (p+q\ell) =
  F^q + q\ell - (p+q\ell) =
  G,
\]
and (a) holds.

For all $i,j\ge 0$, we have
\[
\overline{G}^j(\overline{F}^i(x)) =
G^j(F^i(x) + i\ell) =
G^j(F^i(x)) + i\ell.
\]
This gives (b)  for $i=0,1,\dots,q-1$. The fact that
$\overline{P}_q(x) = P_q(x) + q\ell$ follows from (b) for $i=0$ and from the
definition of these two sets.

Suppose that $\ell$ satisfies the assumption of (c).
From (b) and the choice of $\ell$, we have
\[
\min \Re(\overline{P}_{i+1}(x)) - \max \Re(\overline{P}_i(x)) =
\min \Re(P_{i+1}(x)) - \max \Re(P_i(x)) + \ell > 0
\]
for every $i \in \{0,1,\dots,q-1\}.$ Hence (c) holds.
\end{proof}

\begin{proof}[Proof of Theorem~\ref{ConverseEndInteger}]
It is not difficult to show that,
for every $\ell \in \Z$,
$\Per(p/q,F) = \Per((p+q\ell)/q,F+\ell).$
Consequently, by changing the lifting and the number $p$, if
necessary, we may assume that $P$ has an increasing block structure by
Lemma~\ref{sepblocks}.
Moreover, by replacing the point $x$ by $x-m$, we may also assume that the
branching point of $\chull{P_0(x)}$ is 0 (that is, $m=0$).
To simplify the notation, we shall omit the
dependence from $x$ of the blocks $P_i(x)$ in what follows.

Let $I_1,I_2,I_3$ denote the three $P_0 \cup \{0\}$-basic intervals
in $\chull{P_0}$ that have an endpoint equal to $0$ and let $\CG$ be
the directed graph with vertices $I_1,I_2,I_3$ such that there is an
arrow $I_i \arrowto I_j$ if and only if $\chull{G(\partial
I_i)}\supset I_j$
(notice that arrows in $\CG$ are $G$-coverings and $\CG$ is a
subgraph of the $G$-graph of $\{I_1,I_2,I_3\}$).
Since $P_0$ is a periodic orbit of type 3 of $G$ and $G(0) = 0$,
we can label the intervals $I_1,I_2,I_3$ so that
\begin{equation}\label{eq:Gtype3}
I_1 \maparrow[G] I_2 \maparrow[G]  I_3 \maparrow[G] I_1\quad
    \text{is a loop in }\CG.
\end{equation}
Let $\CI$  be the collection of $P_i\cup\{F^i(0)\}$-basic
intervals for all $0\le i\le q$ (recall that $F^i(0)\in\chull{P_i}$ by
assumption, and thus the elements of $\CI$ are intervals in
$\bigcup_{i=1}^q \chull{P_i}$).
We are going to relate paths in the $F$-graph of $\CI$ with coverings for $G$.
Observe that, if
$
\alpha = J_0 \maparrow J_1 \maparrow \dots\maparrow J_q
$
is a path in the $F$-graph of $\CI$ with $J_0\subset \chull{P_0}$
then,
since the blocks $P_i$ have an increasing block structure,
$J_i$ is a basic interval of $P_i\cup\{F^i(0)\}$ for all
$i\in\{0,1,\dots,q\}$.
Moreover, the fact that $\alpha$ is a path for $F$
implies $J_0\maparrow[G]J_q-p$.
Reciprocally, if $J_0\maparrow[G]J_q$ is an arrow in $\CG$, then
\begin{equation}\label{eq:GFpath}
\exists\, J_1,\ldots J_{q-1}\in \CI,\
J_0 \maparrow J_1 \maparrow \dots\maparrow J_q+p.
\end{equation}
Let us prove \eqref{eq:GFpath}.
We have $F^i(\partial J_0)\subset P_i\cup\{F^i(0)\}$ for all
$1\le i\le q$ because $\partial J_0\subset P_0\cup\{0\}$. Then an
induction on $i=1,\ldots, q$ shows that, for all
$P_i\cup\{F^i(0)\}$-basic intervals $J\subset \chull{F^i(\partial J_0)}$,
there exists a path
\begin{equation}\label{eq:Fpath-induction}
J_0\maparrow J_1^J \maparrow \dots \maparrow J_{i-1}^J\maparrow J
\end{equation}
where $J_j^J$ are $P_j\cup\{F^j(0)\}$-basic intervals for all $1\le j\le i-1$.
The fact that $J_0\maparrow[G]J_q$ is an arrow in $\CG$ means that
$\chull{G(\partial J_0)}\supset J_q$, that is, $\chull{F^q(\partial J_0)}
\supset J_q+p$. Therefore \eqref{eq:GFpath} is given by
\eqref{eq:Fpath-induction} for $i=q$ and $J=J_q+p$.


Combining \eqref{eq:Gtype3} and \eqref{eq:GFpath}, we see
that there exist three pairwise different paths
\begin{align*}
\alpha_1 &= I_1 \maparrow J_1 \maparrow \dots\maparrow I_2 + p\\
\alpha_2 &= I_2 \maparrow J_2 \maparrow \dots\maparrow I_3 + p\\
\alpha_3 &= I_3 \maparrow J_3 \maparrow \dots\maparrow I_1 + p
\end{align*}
in the $F$-graph of $\CI$, of length $q$.

Now we consider two cases:

\begin{case}{1} Two of the intervals $J_i$ coincide. \end{case}
By relabeling, if necessary, we may assume that $J_1 = J_2$. Denote
the interval $J_1 = J_2$ by $L$ and consider the following three
loops:
\begin{align*}
\overline{\alpha}_1 &=
       L \maparrow \dots\maparrow I_2+p \maparrow L + p,\\
\overline{\alpha}_2 &=
       L \maparrow \dots\maparrow I_3+p \maparrow J_3 + p,\\
\overline{\alpha}_3 &=
       J_3 \maparrow \dots\maparrow I_1+p \maparrow L + p\ .
\end{align*}
Then
\[
 G(L)   \supset L \cup J_3\quad \text{and}\quad
 G(J_3) \supset L.
\]
By assumption, $\chull{P_1}$ is included in $[n,n+1]$.
Thus $\Int(L)$ and $\Int(J_3)$ do not contain branching points
since $L\cup J_3\subset \chull{P_1}$.
Then the theorem holds by Proposition~\ref{prop:SemiHorseshoe} and
Lemma~\ref{relationF_Fqmp}.

\begin{case}{2} The intervals $J_i$ are pairwise different. \end{case}
In this case, we have the following loop:
\[
 J_1 \maparrow[G] J_2 \maparrow[G] J_3 \maparrow[G] J_1.
\]
By assumption, $\chull{P_1}$ is an interval in $\IR$.
Moreover, $J_1,J_2,J_3$ are included in  $\chull{P_1}$ and have pairwise
disjoint interiors. Thus,
by relabeling if necessary, we can assume that the intervals
$J_1, J_2,J_3$ are ordered as:
\begin{gather*}
\text{either }J_1\le J_2\le J_3,\\
\text{or }
J_1\ge J_2\ge J_3,
\end{gather*}
with the convention that $J_i\le J_j$ if $\max J_i\le \min J_j$.
Both cases being similar, we assume that we are in the first one,
that is,
\[
\max J_1 \le \min J_2 < \max J_2 \le \min J_3.
\]
Then,
\begin{itemize}
 \item since $J_1 \maparrow[G] J_2,$
       there exists $x_1 \in J_1$ such that $G(x_1) = \min J_2;$
 \item since $J_2 \maparrow[G] J_3,$
       there exists $x_2 \in J_2$ such that $G(x_2) = \max J_3$ and
 \item since $J_3 \maparrow[G] J_1,$
       there exists $x_3 \in J_3$ such that $G(x_3) = \min J_1.$
\end{itemize}
Now we set $K = [x_1,x_2]$ and $L= [x_2, x_3]$.
By continuity of $G$,
\begin{align*}
 G(K) &\supset [\min J_2, \max J_3] \supset [x_2, x_3] = L,
       \ \text{and}\\
 G(L) &\supset [\min J_1, \max J_3] \supset [x_1, x_3] = K \cup L,
\end{align*}
and the theorem holds by Proposition~\ref{prop:SemiHorseshoe} and
Lemma~\ref{relationF_Fqmp}, as in Case~1.
\end{proof}

\section{Orbits in the branches}\label{WeAreInTheBranches}
The aim of this section is to prove
Theorems~\ref{TheoremSharkovskiiintheBranches} and
\ref{LargeOrbitsintheBranches}, which
deal with the periods forced by the lifted periodic orbits
contained in $B$.

\subsection{Situations that imply periodic points of all
periods}

This subsection is devoted to two technical lemmas that
characterize simple situations where $\Per(F) = \IN$
in terms of images of distinguished points.
They will also be used in Section~\ref{sec:0inIntRotR}.

Given $F \in \Lifts$ and $x \in S$ we define the map $F_0$ by
\begin{equation}\label{eq:F0}
 F_0(x) := F(x) - \Re(F(x)).
\end{equation}
To understand the map $F_0$, observe first that $F_0(x) = 0$ whenever
$F(x) \in \IR$.
Moreover, for every $x \in S$ it follows that
$F(x) \in B$ if and only if $\Re(F(x)) \in \IZ$
(more precisely, $F(x) \in B_m$ if and only if $\Re(F(x)) = m$).
Thus, $F_0$ is a continuous map from the whole $S$ to $B_0$.
From Lemma~\ref{lem:FF+k}(a), we deduce that
$F_0(x+k) = F_0(x)$ for all $x \in S$ and all $k \in \IZ$
(that is, $F_0 \in \Li[0]$).

Recall that, if $x,y$ are in the same branch $B_m$, then
$x<y$ means $\Im(x)<\Im(y)$; the other notations related to the order
in $B_m$ are defined consistently.

\begin{lemma}\label{twoarrowscrossing}
Let $F \in \Li$. Let $x,y \in B_0$ and $m\in\IZ $ be such that
$F(x)\in B_m$, $x < y \le F_0(x)$
and $F(y) \notin \Bo_m$.
Assume additionally that $F(0) \notin (x+m, \max B_m]$.
Then $\Per(F) = \IN$.
\end{lemma}

\begin{proof}
First of all, observe that the assumptions $x < y \le F_0(x)$ and the
definition of $F_0$ imply that $F(x) \ge y+m > x+m$.
Hence,
$F(0) \notin (x+m, \max B_m]$ implies $F(0) \ne F(x)$, and thus,
$x \ne 0$.

Consider
$K = [x,y]$ and
$L = [0,x]$, which are closed non-degenerate intervals in $B_0$.
We have
\begin{align*}
 F(K) & \supset \chull{F(x), F(y)}
        \supset \chull{F(x), m} \quad
        \text{because $F(x)\in B_m$ and $F(y) \notin \Bo_m$,} \\
      & \supset (K+m) \cup (L+m) \quad
        \text{because $F(x) \ge y+m > x+m \ge m$.}
\end{align*}
Moreover, since $F(0) \notin (x+m, \max B_m]$ and $y\le F_0(x)$,
\[
 F(L) \supset \chull{F(0), F(x)} \supset K+m.
\]
By Proposition~\ref{prop:SemiHorseshoe}, the map $F-m$
has periodic points of all periods in $K \cup L \subset B_0$.
Therefore,
$\Per(F) = \IN$ by Lemma~\ref{PeriodsAndPeriodsmodiAreFriends}.
This ends the proof of the lemma.
\end{proof}

\begin{lemma}\label{twoarrowscrossingLargeOrbits}
Let $F \in \Li$. Let $x,y \in B_0$ and $m\in\IZ$ be such that $F(x)\in B_m$,
$x < y \le F_0(x)$ and
$|\Re(F(x)) - \Re(F(y))| \ge 1$. Then $\Per(F) = \IN$.
\end{lemma}

\begin{remark}\label{RelatingTheTwoLemmas}
Lemma~\ref{twoarrowscrossingLargeOrbits} is a particular case of
Lemma~\ref{twoarrowscrossing} whenever $F(0)$ is not in a wrong
place, i.e. $F(0) \notin (x+m, \max B_m]$.
\end{remark}

\begin{proof}[Proof of Lemma~\ref{twoarrowscrossingLargeOrbits}]
We can assume additionally that
$F(0) \in B_m$ and $F(0) > x+m$, otherwise Lemma~\ref{twoarrowscrossing}
gives the conclusion (see Remark~\ref{RelatingTheTwoLemmas}).
We shall also assume that $\Re(F(y)) \le m-1$;
the case $\Re(F(y))\ge m+1$ follows in a similar way.

We set $G:=F-m$. Then the three points $x,y, G(x)=F_0(x)$ are in
$B_0$ and $G(x) \ge y > x.$
According to Lemma~\ref{lem:FF+k}(d), $\Per(F)=\Per(G)$, and thus we need
to show that $\Per(G)=\IN$.
We consider two cases.
\begin{case}{1} $G(0) \ge y$. \end{case}
The proof of this case is similar to that of
Lemma~\ref{twoarrowscrossing} by taking
$K = [x,y]$ and $L = [-1,0]$.
Since $\Re(G(y)) \le -1$ we have
\[
 G(K) \supset \chull{G(x), G(y)}
      \supset \chull{G(x), -1}
      \supset K \cup L.
\]
Moreover, since $G(0) \ge y$, we have $G(-1) \in B_{-1}$. Hence,
\[
 G(L) \supset \chull{G(0), G(-1)}
      \supset [x,y] \cup [-1,0]
      = K \cup L.
\]
By Proposition~\ref{prop:SemiHorseshoe}, the map
$G$ has periodic points of all periods in $K \cup L$.

\begin{case}{2} $x < G(0) < y$. \end{case}
In this case, we set $K = [x,y ]$ and $L = \chull{-1,x}$, and we endow
the interval $L$ with the order $<_L$ such that $-1=\min L$.
Observe that $0\neq x$ because $G(0)<y\le G(x)$, and thus $L$
contains the branching point $0$ in its interior.

As in the previous case,
\begin{align*}
 G(K) & \supset K \cup L,\\
 G(L) & \supset \chull{G(-1),G(x)} \supset \chull{-1,G(x)}
        \supset K \cup L.
\end{align*}

However, observe that the covering is negative in the first case and
positive in the second.
In other words, we have
$K \signedcover[-]{G} K,L$ and
$L \signedcover{G} K,L$.
Moreover, $(K+\IZ)\cap (L+\IZ)=\{x\}+\IZ$, and $G(x)\notin L+\IZ$.
Thus Lemma~\ref{lem:+-loop} applies and gives
$\Per(G)\supset \IN\setminus\{2\}.$
So, we have to prove that $2 \in \Per(G)$.
To this end, we shall consider several subcases and
several loops.

Since $G(x) \ge y$ and $\Re(G(y)) \le -1$, there exist points
$x \le x_1 < x_2 < \alpha < y$ in $B_0$ such that
$G(x_1) = y$, $G(x_2) = x_1$ and $G(\alpha) = 0$.
Moreover, we can take $x_2$ and $\alpha$ so that
\begin{align*}
    x_2 &= \max \set{t \in [x_1, y]}{G(t) = x_1}, \text{ and}\\
 \alpha &= \max \set{t \in [x_2, y]}{G(t) = 0} = \max \set{t \in
[x_2,y]}{G(t) \in B_0}.
\end{align*}
Now we consider two subcases.

\begin{case}[Subcase]{2.1} $x < G(0) \le \alpha $. \end{case}
We look at the interval $[x_2,\alpha]$. Observe that, by
Lemma~\ref{lem:FF+k}(a),
\begin{align*}
    G^2(x_2) &= G(x_1) = y > x_2 \text{ and}\\
 G^2(\alpha) &= G(0) \le \alpha.
\end{align*}
Hence, $G^2([x_2,\alpha])\supset [x_2+\alpha]$ and,
since $G^2$ is continuous and there is no branching point
in $[x_2,\alpha]$, there exists a point $z \in (x_2, \alpha]$ such
that $G^2(z) = z$.
From the definition of $x_2$, it follows that
$G([x_2,\alpha]) \cap [x_2, \alpha] = \emptyset$.
Therefore,
$(G(z) + \IZ) \cap [x_2,\alpha] = \emptyset$ and, consequently,
$G(z) - z \notin \IZ$.
Thus, $z$ is periodic {\modi} point of period $2$.

\begin{case}[Subcase]{2.2} $\alpha < G(0) < y$ \end{case}
In this subcase, we need a couple of additional points.
Since $G(0) \in \Bo_0$, it follows that $G(-1) \in \Bo_{-1}$ and,
hence, there exists a point $\beta \in (-1,0)$ such that $G(\beta) =0$.
Using again that $G(\alpha) = 0$ and $\Re(G(y)) \le -1$, we see that
there exists a point $\alpha < u < y$ such that $G(u) = \beta$.
Now we look at the interval $[\alpha,u]$. We have
\begin{align*}
 G^2(\alpha) &= G(0) > \alpha  \text{ and}\\
      G^2(u) &= G(\beta) = 0 < u.
\end{align*}
Hence, there exists a point $z \in (\alpha,u) \subset B_0$ such that
$G^2(z) = z$.
From the definition of $\alpha$, it follows that
$G((\alpha,u)) \cap \Bo_0 = \emptyset$.
So, as in the previous case,
$G(z) - z \notin \IZ$ and $z$ is a periodic {\modi} point of period $2$.
This ends the proof of the lemma.
\end{proof}

\subsection{Proofs of
Theorem~\ref{TheoremSharkovskiiintheBranches}
and Theorem~\ref{LargeOrbitsintheBranches}}

The next lemma relates the maps $F$ and $F_0$ in the situation that
interests us.

\begin{lemma}\label{F0powern}
Let $F\in\Li[d]$. Then the following statements hold:
\begin{enumerate}
\item Assume that there exists $x \in \Bo_0$ and $n \in \IN$ such
that $F_0^i(x) \in \Bo_0$ for all $0\le i\le n$. Then $F^i(x)
\in \cup_{m\in\IZ} \Bo_m$ for all $0\le i\le n$.

\item Assume that there exists $x \in B$ and $n \in \IN$ such that
$F^i(x) \in B$ for all $0\le i\le n$. Then
\[
  F^n(x) = F_0^n(x) + \sum_{k=0}^{n-1} d^k \Re(F(F_0^{n-1-k}(x))) \in
           F_0^n(x) + \IZ.
\]
\end{enumerate}
\end{lemma}

\begin{proof}
Observe that
if $F(x) \in \IR$ then $F_0(x) = 0 \notin \Bo_0$.
Thus (a) holds.
Statement (b) follows from the iterative use of
Lemma~\ref{lem:FF+k}(a) and the definition of $F_0$.
\end{proof}

Given a lifted periodic orbit $P$ that lives in the branches (that is,
$P\subset B$), we set
\begin{equation}\label{eq:P0}
   P_0 := P \cap B_0 = \set{x - \Re(x)}{x \in P}.
\end{equation}

\begin{remark}\label{P0almostperiodic}
From the definitions of $F_0$ and $P_0$, we deduce that
$F_0(P_0) \subset P_0$ and the cardinality of $P_0$ coincides with the
$F$-period of $P$.
\end{remark}

The next lemma summarizes the relation between $P$, $P_0$ and $F_0$.
Its proof is omitted since it follows easily from Lemma~\ref{F0powern}
and Remark~\ref{P0almostperiodic}.

\begin{lemma}\label{P0isperiodic}
Let $F \in \Lifts$ and let $P$ be a lifted periodic orbit of $F$ that
lives in the branches.
Then $P_0$ is a periodic orbit of $F_0$ and the $F_0$-period of $P_0$
coincides with the $F$-period of $P$.
\end{lemma}

We are ready to prove
Theorems~\ref{TheoremSharkovskiiintheBranches} and
\ref{LargeOrbitsintheBranches}. We recall their statement before the proof.

\begin{MainTheorem}
Let $F \in \Lifts$ and let $P$ be a lifted periodic orbit of $F$ of
period $p$ that lives in the branches.
Then $\Per(F) \supset \Shs(p)$.
Moreover, for every $d \in \IZ$ and every $p \in \N_{\Sho}$,
there exists a map $F_p \in \Li[d]$ such that $\Per(F_p) = \Shs(p)$.
\end{MainTheorem}

\begin{proof}
Since $\chull{P_0}$ is a compact interval included in $B_0$, the
retraction on $\chull{P_0}$ is the continuous map
$\map{\ret_{\chull{P_0}}}{S}[\chull{P_0}]$ defined by:
\[
 \ret_{\chull{P_0}}(x) = \begin{cases}
           x   & \text{if $x \in \chull{P_0}$}\\
           \max P_0 & \text{if }x \in B_0 \text{ and } x\ge \max P_0,\\
           \min P_0 &\text{otherwise}.
          \end{cases}
\]
We define
$\psi:=\ret_{\chull{P_0}}\circ F_0\evalat{\chull{P_0}}$.
Then $\map{\psi}{\chull{P_0}}$ is a continuous interval map such that
$\psi\evalat{P_0} = F_0\evalat{P_0}$ and
\begin{equation}\label{SpecialCondition}
  \psi(z) = F_0(z) \text{ for every }
      z \in \chull{P_0}\setminus \psi^{-1}(\{\min P_0,\max P_0\}).
\end{equation}

By Lemma~\ref{P0isperiodic}, $P_0$ is a periodic orbit of $\psi$ of
period $p$. Fix $q\in\Shs(p)$ with $q\neq p$.
By Sharkovsky's theorem on the interval (see
\cite{SharOri, SharTrans} or Theorem~\ref{GMT1} for $n=2$),
there exists a periodic orbit $Q \subset
\chull{P_0}$ of $\psi$ of period $q$.
We have to show that $F$ has a lifted periodic orbit of period $q$.

Notice that
$Q \cap P_0 = \emptyset$ and $Q \cap \psi^{-1}(P_0) = \emptyset$
since both are periodic orbits of $\psi$ of different period.
Therefore, $Q \subset \Bo_0$ and $\psi\evalat{Q} = F_0\evalat{Q}$ by
\eqref{SpecialCondition}. Let $d$ denote the degree of $F$.
Then, by Lemma~\ref{F0powern},
\begin{equation}\label{SpecialConditionpsi}
\forall x\in Q,\ \forall n\in\IN,\ F^n(x) =
\psi^n(x) + \sum_{k=0}^{n-1} d^k \Re(F(\psi^{n-1-k}(x)))\in \psi^n(x)+\IZ.
\end{equation}

To prove that $F$ has a periodic {\modi} point of period $q$, we take
any $x \in Q$ and we prove that
$F^k(x) - x \notin \IZ$ for $k=1,2,\dots,q-1$ and
$F^q(x) -x \in \IZ$.
This last statement follows trivially from \eqref{SpecialConditionpsi}
because $\psi^q(x) = x$.
Assume that
$F^k(x) = x + l$ for some $k \in \{1,2,\dots,q-1\}$ and some $l \in \IZ$.
Then, again from \eqref{SpecialConditionpsi},
$\psi^k(x)  = x + \widetilde{l}$ for some $\widetilde{l} \in \IZ$.
Since both $x$ and $\psi^k(x)$ belong to
$Q \subset \chull{P_0} \subset B_0$, it follows that
$\widetilde{l} = 0$ and, hence, $\psi^k(x)  = x$; a contradiction.
Consequently, $F^k(x) - x \notin \IZ$ for $k=1,2,\dots,q-1$.

The proof of the second part is easy. Fix $p\in\N_{\Sho}$.
By \cite{Stefan} (see also
\cite{ALM}), there exists a map $f_p \in \CC^0([0,1])$  such that the
set of periods of $f_p$ is precisely $\Shs(p)$.
Now we define the map $F_p \in \Li[d]$ as follows. First we define
$F_p$ on $B_0$ by setting
\[
   \forall x\in[0,1],\ F_p(x\iota) := f_p(x)\iota,
\]
where $\iota$ denotes the square root of
$-1$. Notice that this formula defines $F_p(0)$.
Then we define $F_p$ such that it maps the interval $[0,1]$ onto
$\chull{F_p(0), F_p(0) + d}$ in an expansive (affine) way. With this
we have defined $F_p$ in the set of all $x\in S$ such that $\Re(x) \in
[0,1)$. Finally, we extend $F_p$ to the whole $S$ by the formula
$F_p(x) = F_p(x-\lfloor\Re(x)\rfloor) + d\lfloor\Re(x)\rfloor$, where
$\lfloor\cdot\rfloor$ denotes the integer part function. Clearly, the
map $F_p$ is continuous and has degree $d$. Moreover, each periodic
orbit of $f_p$ corresponds to a periodic orbit of $F_p\evalat{B_0}$.
Hence, $\Per(F_p) \supset \Shs(p)$.
To end the proof of the theorem we have to show that, indeed,
both sets coincide.

To see this, we note that $F_p(B) \subset B$ because
$F_p(B_0) \subset B_0$. We claim that $F_p$ has no periodic {\modi}
points in $S\setminus B = \IR\setminus \IZ$ other that fixed {\modi}
points. Indeed, when  $d = 0,$ $F_p(\IR) = F_p(0) \in B_0$ and there
are no periodic {\modi} points in $\IR\setminus \IZ$. When $d \ne 0$,
there exist points $0 \le x_1 < x_2 \le 1$ such that $F_p([0,x_1])
\subset B_0,$ $F_p([x_2,1]) \subset B_d$ and $F_p([x_1,x_2]) = [0,d]$.
Therefore, there are no periodic {\modi} points in
$[0,x_1] \cup [x_2,1]$ other than, perhaps, $0$ and $1$
(which are already contained in B);
and the only periodic {\modi} points in $(x_1,x_2)$ are fixed
{\modi} points because $F_p\evalat{(x_2,x_2)}$
is expansive.
This proves the claim. Since $F_p$ has already fixed {\modi} points
in $B$, there are no new periods of $F_p$ in $S \setminus B$.

Now we are going to show that, if $x \in B$ is a periodic {\modi} point of
period $q$, then $q \in \Shs(p)$.
Clearly, $\widetilde{x} := x - \Re(x) \in B_0$ and
$F_p^n(\widetilde{x}) \in B_0$ for every $n \ge 0$.
Then, by Lemma~\ref{PeriodsAndPeriodsmodiAreFriends},
$\widetilde{x}$ is a periodic point of $F_p$ of period $q$ whose
orbit is contained in $B_0$. Therefore, $q$ is a period of the
original map $f_p$ and, thus, $\Per(F_p) = \Shs(p)$.
This ends the proof of the theorem.
\end{proof}

\begin{MainTheorem}
Let $F \in \Li$ and let $Q$ be a large orbit of $F$ such that $Q$
lives in the branches. Then $\Per(F) = \IN$.
\end{MainTheorem}

\begin{proof}
Let $P = Q + \IZ \subset B$ be the lifted orbit corresponding to $Q$
and set $q:=\Card(Q)$.
Recall that $F_0$ and $P_0$ are defined by \eqref{eq:F0} and \eqref{eq:P0}.
By Lemma~\ref{P0isperiodic}, $P_0$ is a periodic orbit of $F_0$ of
period $q$.
We are going to show, by a recursive argument, that there exist $x,y \in P_0$
such that $x < y \le F_0(x)$ and
$\Re(F(x)) \ne \Re(F(y))$. Then the theorem follows from
Lemma~\ref{twoarrowscrossingLargeOrbits}.

We set $A_0 := \{\min P_0\}$ and, for all $i \ge 0$, we define
\[
 A_{i+1} := \set{z \in P_0}{z \le \max F_0(A_i)}.
\]
It follows from this definition that, if
$\max F_0(A_i)\le \max A_i$, then $F_0(A_i)\subset A_i$, which implies
that $A_i=P_0$ because $A_i$ is included  in $P_0$, which is a periodic
orbit of $F_0$. Therefore, either
$A_i \varsubsetneq A_{i+1}$ (when $\max F_0(A_i)> \max A_i$),
or $A_i=P_0$. Clearly, $A_{i+1}=P_0$ whenever $A_i=P_0$.
This implies that
\begin{equation}\label{eq:Ai}
\forall i\ge 0,\ A_i\subset A_{i+1}\quad\text{and}\quad
\forall i\ge q-1,\ A_i = P_0.
\end{equation}
On the other hand, the function $\Re(F(\cdot))$ is not constant on
$P_0$. To prove this, assume that there exists $m\in \IZ$ such that
\begin{equation}\label{eq:FP0m}
\Re(F(P_0)) = \{m\}.
\end{equation}
Choose $z \in P_0$ and let $s \in \IN$ be such that
$z + s \in Q$. Then, since $Q$ is a true periodic orbit of $F$ and
$P_0$ is a periodic orbit of $F_0$, both of period $q$, we have
$F^q (z+s) = z+s$ and $F_0^q (z) = z$.
Lemma~\ref{F0powern}(b) implies that $F^q(z)=F_0^q(z)+qm$ (note that
$\forall k, \Re \circ F \circ F_0^{q-1-k} = m$ by \eqref{eq:FP0m}).
We then have
\begin{align*}
z+s &= F^q (z + s) = F^q (z) + s\quad\text{by Lemma~\ref{lem:FF+k}(a)}\\
    &= F_0^q (z) + q m + s\\
    &= z + q m + s.
\end{align*}
Hence, $m = 0$ and, consequently, $\forall n\ge 0$,
$F^n(z+s) = F^n(z) + s = F_0^n(z) + s$, again by Lemma~\ref{F0powern}(b)
and \eqref{eq:FP0m}. So,
\begin{align*}
 Q &= \set{F^n(z+s)}{n=0,1,\dots,q - 1} \\
   &= \set{F_0^n(z)}{n=0,1,\dots,q - 1} + s \\
   &= P_0 + s \subset B_s.
\end{align*}
This contradicts the fact that $Q$ is a large orbit and, hence,
the function $\Re(F(\cdot))$ is not constant on $P_0$.
Using this fact and \eqref{eq:Ai}, we see that
there exists $1 \le k \le q-1$ such that the function $\Re(F(\cdot))$
is constant on $A_{k-1}$ (and hence its value is $\Re(F(\min P_0))$)
but there exists $y \in A_k \setminus A_{k-1}$ such that
$\Re(F(y)) \ne \Re(F(\min P_0))$.
By definition, $y \le \max F_0(A_{k-1})$.
Let $x \in A_{k-1}$ be such that $F_0(x) = \max F_0(A_{k-1})$.
Then, since $y \notin A_{k-1}$, we have
$x < y \le \max F_0(A_{k-1}) = F_0(x)$.
Moreover, $\Re(F(\min P_0)) = \Re(F(x))$ because $x\in A_{k-1}$,
and thus we have $\Re(F(y)) \ne  \Re(F(x))$.
This ends the proof of the theorem.
\end{proof}

\section{Periods (mod 1) when 0 is in the interior of the rotation
interval}\label{sec:0inIntRotR}

This section is devoted to prove the next theorem.

\begin{MainTheorem}
Let $F\in\Li$. If $\Int(\RotR(F)) \cap \Z \ne \emptyset$, then
$\Per(F)$ is equal to,
either $\IN$, or $\IN\setminus\{1\}$, or $\IN\setminus\{2\}$.
Moreover, there exist maps $F_0, F_1, F_2 \in \Li$
with $0 \in \Int(\RotR(F_i))$ for $i=0,1,2$ such that
$\Per(F_0) = \IN$,
$\Per(F_1) = \IN\setminus\{1\}$ and
$\Per(F_2) = \IN\setminus\{2\}$.
\end{MainTheorem}

In the first subsection, we construct the maps $F_0$, $F_1$ and $F_2$
from the statement of Theorem~\ref{theo:0inInterior}.
Then, in Subsection~\ref{subsec:OIntRot}, we prove two lemmas, both
giving conditions to obtain $\Per(F)\supset \N \setminus \{1\}$.
Finally we prove the first statement of
Theorem~\ref{theo:0inInterior} in the last and biggest subsection.

\subsection{Construction of examples}

We give below two examples of maps with $0\in\Int(\RotR(F))$ and
$\Per(F)=\IN\setminus\{1\}$ (resp. $\Per(F)=\IN\setminus\{2\}$).
The case $0\in\Int(\RotR(F))$ and $\Per(F)=\IN$ is trivially obtained
from a lifting of a circle map with this property
(just extend the map to $S$ by collapsing $B_0$ to $F(0)$ under the
action of $F$); see e.g. \cite[Section 3.10]{ALM} for such circle
maps.

\begin{example}\label{ex:0inintRot-1}
We are going to build a map $F\in\Li$ such that $0\in\Int(\RotR(F))$
and $\Per(F)=\IN\setminus\{1\}$. Moreover, there is a large orbit of
period $n$ for some fixed $n\ge 3$, which shows that the existence of
a large orbit is not enough to imply all periods \modi.

We fix an integer $n\ge 3$. Let $a_0,a_1,\ldots, a_n\in [0,1]$ be such
that
$0=a_0<a_1<a_2< \cdots<a_{n-1}<a_n=1$.
We set $A_i=[a_{i-1},a_i]$ for all $1\le i\le n$.
We define $F\in \Li$ such that $F(a_i)=a_{i-1}$ for all $3\le i\le n$,
$F(a_2)=\max B_0$, $F(a_1)=0$, $F(\max B_0)=a_2+1$, and $F$ is affine
on $B_0$ and $A_i$ for all $1\le i\le n$. The map $F$ and its Markov
graph are illustrated in Figure~\ref{fig:0inintRot-1}.
\begin{figure}[htb]
\centerline{\includegraphics{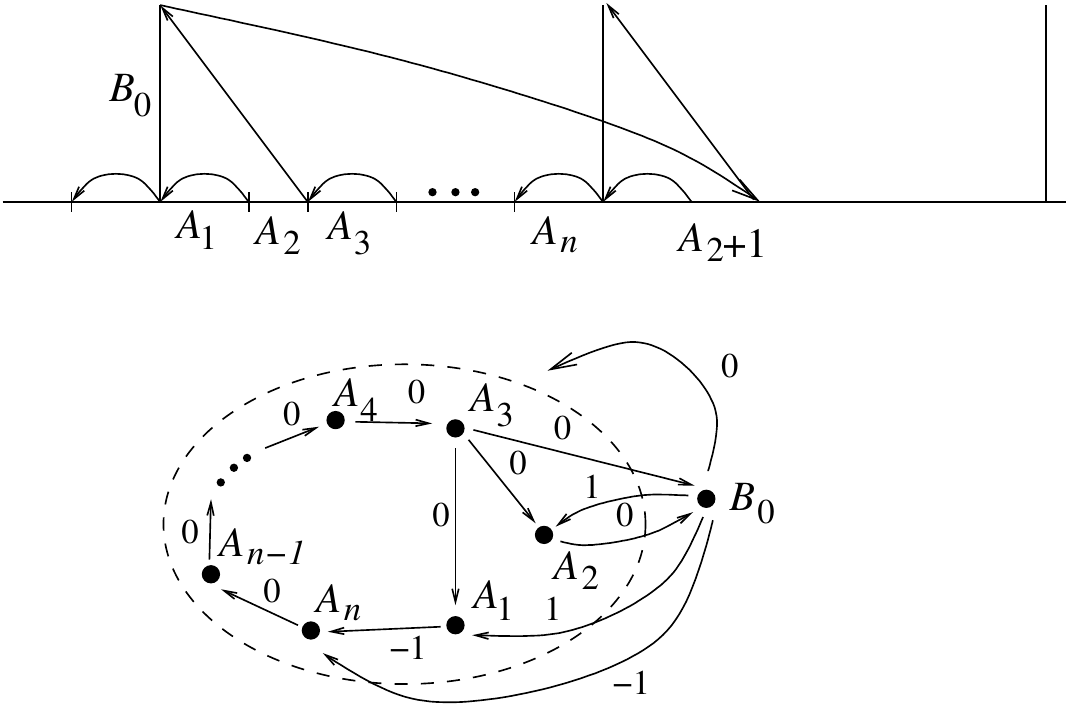}}
\caption{Above: the map $F$ of Example~\ref{ex:0inintRot-1}. Below:
its Markov graph. The arrow from $B_0$ to the dotted set means that
there are arrows $B_0\signedcover[0]{}A_i$ for all $1\le i\le
n$.}\label{fig:0inintRot-1}
\end{figure}

By using the tools from \cite[Subsection~6.1]{AlsRue2008}
one can compute from its Markov graph that $\Per(F)=\Per(0,F)=\{n\ge
2\}$ and $\RotR(F)=\left[-\frac1{n-1},\frac12\right]\ni 0$.
The loop
\[
B_0 \signedcover[1]{}A_1\signedcover[-1]{}A_n
\signedcover[0]{}A_{n-1}\signedcover[0]{}\cdots
\signedcover[0]{}A_3\signedcover[0]{}B_0
\]
gives a large orbit of period $n$.
\end{example}

\begin{example}\label{ex:0inintRot-2}
We are going to build a map $F\in\Li$ such that $0\in\Int(\RotR(F))$
and
$\Per(F)=\IN\setminus\{2\}$.

Let $t_0,t_1,t_2, z_0,z_1\in [0,1]$ be such that
$0<t_2<t_1<t_0<z_0<z_1<1$.
We set $I_2=[0,t_2]$, $I_1=[t_2,t_1]$, $I_0=[t_1,t_0]$, $C=[t_0,z_0]$,
$J_0=[z_0,z_1]$ and $J_1=[z_1,1]$.
We define $F\in \Li$ such that $F(t_0)=t_1$, $F(t_1)=t_2$,
$F(t_2)=t_0-1$,
$F(z_0)=z_1$, $F(z_1)=\max B_1$, $F(\max B_0)=z_0$, $F(0)=0$
and $F$ is affine on $B_0, I_0,I_1, I_2, J_0, J_1, C$.
The map $F$ and its Markov graph are illustrated in
Figure~\ref{fig:0inintRot-2}.
\begin{figure}[htb]
\centerline{\includegraphics{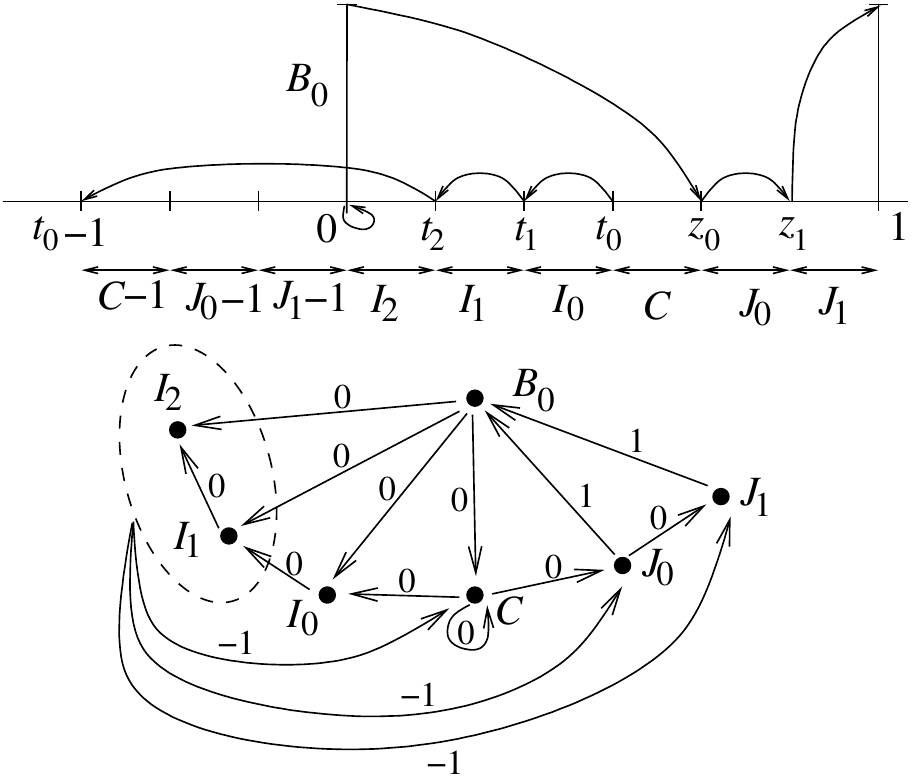}}
\caption{Above: the map $F$ of Example~\ref{ex:0inintRot-2}.
Below: its Markov graph.
The arrows from the dotted set mean that there are arrows
$I_i\signedcover[-1]{}C,J_0,J_1$ for $i=1,2$.}\label{fig:0inintRot-2}
\end{figure}

By using the tools from \cite[Subsection~6.1]{AlsRue2008}
and using the loops
\[
  C \signedcover[0]{}
  J_0 \signedcover[1]{}
  B_0 \signedcover[0]{} C,
  \quad
  C \signedcover[0]{} C
  \quad\text{and}\quad
  C \signedcover[0]{}
  I_0 \signedcover[0]{}
  I_0 \signedcover[-1]{} C,
\]
one can compute that $\Per(F)=\IN\setminus\{2\}$
and $\RotR(F)=\left[-\frac13,\frac13\right]$.
\end{example}

\subsection{Situations that imply periodic points of all periods
except 1}\label{subsec:OIntRot}

The aim of this subsection is to prove Lemmas \ref{lem:allperiods-1}
and \ref{lem:F(R)-left} below, both giving conditions to obtain
$\Per(F)\supset \N \setminus \{1\}$. They will be used in the proof of
Theorem~\ref{theo:0inInterior}.

There is a common idea in the hypotheses of both lemmas: some points
of $\IR$ go to the left whereas others go sufficiently to the right
and have an orbit passing through the branches.
In Lemma~\ref{lem:allperiods-1}, the assumption is that there is a
point $x\in\IR$ such that $F(x)$ is in the branch $B_0$ and $F^2(x)$
is much to the right (or much to the left) of $F(0)$.
In Lemma~\ref{lem:F(R)-left}, assumption (a) means that all points in
$\IR$ go rather to the left (or at least do not go much to the right)
under one iteration, whereas assumption (b) implies that there is one
point $x_0$ in $\IR$ whose orbit tends to $+\infty$; because of (a),
the orbit of $x_0$ must pass through the branches.

Intuitively, the fact that some points of the real line
go to the left whereas others go to the right is clearly related to
the fact that there exist points $x_,x'\in\IR$ such that $\rhos(x)<0$
and $\rhos(x')>0$, and hence $0\in\Int(\RotR(F))$.

\begin{lemma}\label{lem:allperiods-1}
Let $F\in\Li$. Suppose that there exists $y_0\in F(\IR)\cap B_0$
such that, either
$\Re(F(y_0))\ge \lceil \Re(F(0))\rceil +1$, or
$\Re(F(y_0))\le \lfloor  \Re(F(0))\rfloor -1$.
Then $\Per(F)\supset \N \setminus \{1\}$.
\end{lemma}

\begin{lemma}\label{lem:F(R)-left}
Let $F\in\Li$. Suppose that
\begin{enumerate}
\item $\forall x\in \IR, x<0\Longrightarrow \Re(F(x))<0$,
\item $\exists x_0\in\IR, \rho(x_0)>0$.
\end{enumerate}
Then $\Per(F)\supset\IN\setminus\{1\}$.
\end{lemma}

We also need two lemmas that, unfortunately, are rather technical.
Roughly speaking, the conclusion of Lemma~\ref{lem:PotImp:ToutVaBien}
is that, either we have a ``good'' point in $F(\IR)$ and we may hope
to apply Lemma~\ref{lem:allperiods-1}, or we are in a ``good''
situation in view of Lemmas~\ref{twoarrowscrossing} or
\ref{twoarrowscrossingLargeOrbits}. Lemma~\ref{lem:super} summarizes
the various conclusions we can obtain in this situation.

\begin{lemma}\label{lem:PotImp:ToutVaBien}
Let $F\in \Li$, $z \in \IR$ and $u \in \Orb(z,F) \setminus
\IR.$ Then there exists $y \in \Orb(z,F) \setminus \IR$ satisfying
\[
  y-\Re(y) \le u -\Re(u)
    \quad\text{and}\quad
  \Re(F(y)) - \Re(y) = \Re(F(u)) - \Re(u)
\]
and such that
\begin{itemize}
\item either $y \in F(\IR)$,
\item or there exists $x\in B_0$ such
that $x < y-\Re(y) \le F_0(x)$ and $\Re(F(x)) \ne \Re(F(y)) - \Re(y)$.
\end{itemize}
\end{lemma}

\begin{lemma}\label{lem:super}
Let $F\in \Li$, $z \in \IR$ and $u \in \Orb(z,F) \cap B_0$.
Then, there exists $y \in \Orb(z,F) \cap B_0$ such that $y \le u$ and
$\Re(F(y)) = \Re(F(u))$, and one of the following situations occurs:
\begin{itemize}
\item $y \in F(\IR)$,
\item $\Per(F) = \IN$,
\item $y \notin F(\IR)$
and there exists a point $x\in B_0$ such that $x < y \le F_0(x),$
$F(0) \in (x+m, \max B_m]$ and
$F(y) \in (m-1, m+1)\setminus\{m\}\subset \R$,
where $m:=\Re(F(x)) \in \Z$.
\end{itemize}
\end{lemma}

Next we prove the above four lemmas.

\begin{proof}[Proof of Lemma~\ref{lem:allperiods-1}]
We assume that
$\Re(F(y_0)) \ge \lceil \Re(F(0)) \rceil + 1;$
the other case is symmetric.
In particular $0 \ne y_0 \in \Bo_0$.
By the continuity of $F$,
there exist $y_1,y_2 \in B_0$, $y_1 < y_2 \le y_0$,
such that
$F(y_1) = \lceil \Re(F(0)) \rceil$ and
$F(y_2) = \lceil \Re(F(0)) \rceil + 1.$
Let $D = [y_1,y_2] \subset B_0.$
We have $F(D) \supset [F(y_1), F(y_2)]$,
and hence $D \arrowto [0,1] + \lceil \Re(F(0)) \rceil.$
Let $\widetilde{a} \in \IR$ be such that
$F(\widetilde{a}) = \max  F(\IR) \cap B_0,$
$q = \lfloor \widetilde{a} \rfloor$ and
$a = \widetilde{a} - q \in [0,1).$
We have $F(a) \in B_{-q}$ and $F(a) + q \ge y_0$.
In the rest of the proof, all the coverings are for the map $F$
and the notation $I \arrowto J\ \modi$ means that $I \arrowto J+n$ for
some $n\in\IZ$.

\begin{case}{1} $F(0) \notin B$ (see Figure~\ref{fig:case1.1}). \end{case}
This assumption implies that $y_1 \neq 0$,
and thus $D \cap \IR = \emptyset$.
Set $A_1 = [0,a]$ and $A_2 = [a,1]$.
Since $F(a) \in B_{-q}$, the set $F(A_1)$ contains
$\chull{F(0),F(a)} \supset \chull{F(0),-q},$ and similarly
$F(A_2)$ contains $\chull{-q,F(1)} = \chull{-q, F(0)+1}.$
Thus, if $F(0) \le a-q-1$ then $A_1 \arrowto A_{2}-q-1,$
and if $F(0) \ge a-q-1$ then $A_2 \arrowto A_{1}-q$.
Moreover, we have
$A_1 \arrowto D-q$ and
$A_2 \arrowto D-q$ because $F(a) + q \ge y_0 \ge y_2$ and
$F(0), F(1)\notin B$.
Therefore we have one of the covering graphs of
Figure~\ref{fig:Markov-diagram}.
\begin{figure}[htb]
\centerline{\includegraphics{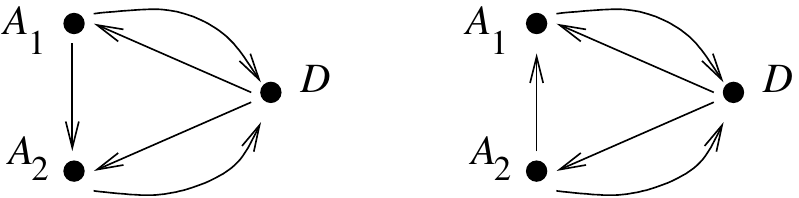}}
\caption{The two possible covering graphs in case~1 (arrows are
$\modi$).}\label{fig:Markov-diagram}
\end{figure}

Suppose that we are in the first case, i.e.
$A_1 \arrowto A_2\ \modi$ (see Figure~\ref{fig:case1.1}).
Since $A_2 \arrowto D\ \modi$,
there exists $c \in A_2$ such that  $F(c) = y_1\ \modi$.
Moreover $c\notin \{a,1\}$ because $F(a) \ge y_2$ and $F(1) \in \IR$.
Similarly, there exist $y_3 \in (y_1,y_2)$ such that $F(y_3) = c\ \modi$,
and $b \in (a,c)$ such that $F(b) = y_3\ \modi$.
\begin{figure}[htb]
\centerline{\includegraphics[width=\textwidth]{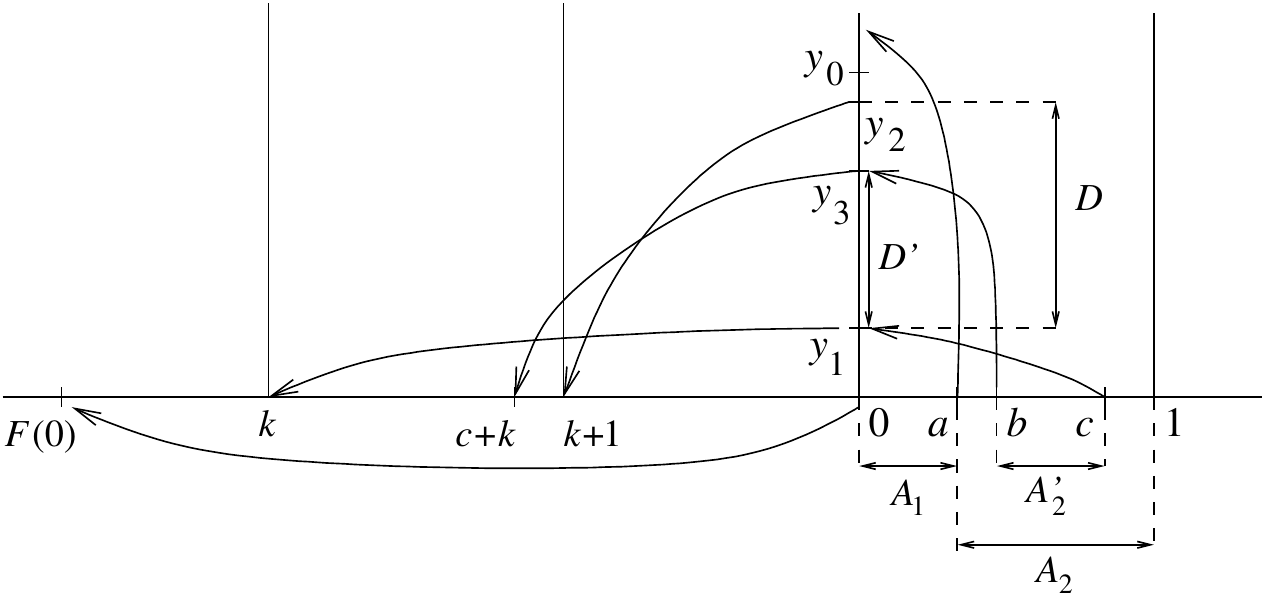}}
\caption{Positions of the different points in Case~1,
where $k = \lceil F(0) \rceil$ (the figure is drawn with $q=0$).
}\label{fig:case1.1}
\end{figure}
Let $D' = [y_1,y_3] \subset D$ and $A_2'= [b,c] \subset A_2$.
Then $D' \arrowto A_1 \cup A_2'\ \modi$ and $A_2 \arrowto D'\ \modi.$
That is, we have the covering graph shown on the left picture of
Figure~\ref{fig:Markov-diagram} by replacing
$A_2$ and $D$ by $A_2'$ and $D',$ respectively.
Moreover, the sets $A_1 + \IZ,\ A_2' + \IZ$ and $D' + \IZ$ are disjoint,
and $A_1, A_2', D'$ contain no branching point in their interior.
Therefore, to show that there exist periodic $\modi$ points of period $n$,
it is enough to show that there exists a non-repetitive loop of
length $n$ in the covering graph.
Consider the following loops in the covering graph:
\begin{align*}
\CC_2  & := D' \arrowto A_1 \arrowto D',\\
\CC_2' & := D' \arrowto A_2' \arrowto D',\text{ and}\\
\CC_3  & := D' \arrowto A_1 \arrowto A_2' \arrowto D',
\end{align*}
where the arrows are {\modi}.
Fix $n\ge 2$.
If $n$ is even, we write $n=2m$ and we consider the loop
$\CC_2'(\CC_2)^{m-1}$.
If $n$ is odd, we write $n=2m+1$ and we consider the loop
$\CC_3(\CC_2)^{m-1}$.
In both cases, we obtain a non-repetitive loop of length $n$.
By Proposition~\ref{prop:covering}, there exists a point $x\in D'$
such that $F^n(x)-x\in\IZ$ and
\begin{gather*}
\forall\, 0\le i\le m-1,\ F^{n-2i}(x)\in D'+\IZ,
\quad\forall\, 1\le i\le m-1,\
F^{n-2i+1}(x)\in A_1+\IZ,\\
F^{n-2m+1}(x)\in A_2'+\IZ\quad
\text{and, if $n$ is odd, } F(x)\in A_1+\IZ.
\end{gather*}
Thus $x$ is periodic {\modi} for $F$ and its period divides $n$.
Since the intervals $A_1,A_2', D'$ are disjoint {\modi}, one can show
that its period {\modi} is exactly $n$. Indeed, consider $1<d<n$.
Then $F^{n-2m+1}(x)\in A_2'+\IZ$ and $F^{n-2m+1+d}(x)$ belongs to,
either $A_1+\IZ$ , or $D'+\IZ$, and thus the period {\modi} of $x$ is
not $d$.

The second case (i.e. when $A_2 \arrowto A_1$) is similar: there
exist $c\in (0,a)$, $y_3\in (y_1,y_2)$ and $c\in (b,a)$ such that
$F(c)=y_1\ \modi$, $F(y_3)=c\ \modi$ and $F(b)=y_3\ \modi$.
If we let $A_1'=[c,b]$ and $D'=[y_1,y_3]$, then we have the
covering graph shown on the right picture of
Figure~\ref{fig:Markov-diagram} by replacing
$A_1$ and $D$ by $A_1'$ and $D',$ respectively.
The rest of the proof is the same as before by interchanging
the roles of $A_1, A_2$.
Therefore, $F$ has periodic $\modi$ points of period $n$ for all
$n\ge 2$.

\begin{figure}[htb]
\centerline{\includegraphics[width=\textwidth]{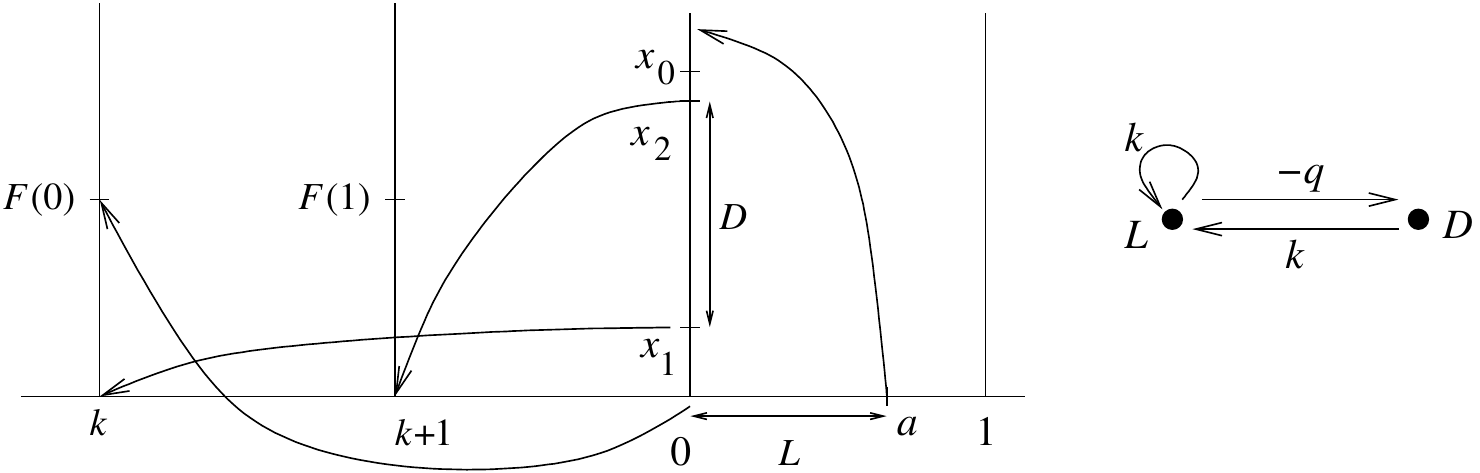}}
\caption{
Left side: positions of the different points in Case~2,
where $k=\lceil F(0)\rceil$ and $k<-q$
(the figure is drawn with $q=0$).
Right side: the covering graph in Case~2
(both when $k\ge -q$ and $k<-q$).}\label{fig:case2}
\end{figure}

\begin{case}{2} $F(0)\in B$.\end{case}
Let $k=\Re(F(0))\in \IZ$
(that is, $F(0) \in B_k$ and $F(1) \in B_{k+1}$).
Observe that the set $F([0,1])$ contains the points
$F(a), F(0), F(1)$, with $F(a)\in B_{-q}$ and $F(a)+q \ge y_0$.
When $k \ge -q$, we set $L=[a,1]$. Then,
\[
F(L) \supset \chull{F(a),F(1)} \supset \chull{y_0-q,k+1} \supset
\chull{y_0-q,-q} \cup \chull{k, k +1} \supset (D-q) \cup (L+k).
\]

When $k < -q$, we set $L=[0,a]$ (see Figure~\ref{fig:case2}).
Then,
\[
F(L) \supset \chull{F(0),F(a)} \supset \chull{k,y_0-q} \supset
\chull{k,k+1} \cup \chull{-q,y_0-q} \supset (D-q)\cup (L+k).
\]
Observe that, in both cases, $F(D) \supset [0,1] + k \supset L+k$
and, hence, $F$ has the covering graph on the right side of
Figure~\ref{fig:case2}.
Thus, $\Per(F)=\IN$ by Lemma~\ref{lem:SemiHorseshoe-mod1}.
\end{proof}

\begin{proof}[Proof of Lemma~\ref{lem:F(R)-left}]
We set
$E_0 := \R$ and $E_i := F(E_{i-1})$ for $i \ge 1.$
Since $F(\R) \supset \R$, $E_i$ is a non-decreasing sequence
of closed connected subsets of $S$.
Thus, $E_i \cap B_0$ is a closed subinterval of $B_0$ containing 0.

The sets $E_i$ are periodic {\modi},
i.e. $E_i = E_i + k$ for every $i \in \N$ and $k \in \Z.$
Indeed, $E_0$ is clearly periodic {\modi}.
If $E_i = E_i + k$ for some $i \in \N$ and every $k \in \Z,$ then
\[
E_{i} = F(E_i) = F(E_i + k) = F(E_i) + k = E_{i+1} + k.
\]

We claim that there exists $n \in \N$ such that
$\max \Re(F(E_n \cap B_0)) \ge 1.$
To prove the claim, set
$
\R_{<1} := \set{x\in S}{\Re(x) < 1}
         = (-\infty,1) \cup \bigcup_{k\le 0} B_k
$
and assume that $\Re(F(E_i \cap B_0)) < 1$ for every $i \in \N$.
By Lemma~\ref{lem:FF+k}(a) and assumption (a),
\[
 F(E_i \cap \R_{<1}) \subset E_{i+1} \cap \R_{<1}
\]
for every $i \in \N$.
Consequently,
\[
 F^i(E_0 \cap \R_{<1}) \subset E_i \cap \R_{<1} \subset \R_{<1}
\]
for every $i \in \N$.
Thus, for all $x \in (-\infty,1) = E_0 \cap \R_{<1},$
$\rho(x) \le 0.$
Since $\rhos(x+k)=\rhos(x)$ for every $x \in S$ and $k \in \Z$
we get $\rhos(x) \le 0$ for every $x \in \R;$
a contradiction with assumption (b).
This proves the claim.

Let $n \in \N$ be the smallest integer such that
$\max \Re(F(E_n \cap B_0)) \ge 1.$

Observe that the continuity of $F$ and the assumption (a) imply
that $\Re(F(0)) \le 0$ (in particular $\Re(F(E_0 \cap B_0)) < 1$).
Hence, $n \ge 1.$
If $n = 1$ then Lemma~\ref{lem:allperiods-1} applies and we have
$\Per(F) \supset \N \setminus \{1\}$.

So, in the rest of the proof we assume $n \ge 2.$
Since $E_n \cap B_0$ is a closed subinterval of $B_0$
containing 0, and $\Re(F(0)) \le 0,$
the continuity of $F$ implies that there exists
$y \in E_n \cap \Bo_0$ such that $F(y) = 1.$
By the minimality of $n,$ $y \notin E_{n-1}.$

Let $\overline{x} \in E_{n-1}$ be such that
$F(\overline{x}) = \max E_n \cap B_0 \ge y.$
If $\overline{x} \in E_{n-2}$ then
$E_{n-1} \cap B_0 \supset [0, F(\overline{x})] \ni y;$
a contradiction.
Consequently, $\overline{x} \in \Bo_{k}$ for some $k \in \Z$
because $\R \subset E_{n-2}.$
Set $x = \overline{x} - k \in E_{n-1} \cap \Bo_0.$
If $x \ge y$ then the connectedness of $E_{n-1}$ implies that
$y \in E_{n-1};$ a contradiction.
Hence, $x < y.$
On the other hand, $F_0(x) = F(\overline{x}) \ge y$ and
$F(x) = F(\overline{x}) - k \in B_{-k}.$
In particular $\Re(F(x)) \in \Z.$
The minimality of $n$ and the fact that $x \in E_{n-1} \cap B_0$
implies that $\Re(F(x)) < 1$ and, hence, $\Re(F(x)) \le 0.$ Therefore,
$
|\Re(F(x))-\Re(F(y))| = \Re(F(y)) - \Re(F(x)) = 1 - \Re(F(x)) \ge 1.
$
Then the lemma follows from Lemma~\ref{twoarrowscrossingLargeOrbits}.
\end{proof}

\begin{proof}[Proof of Lemma~\ref{lem:PotImp:ToutVaBien}]
If $u \in F(\IR)$ then we are done by taking $y=u$.
So, in what follows we assume that $u \notin F(\IR)$.
Then, since $z \in \IR$ and $u \in \Orb(z,F)$ there exists
$\overline{z} \in \Orb(z,F) \cap F(\IR)$ and $l \ge 1$ such that
\begin{equation}\label{eq:Fi(z)}
  F^l(\overline{z}) = u\text{ and }
  F^i(\overline{z}) \notin  F(\IR)\text{ for } i=1,2,\dots,l.
\end{equation}
Since $F(\overline{z}) \notin F(\IR),$ $\overline{z} \notin \IR$.
Also, since $F(\IR) \supset \IR,$
$F^i(\overline{z}) \in \cup_{j\in\IZ} \Bo_j$ for $i=1,2,\dots,l$.
Notice that
$\overline{z} - \Re(\overline{z}), u - \Re(u) \in \Bo_0$ and
$0 < \overline{z} - \Re(\overline{z}) < u - \Re(u)$.
Otherwise,
$\overline{z}-\Re(\overline{z}) \ge u - \Re(u)$ and,
since $F(\IR)$ contains $\IR\cup\{\overline{z}\}$ and is connected,
we obtain
$F(\IR) \supset \chull{0,\overline{z}-\Re(\overline{z})}+\IZ \ni u;$
a contradiction.

If $\Re(F(u)) - \Re(u) = \Re(F(\overline{z})) - \Re(\overline{z})$,
then we set $y = \overline{z}$ and the lemma follows.

So, in the rest of the proof, we set
$\widetilde{z} := \overline{z} - \Re(\overline{z})
              \in \Bo_0 \cap F(\IR)$ and we assume that
\[
  \Re(F(\widetilde{z})) = \Re(F(\overline{z})) - \Re(\overline{z})
                        \ne \Re(F(u)) - \Re(u).
\]
By Lemma~\ref{F0powern}(b) and the fact that $F_0$ has degree 0,
$F^i(\overline{z}) - F_0^i(\widetilde{z}) \in \Z$
for $i=0,1,2,\dots,l.$
Consequently, $F_0^i(\widetilde{z}) \in \Bo_0$ and
\begin{equation}\label{eq:displ}
F^i(\overline{z}) = F_0^i(\widetilde{z}) + \Re(F^i(\overline{z}))
\end{equation}
for $i=0,1,2,\dots,l.$
In particular $u = F^l(\overline{z}) = F_0^l(\widetilde{z}) + \Re(u).$
Hence $F_0^l(\widetilde{z}) = u-\Re(u) > \widetilde{z}$ and,
hence,
\begin{equation}\label{eq:p}
\text{there exists } p \in \{0,1,2,\dots,l-1\} \text{ such that }
  F_0^p(\widetilde{z}) < u - \Re(u) \le F_0^{p+1}(\widetilde{z}).
\end{equation}
If $\Re(F(F_0^p(\widetilde{z}))) \ne \Re(F(u)) - \Re(u)$, then we set
$x = F_0^p(\widetilde{z})$ and $y = u$ and the lemma follows.

Otherwise, we set $l_1 := p < l$ and
$u_1 := F^p(\overline{z}) \in \Orb(z,F) \setminus \IR$ and from
\eqref{eq:displ} and Lemma~\ref{lem:FF+k}(a) we obtain
\begin{align*}
u - \Re(u) &> F_0^p(\widetilde{z}) = u_1 - \Re(u_1) \text{ and}\\
\Re(F(u)) - \Re(u)
 &= \Re(F(F_0^p(\widetilde{z})))
  = \Re(F(F_0^p(\widetilde{z}) + \Re(u_1))) - \Re(u_1) \\
 &= \Re(F(u_1)) - \Re(u_1).
\end{align*}
As in \eqref{eq:p}, the first of these inequalities implies
that there exists
$p_1\in\{0,\ldots, p-1\}$ such that
\[
  F_0^{p_1}(\widetilde{z}) <
  u_1-\Re(u_1) \le
  F_0^{p_1+1}(\widetilde{z}).
\]
If $l_i = p = 0$ then
$u_1 = \overline{z},$
$\widetilde{z} = u_1 - \Re(u_1)$ and, hence,
$\Re(F(\widetilde{z})) = \Re(F(u_1)) - \Re(u_1).$
This contradicts the fact that
$\Re(F(\widetilde{z})) \ne \Re(F(u)) - \Re(u).$
Consequently, $l_1 = p > 0$ and
$u_1 \notin F(\IR)$ according to \eqref{eq:Fi(z)}.
As above, this implies that $u_1 - \Re(u_1) > \widetilde{z}.$
So we can replace $u$ by $u_1$ and $l$ by $l_1$ without modifying the
current assumptions and we can repeat iteratively the above process to
obtain a sequence
$0 < l_m < l_{m-1} < \dots < l_1 < l$ with $1 \le m < l$ and
$p_m \in \{0,1,2,\dots,l_m-1\}$ such that
\begin{itemize}
\item $u_i := F^{l_i}(\overline{z}) \in \Orb(z,F) \setminus \IR$ and
      $\Re(F(u)) - \Re(u) = \Re(F(u_i)) - \Re(u_i)$
      for $i=1,2,\dots,m;$

\item $u - \Re(u) > u_1 - \Re(u_1) > u_2 - \Re(u_2) > \dots
                  > u_m - \Re(u_m) > \widetilde{z};$

\item $F_0^{p_m}(\widetilde{z}) < u_m - \Re(u_m) \le
                 F_0^{p_m+1}(\widetilde{z})$ and
      $\Re(F(F_0^{p_m}(\widetilde{z}))) \ne \Re(F(u_m)) - \Re(u_m)$.
\end{itemize}
Notice that such a sequence exists because we are in the case when
$\Re(F(\widetilde{z})) \ne \Re(F(u)) - \Re(u).$
Then the lemma follows by taking
$x = F_0^{p_m}(\widetilde{z})$ and $y = u_m$.
\end{proof}

\begin{proof}[Proof of Lemma~\ref{lem:super}]
If $u=0$, then $u\in F(\IR)$ and we take $y=u$.
From now on, we assume that $u\in \Bo_0$.
By Lemma~\ref{lem:PotImp:ToutVaBien}, we know that there exists
$y \in \Orb(z,F) \cap \Bo_0$ satisfying
$y \le u$ and $\Re(F(y)) = \Re(F(u))$
and such that,
\begin{enumerate}
\item  either $y \in F(\IR)$,
\item or there exists $x\in B_0$ such that $x < y \le F_0(x)$ and
$m:=\Re(F(x)) \ne \Re(F(y))$.
\end{enumerate}

In case~(a), the lemma holds. So, assume that there exists a point $x$
as in case~(b). Observe that $m \in \Z$ and $F(y) \notin B_m$ because
$F_0(x) \notin \R$. So, by Lemma~\ref{twoarrowscrossing}, the lemma
holds unless $F(0) \in (x+m, \max B_m]$.

Assume that $F(0) \in (x+m, \max B_m]$.
In view of Lemma~\ref{twoarrowscrossingLargeOrbits}, we have
again that $\Per(F) = \IN$ unless $|m - \Re(F(y))| < 1$. Finally,
if  $|m - \Re(F(y))| <1$, then $F(y)\in (m-1,m+1)\setminus\{m\}$
because $F(y) \notin B_m$.
This ends the proof of the lemma.
\end{proof}


\subsection{Proof of Theorem~\ref{theo:0inInterior}}

The proof of Theorem~\ref{theo:0inInterior} is quite long. In the rest
of the section, we are going to assume that $\Int(\RotR(F))$ contains
$0$ (if it contains another integer $m$, we come down to $0$ by
considering the map $F-m$). The first step consists in exhibiting a
particular configuration of points. Then we shall split the proof into
several cases, depending of the positions of these points.

\subsubsection{A particular configuration of points}

We proceed along the lines of the proof of \cite[Lemma~3.9.1]{ALM}.
We first introduce some notation.

Since $0 \in \Int(\RotR(F))$,
there exist $a,b \in \Int(\RotR(F))$ such that $a < 0 < b$, and
there exist $x_a,x_b \in \R$ such that
$\rhos(x_a) = a < 0 < b = \rhos(x_b).$
We may assume that $x_b < x_a$ (by taking $x_b - k$ instead of $x_b$
with $k \in \Z$ appropriate).

\begin{remark}\label{rem:boundedcompactnumber}
Since $\rhos(x_a)<0$ (resp. $\rhos(x_b)>0$), the sequence
$\left(\Re(F^n(x_a))\right)_{n\ge 0}$ tends to $-\infty$
(resp. $\left(\Re(F^n(x_b))\right)_{n\ge 0}$ tends to $+\infty$).
Thus the orbits of both points have a finite number of
elements in each compact subset of $S$.
\end{remark}

Now we define
\begin{align*}
\overline{M} &:=
  \set{F^k(x_b)}{k \ge 0 \text{ and } \Re(F^l(x_b)) > \Re(F^k(x_b))
    \text{ for every } l>k},\text{ and}\\
\underline{M} &:=
  \set{F^k(x_a)}{k \ge 0 \text{ and } \Re(F^l(x_a)) < \Re(F^k(x_a))
    \text{ for every } l>k}.
\end{align*}
Observe that
$\overline{M} \subset \Orb(x_b,F)$ and
$\underline{M} \subset \Orb(x_a,F).$
Hence,
$\overline{M} \cap \underline{M} = \emptyset$
because $x_a$ and $x_b$ have different rotation numbers.

The next lemma summarizes the properties of
$\overline{M}$ and $\underline{M}.$

\begin{lemma}\label{lem:M&M}
The following statements hold for the sets
$\overline{M}$ and $\underline{M}.$
\begin{enumerate}
\item For every $x \in \R,$
$\Card(\Re^{-1}(x) \cap \overline{M}) \le 1$
and
$\Card(\Re^{-1}(x) \cap \underline{M}) \le 1.$

\item Let $L\in\IR$.
For every $w\in \Orb(x_b,F)$ there exists
a point $\overline{x}\in\overline{M}$ such that
$\Re(\overline{x})=\min (\Re(\Orb(w,F))\cap [L,+\infty))$ and
for every $w'\in \Orb(x_a,F)$,
there exists $\underline{x}\in\underline{M}$ such that
$\Re(\underline{x})=\max (\Re(Orb(w',F))\cap (-\infty,L])$.

\item $\min \Re(\overline{M}) = \min \Re(\Orb(x_b,F)) \le x_b,$
and\\
$\max \Re(\underline{M}) = \max \Re(\Orb(x_a,F)) \ge x_a.$

\item $\sup \Re(\overline{M}) = +\infty$
and
$\inf \Re(\underline{M}) = -\infty.$

\item If $x\in \overline M$,
there exists $x'\in \overline M\cap\Orb(x,F)$ such that
$\Re(x)<\Re(x')\le \Re(F(x))$.
The same holds with reverse inequalities with $x,x'\in \underline M$.

\item For any $x_0\in\IR$ and $x\in \overline M$ with
$\Re(x)\le x_0$, there exists $x'\in\overline M$ such
that $\Re(x)\le \Re(x')\le x_0<\Re(F(x'))$.
If $\Re(x')=\Re(x)$ then $x'=x$.
The same holds with reverse inequalities if $x\in \underline M$.
\end{enumerate}
\end{lemma}

\begin{proof}
We prove the lemma for the set $\overline{M}.$
The proofs for the set $\underline{M}$ follow similarly.

Let $F^k(x_b),F^l(x_b) \in \overline{M}$ with $k < l.$
From the definition of the set $\overline{M}$, it follows that
$\Re(F^l(x_b)) > \Re(F^k(x_b)).$
So, (a) holds.

We have $\lim_{n\to+\infty}\Re(F^n(x_b))=+\infty$
(Remark~\ref{rem:boundedcompactnumber}) and thus,
for every $L \in \R$ and every $w\in\Orb(x_b,F)$,
the set $\Re(\Orb(w,F)) \cap [L,+\infty)$ contains infinitely many
elements.
We can define $\xi:=\min (\Re(\Orb(w,F))\cap [L,+\infty)).$
The set $\Re^{-1}(\xi) \cap \Orb(w,F)$ is finite
by Remark~\ref{rem:boundedcompactnumber}.
Thus we can define $i:=\max\set{n\ge 0}{\Re(F^n(w))=\xi}$.
It follows that, for every $j > i$,
$F^j(w) \notin \Re^{-1}(\xi)$
and hence, by the minimality of $\xi$,
$\Re(F^j(w)) > \xi = \Re(F^i(w)).$
So $F^i(w) \in \overline{M}.$
This proves (b) with $\overline{x}=F^i(w)$.
To prove (c) we repeat the proof of (b) by choosing
$w=x_b$ and $L \le \min \Re(\Orb(x_b,F)).$
Then, we obtain $\xi=\min (\Re(\Orb(x_b,F))$
by the definition of $\xi$.
Since $\overline{M}\subset \Orb(x_b,F)$ and $\xi\in \Re(\overline{M})$,
this implies that $\min \Re(\overline{M})=\min \Re(\Orb(x_b,F))$.
Moreover, it is obvious that $ \min \Re(\Orb(x_b,F))\le x_b$,
and thus we obtain (c).
To prove (d), it is enough to use (b) with $L$ tending to $+\infty$.

Suppose that $x\in\overline{M}$.
Consider the set $A=\set{F^i(x)}{i>0}$.
Then $\min A >x$ because $x\in \overline M$.
Applying (b) with $w=x$ and $L=\min A\in \Orb(x,F)$,
we see that there exists
$x'\in\overline M$ such that $\Re(x')=\min \Re(A)$.
By definition of $A$, we have $\Re(x')\le \Re(F(x))$ and this gives (e).

Let $x_0\in\IR$ and let $x\in\overline M$ be such that $\Re(x)\le x_0$.
The set $\Re(\overline M)\cap (-\infty, x_0]$ is non-empty because it
contains $\Re(x)$.
Thus there exists $x'\in\overline M$ such that $\Re(x')$ is
equal to the maximum of this set.
Clearly, $\Re(x)\le \Re(x')\le x_0$.
Suppose that $\Re(F(x'))\le x_0$ and consider the set
$A=\set{F^i(x')}{i>0}$.
Then $\min \Re(A)\le x_0$ and
there exists $x''\in \overline M$ with
$\Re(x'')= \min (\Re(A))$ by (b).
By the definitions of $A$ and $\overline M$, we have
$\min \Re(A)>x'$.
Thus the existence of $x''$ contradicts the definition of $x'$, and
hence $\Re(F(x'))>x_0$.
If $x'\neq x$, then $x'=F^i(x)$ for some $i>0$, and
thus $\Re(x')>\Re(x)$ by definition of $\overline M$. This proves (f).
\end{proof}

Lemma~\ref{lem:M&M}(c) states that
$\min \Re(\overline{M}) \le x_b < x_a \le \max \Re(\underline{M}).$
Consequently, by Lemma~\ref{lem:M&M}(d), there exist points
$z \in \overline{M}$ and $t\in \underline{M}$
such that
$\Re(z) < \Re(t)$
and there are no points of
$\Re(\overline{M} \cup \underline{M})$
in the interval $(\Re(z),\Re(t))$.
By Lemma~\ref{lem:M&M}(b), the inequality $\Re(F(z))<\Re(t)$
(resp. $\Re(F(t))>\Re(z)$) would contradict the definition of $z,t$.
Hence $\Re(F(t))\le \Re(z) < \Re(t) \le \Re(F(z))$.
Let $z'\in \overline M$ (resp. $t'\in\underline M$) be given by
Lemma~\ref{lem:M&M}(e) for $x=z$ (resp. $x=t$). The summary of the
properties of $z,t, z', t'$ is then:
\begin{equation}\label{eq:all-inequalities}
\begin{split}
\Re(F(t)) \le \Re(t') & \le \Re(z) < \Re(t) \le \Re(z') \le \Re(F(z))
              \text{ and}\\
\Re(F(t')) & < \Re(t') < \Re(z') < \Re(F(z')).
\end{split}
\end{equation}

We shall keep the notations $z,z',t,t'$ in the whole section.
Moreover, without loss of generality, we assume that $\Re(t) \in[0,1).$
The points $z$ and $t$ can have the following respective positions:
\begin{enumerate}[(A)]
\item $\Re(t)-\Re(z)\ge 1$,
\item $z,t\in\IR$ and $t-z<1$,
\item $z \in \Bo_0$ and $t \in (0,1),$
\item $t \in \Bo_0$ and $z \in (-1,0).$
\end{enumerate}
In the next three subsections, we shall consider
Cases~(A), (B) and (C) respectively.
Case~(D) follows symmetrically from Case~(C).

Before dealing with these three cases, we state some lemmas which
imply the existence of all periods {\modi}, except perhaps 1, when the
points $t,t',z,z'$ defined above and $F(0)$ satisfy some simple
conditions.

\begin{lemma}\label{lem:F(t)<t-1}
Suppose that $t\in\IR$ and $\Re(F(t))\le t-1$.
If either $z'\in\IR$ or $\Re(F(0))\ge 0$, then $\Per(F)=\IN$.
\end{lemma}

\begin{proof}
If $z'\in\IR$, we have $z'<\Re(F(z'))$ by \eqref{eq:all-inequalities}.
Let $x$ be the point in $z'+\IZ$ such that $t< x<t+1$
(the case $x=t$ is not possible because $x$ and $t$ have different
rotation numbers). By Lemma~\ref{lem:FF+k}(a) we also have $x <
\Re(F(x)).$

When $\Re(F(0))\ge 0$ we set $x=1$ and, as above, $x < \Re(F(x)).$
Since $t\in \R$, $0 \le t < 1$.
If $t=0$ then, $0 \le \Re(F(t)) \le -1;$
a contradiction.
Hence, as in the previous case, $t < x < t+1.$

Thus the interval $I=[t,x]$ is of length less than 1 and we have
$I\signedcover{F}[t-1,x]$ and hence
$I\signedcover{F} I \cup (I-1)$.
Then $\Per(F)=\IN$ by Corollary~\ref{cory:+horseshoeFF-1}.
\end{proof}

\begin{lemma}\label{lem:F(0)>t}
Suppose that $z\in B_0$, $t,t'\in\IR$ and $\Re(F(0))\ge t$.
Then $\Per(F)=\IN$.
\end{lemma}

\begin{proof}
The fact that $z\in B_0$ and \eqref{eq:all-inequalities} imply
that $t'\le 0=\Re(z)<t$.
Let $t''\in t'+\IZ$ be such that $t''\in (-1,0]$.
Necessarily, $t'\le t''$.
Using \eqref{eq:all-inequalities}, we obtain
$F([t'',0])\supset [t'',t]=[t'',0]\cup [0,t]$ and
$F([0,t])\supset [t',t]\supset [t'',t]$.
Since $[t'',0]$ and $[0,t]$ contain no branching points in their
interior, Proposition~\ref{prop:SemiHorseshoe} applies
to the intervals $[t'',0]$ and $[0,t]$, and $\Per(0,F)=\IN$.
This clearly implies that $\Per(F)=\IN$.
\end{proof}

\begin{lemma}\label{lem:bigF(0)}
Suppose that $z\in B_0$, $t\in\IR$  and $|\Re(F(0))|\ge 1$.
Then $\Per(F)\supset\IN\setminus\{1\}$.
\end{lemma}

\begin{proof}
The fact that $z\in B_0$ and \eqref{eq:all-inequalities} imply
that $\Re(F(t))\le 0=\Re(z)<t$.
First we suppose that $\Re(F(0))\ge 1$.
Then $F([0,t])\supset [0,1]$ and $F([t,1])\supset [0,1]$.
Moreover, the two intervals $[0,t]$ and $[t,1]$ contain no branching
point in their interior.
Thus Proposition~\ref{prop:SemiHorseshoe} applies and $\Per(F)=\IN$.

Secondly we suppose that $\Re(F(0))\le -1$.
If, for all $x\in (-\infty, 0)$, $\Re(F(x))<0$, then
Lemma~\ref{lem:F(R)-left} applies (with $x_0=x_b$) and
$\Per(F)\supset\IN\setminus\{1\}$.
Otherwise there exists $x\in (-\infty, 0)$ such that $\Re(F(x))\ge 0$.
Let $b$ be the unique point in $x+\IZ\cap [0,1)$.
Thus $b\ge x+1$ and $\Re(F(b))\ge 1$.
Set $I=[0,b]$.
Then $I\signedcover{F} I\cup (I-1)$ and $\Per(F)=\IN$ by
Corollary~\ref{cory:+horseshoeFF-1}.
\end{proof}

\subsubsection{Case (A): $\Re(t) - \Re(z) > 1$.}
This case is solved in the next lemma.

\begin{lemma}\label{lem:largegap}
Assume that $\Re(t) -\Re(z) \ge 1$.
Then $\Per(F) \supset \N\setminus\{1\}$.
\end{lemma}

\begin{proof}
We assume that $\Re(F(0)) \ge 0$ and we shall use the point $t$.
If $\Re(F(0)) \le 0$, the proof is similar by using the point $z$
instead
of $t$.

By \eqref{eq:all-inequalities} and our assumption,
\[
\Re(F(t)) \le \Re(t') \le \Re(z) \le \Re(t) -1.
\]
So, when $t \in B_0$, then $\Re(F(t))<\Re(t)=0$.
Therefore, $t \ne 0$ because $\Re(F(0)) \ge 0.$
Consequently, either $t \in \Bo_0$, or $t \in (0,1)$.

When $t \in (0,1)$, we have $\Re(t) = t$ and, hence,
$\Re(F(t)) \le t -1$.  Thus, the lemma follows from
Lemma~\ref{lem:F(t)<t-1}.

Assume now that $t \in \Bo_0$ (and, hence, $\Re(F(t)) \le \Re(t)-1=-1$).
By Lemma~\ref{lem:super} (applied with $x_a$ and $t$ instead of $z$ and $u$),
we know that, either $Per(F)=\IN$, and the lemma holds; or
there exists $y \in B_0$ satisfying
$y \le t$ and $\Re(F(y)) = \Re(F(t))$ and such that
\begin{itemize}
\item either $y \in F(\IR)$,
\item or there exists a point
$x\in B_0$ such that $x < y \le F_0(x),$
$F(0) \in (x+m, \max B_m]$ and
$F(y) \in (m-1, m+1)\setminus\{m\}\subset \R$,
where $m:=\Re(F(x)) \in \Z$.
\end{itemize}
In the second case, $m \ge 0$ because $\Re(F(0)) \ge 0.$
But $\Re(F(y))=\Re(F(t))\le -1.$
Hence, $\Re(F(y)) \le m-1$, and thus
the second case is not possible.
Consequently, $y \in F(\IR).$
Since
$\Re(F(y)) \le -1 \le \lfloor  \Re(F(0)) \rfloor -1$, we can use
Lemma~\ref{lem:allperiods-1}. Hence, $\Per(F) \supset
\N\setminus\{1\}$ in this case.
\end{proof}

\subsubsection{Case (B): $z,t\in\IR$ and $t-z<1$}
This case is dealt by the next lemma.

\begin{lemma}\label{lem:sortgapR}
Assume that $t,z \in \R$ and $t -z < 1$.
Then $\Per(F) = \N$.
\end{lemma}

\begin{proof}
We assume that $\Re(F(0)) \ge 0$ and we shall use the point $t$.
If $\Re(F(0)) \le 0$, the proof is similar by using the point $z$
instead
of $t$.

Assume first that $z \ge 0$. From \eqref{eq:all-inequalities},
it follows that
\[
 \Re(F(t)) \le z < t < \Re(z')\le \Re(F(z))\quad\text{and}\quad \Re(z') <
\Re(F(z')).
\]
Let $I=[z,t]$. There is no branching point in the interior of
$I$ since we have assumed
$z \ge 0$.
If $\Re(F(z))<1$, then $z'\in (0,1)$ and we set $J=[t,z']$ (see the
left part of Figure~\ref{fig:caseB1}). If
$\Re(F(z))\ge 1$, we set $J=[t,1]$ (see the right part of
Figure~\ref{fig:caseB1}). In both cases,
there is no branching point in $J$, $F(I)\supset I\cup J$ and
$F(J)\supset I\cup J$. Then Proposition~\ref{prop:SemiHorseshoe}
applies and
$\Per(0,F) = \N$, and hence $\Per(F)=\IN$.

\begin{figure}[htb]
\includegraphics{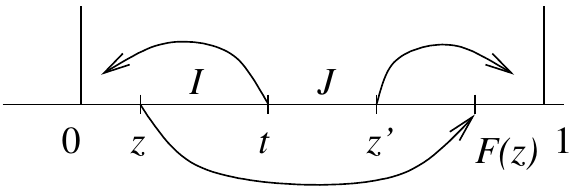}$\qquad$
\includegraphics{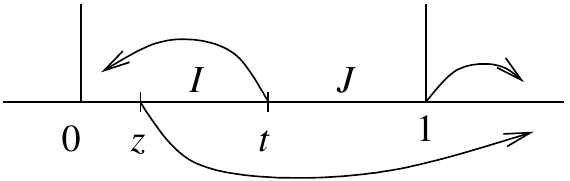}
\caption{When $z\ge 0$, the two possible locations of the intervals
$I,J$,  forming a horseshoe in both cases.}\label{fig:caseB1}
\end{figure}

When $\Re(F(t)) \le t -1$, the lemma follows from
Lemma~\ref{lem:F(t)<t-1}. So, in the rest of the
proof we can we assume that $t-1 < \Re(F(t)) \le z < 0$.
From
\eqref{eq:all-inequalities}, it follows that
\[
 t-1 < \Re(F(t)) \le t' < z < 0, \
 \Re(F(t')) < t'\quad
 \text{and}\quad
 \Re(F(z)) \ge t.
\]
This configuration is depicted in Figure~\ref{fig:caseB2}. Then
\begin{align*}
 F([t', z]) & \supset [t', t] \supset [t',z] \cup [0,t],
              \text{ and}\\
 F([0,t]) & \supset [t',0] \supset [t',z].
\end{align*}

\begin{figure}[htb]
\centerline{\includegraphics{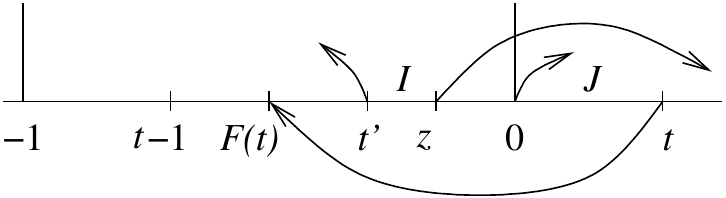}}
\caption{When $t-1 < \Re(F(t)) \le z < 0$, the intervals $I=[t',z']$ and
$J=[0,t]$ form a horseshoe.}\label{fig:caseB2}
\end{figure}

Since the intervals $(t',z)$ and $(0,t)$ contain no branching points,
Proposition~\ref{prop:SemiHorseshoe} applies and $\Per(F)=\IN$.
\end{proof}

\subsubsection{Case (C): $z\in \Bo_0$ and $t\in (0,1)$}

We want to show that, in this situation,
either $\Per(F)\supset\IN\setminus \{1\}$ or
$\Per(F)\supset\IN\setminus \{2\}$. This is the most difficult case.
To deal with it we need some additional points.

Lemma~\ref{lem:PotImp:ToutVaBien} applied with
$z\in\Orb(x_b) \setminus\IR$ instead of $u\in\Orb(z) \setminus\IR$
gives a point $y$ such that $y_0:=y-\Re(y) \in \Bo_0,$
$\Re(F(y_0)) = \Re(F(y)) - \Re(y) = \Re(F(z))$ and, either
\begin{equation}\label{eq:Cx}
\begin{split}
& y \in F(\IR),\text{ or}\\
& \text{$\exists$ $x\in B_0$ such that
      $x< y_0\le F_0(x)$ and $\Re(F(x))\neq \Re(F(y_0))$.}
\end{split}
\end{equation}
Observe that, since $F$ has degree one, $F(\IR)$ is periodic {\modi} and,
hence, $y \in F(\IR)$ implies $y_0 \in F(\IR).$
Also, $z\in \Bo_0$ implies
$\Re(F(y_0)) = \Re(F(z)) > \Re(z) = 0$
by \eqref{eq:all-inequalities}.

Let $a\in [0,1)$ be such that $F_0(a)=\max (F(\IR)\cap B_0)$, and let
$q\in\IZ$ be such that $F(a)\in B_q$. In the rest of this subsection, we
shall keep the notations $y_0,a, q$ to refer to these objects.

We are going to consider three subcases, depending on the positions of
$y_0$ and $t'$:
\begin{enumerate}[(C1)]
\item $y_0\not\in F(\IR)$,
\item $y_0\in F(\IR)$ and $t'\in B_0$,
\item $y_0\in F(\IR)$ and $t'\not\in B_0$.
\end{enumerate}
Cases~(C1), (C2) and (C3) are respectively proved in Lemmas
\ref{lem:caseC1},
\ref{lem:caseC2} and \ref{lem:caseC3}. Altogether, they give Case~(C).

\begin{lemma}\label{lem:caseC1}
If $y_0\notin F(\IR)$ then, either $\Per(F)\supset \IN\setminus\{2\}$
or $\Per(F)\supset \IN\setminus\{1\}$.
\end{lemma}

We first state a part of the proof as a separate lemma
because it will be used again in Case~(C2).

\begin{lemma}\label{lem:yinB}
Suppose that there exist points $w,x,y\in B_0$ and $m\in\IZ$ such that
$|\Re(F(w))|<1$, $\Re(F(w))=\Re(F(y))$,
$F(x)\in B_m$, $x<y\le F_0(x)$, $w\in\overline M$ (resp.
$w\in\underline M$)
and $m\le 0$ (resp. $m\ge 0$). Then $\Per(F)\supset
\IN\setminus\{2\}$.
\end{lemma}

\begin{proof}
We prove the lemma in the case $w\in\overline M.$ The other one is
symmetric. According to Lemma~\ref{lem:M&M}(e), there is a point
$w'\in \overline M$ such that $\Re(w)<\Re(w')\le \Re(F(w))$. Since $w\in B_0$ and
$|\Re(F(w))|<1$, the point $w'$ belongs to $(0,1)$. Moreover,
$\Re(F(w'))>w'$ because $w'\in \overline M$.
Let $I=\chull{w',x}$, endowed with the order for which $\min I=w'$, and
$J=[x,y]\subset B_0$ (with the order of $B_0$); see
Figure~\ref{fig:C1}.
\begin{figure}[htb]
\centerline{\includegraphics{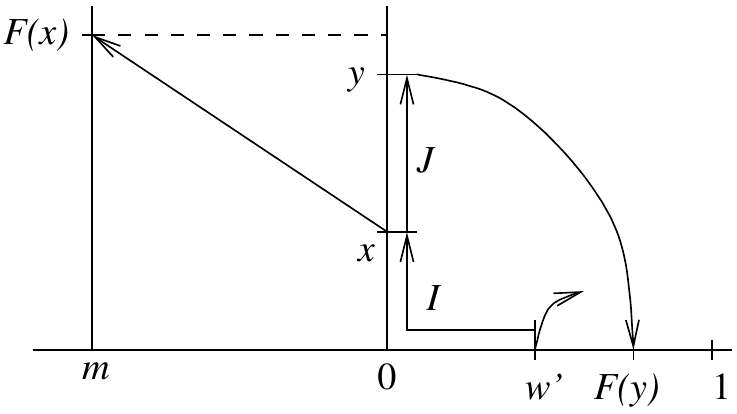}}
\caption{Intervals $I$ and $J$, with arrows indicating their order.
Though not needed in the proof, it can be noticed that the assumptions
imply $\Re(F(w))\in (0,1)$, and hence
$F(y)=F(w)\in (0,1)$.}\label{fig:C1}
\end{figure}
Then $I$ positively covers $I+m$ and $J+m$,
and $J$ negatively covers $I+m$ and $J+m$.
Moreover, $(I+\IZ)\cap (J+\IZ)=\{x\}+\IZ$, and $F(x)\notin I+\IZ$.
Thus Lemma~\ref{lem:+-loop} applies and gives
$\Per(F)\supset \IN\setminus\{2\}$.
\end{proof}

\begin{proof}[Proof of Lemma~\ref{lem:caseC1}]
Since $y_0\notin F(\IR)$,
there exists $x\in B_0$ such that $x< y_0\le F_0(x)$
and $\Re(F(x))\neq \Re(F(y_0))$ by \eqref{eq:Cx}.
Set $m:=\Re(F(x))\in\IZ$ (thus, $F(x) \in B_m$).
If $|\Re(F(y_0))-m| \ge 1$, Lemma~\ref{twoarrowscrossingLargeOrbits}
applies and $\Per(F)=\IN$.
Since $\Re(F(y_0)) > 0,$ the condition $|\Re(F(y_0))-m| \ge 1$ is satisfied,
in particular, when $m \le -1$ or $F(y_0) \in B$.
On the other hand, if $\Re(F(0))\ge 1$, then
$\Per(F)\supset\IN\setminus\{1\}$ by Lemma~\ref{lem:bigF(0)}.
So, in the rest of the proof we can assume that $F(y_0) \notin B,$
$m \ge 0$ and $\Re(F(0)) < 1.$
If $m \ge 1,$ Lemma~\ref{twoarrowscrossing} gives $\Per(F)=\IN.$
Therefore, we are left with the case $m=0$, $F(0)\in (x,\max B_0]$ and
$\Re(F(y_0))<m+1=1$. Then $\Re(F(z))=\Re(F(y_0))\in (0,1)$,
and finally
Lemma~\ref{lem:yinB}, applied to $w=z\in \Bo_0$, $x$, $y=y_0$, gives
$\Per(F)\supset \IN\setminus\{2\}$.
\end{proof}

Now we study Case~(C2).

\begin{lemma}\label{lem:caseC2}
Assume that $y_0\in F(\IR)$ and $t'\in B_0$. Then, either
$\Per(F)\supset\IN\setminus\{2\}$, or
$\Per(F)\supset\IN\setminus\{1\}$.
\end{lemma}

Again, we state a part of the proof as a lemma,
in order to use it again in Case~(C3).

\begin{lemma}\label{lem:3M}
If there exist $z_0, t_1, t_2\in \IR$ such that $0\le t_1\le z_0\le
t_2\le 1$,
$z_0\in \overline M$ and $t_1,t_2\in \underline M+\IZ$, then
$\Per(0,F)=\IN$.
\end{lemma}

\begin{proof}
Let $k_1,k_2\in\IZ$ be such that $t_1\in\underline M+k_1$ and $t_2\in
\underline M+k_2$. The points $t_1, t_2$ cannot be equal to $z_0$
because $\rhos(z_0)>0$ and $\rhos(t_1)=\rhos(t_2)<0$.
According to Lemma~\ref{lem:M&M}(f) (applied with $x_0=z_0-k_2$ and
$x=t_2-k_2$), there exists $t_2'\in \underline M+k_2$ such that
$\Re(F(t_2'))<z_0\le \Re(t_2')\le t_2$.
We choose this point so that $\Re(t_2')$ is minimal.
Since $0<z_0<t_2\le 1$ then, either $t_2'$ is
in $(0,1)$, or $\Re(t_2')=1=t_2$, in which case $t_2'=t_2$. Thus $t_2'$
is in $(0,1]\subset \IR$.
Similarly, there exists $z_0'\in (0,1) \cap \overline{M}$ such
that $z_0\le z_0'\le t_2' <\Re(F(z_0'))$.
Since $z_0' \in \overline{M}$ and $t_2'\in \underline{M}+k_2,$
$z'_0 < t'_2$ because they have different rotation numbers.
By Lemma~\ref{lem:M&M}(e), there exists $t_2'' \in \underline{M}+k_2$
such that $\Re(F(t_2'))\le t_2''< t_2'$.
Moreover, $t_2''<z_0$ by the minimality of $\Re(t_2')$.
We set $t_1':=\max (t_1, t_2'')$.
Then $t_1'\in(\underline M+k_1)\cup (\underline M+k_2)$
and $\max(t_1, \Re(F(t_2'))\le t_1' < z_0$.
\begin{figure}[htb]
\centerline{\includegraphics{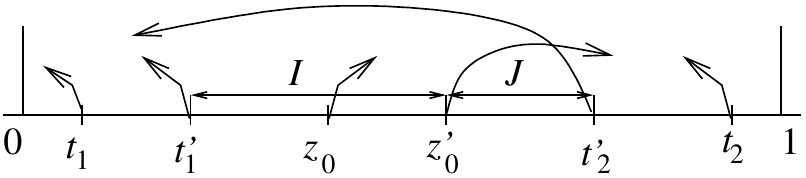}}
\caption{Positions of the points in Lemma~\ref{lem:3M}; the intervals
$I=[t_1',z_0']$ and
$J=[z_0,t_2']$ form a horseshoe.}\label{fig:C2}
\end{figure}
Thus $t_1'\in [0,1)$
and $\Re(F(t_1'))\le t_1'$ because $t_1'\in\underline M+\IZ$. Then the
points have the following positions (see Figure~\ref{fig:C2}):
\[
   \max(\Re(F(t_2'), \Re(F(t_1')) \le t_1'< z_0' < t_2' < \Re(F(z_0')).
\]
So, Proposition~\ref{prop:SemiHorseshoe} with
$[t_1',z_0']$ and $[z_0', t_2']$ applies.
Thus, $\Per(0,F)=\IN$.
\end{proof}

\begin{proof}[Proof of Lemma~\ref{lem:caseC2}]
If $\Re(F(0))\notin (-1,1),$ the result follows from
Lemma~\ref{lem:bigF(0)}.
So, we can assume that  $\Re(F(0))\in (-1,1).$

We apply Lemma~\ref{lem:super} with $z=x_a$ and
$u=t'\in\Orb(x_a)\cap B_0$, to obtain a point $y\in B_0$
such that $\Re(F(y))=\Re(F(t'))$, and:
\begin{enumerate}[(i)]
\item either $\Per(F)=\IN$ (and we are over),
\item or $y\in F(\IR)$,
\item or there exists $x'\in B_0$ such that $x'<y\le F_0(x')$ and $F(0)\in
B_{m'}$, where $m':=\Re(F(x'))\in\IZ$ and $F(y) \in (m'-1,m'+1)\setminus\{m'\}.$
\end{enumerate}
In the last case, necessarily $m'=0$
because we have assumed $\Re(F(0))\in (-1,1)$.
Hence, $\Re(F(t')) = \Re(F(y)) \in(-1,1),$ and we can
apply Lemma~\ref{lem:yinB} with
$w=t'$, $x'$, $y$ to obtain $\Per(F)\supset\IN\setminus\{2\}$.

From now on, we suppose that we are in case~(ii), that is, $y \in
F(\IR)$.
Since we have assumed that $y_0\in F(\IR)$, we have $F_0(a)\ge
\max(y,y_0)$ (in $B_0$). Let $J = \chull{y,y_0};$
this interval is included in $B_0$ and thus
contains no branching in its interior.
If, for every $x \in (-\infty,0)$, $\Re(F(x))<0$, then
Lemma~\ref{lem:F(R)-left} applies (with $x_0=x_b$) and
$\Per(F)\supset \IN\setminus\{1\}$. Otherwise, there
exists a point $x\in (-\infty,0)$ such that $\Re(F(x))\ge 0$.
Let $b$ be the unique point in $(x+\IZ)\cap
[0,1)$. Then $b\ge x+1$ and $\Re(F(b))\ge 1$.
Since $t,b\in [0,1]$ and $\Re(F(t))\le \Re(z)=0$ by
\eqref{eq:all-inequalities},
we have $F([0,1])\supset [0,1]$.
Moreover, since $a\in [0,1]$ and $F(a)\in B_q$, we have
$F([0,1])\supset [0, F_0(a)]+q$ because,
either $F(0)\notin B_q$ or $F(1)\notin B_q$.
Thus $F([0,1])\supset J+q$.
On the other hand, $F(J)\supset [\Re(F(y)),\Re(F(y_0))]$
and $\Re(F(y))=\Re(F(t'))\le \Re(z)=0$. Thus, if
\begin{equation}\label{eq:Fy0}
\Re(F(y_0))\ge 1,
\end{equation}
then $F(J)\supset [0,1]$
and we have the situation and the coverings represented in
Figure~\ref{fig:C22}.
Then $\Per(F)=\IN$ by Lemma~\ref{lem:SemiHorseshoe-mod1}.
\begin{figure}[htb]
\centerline{\includegraphics{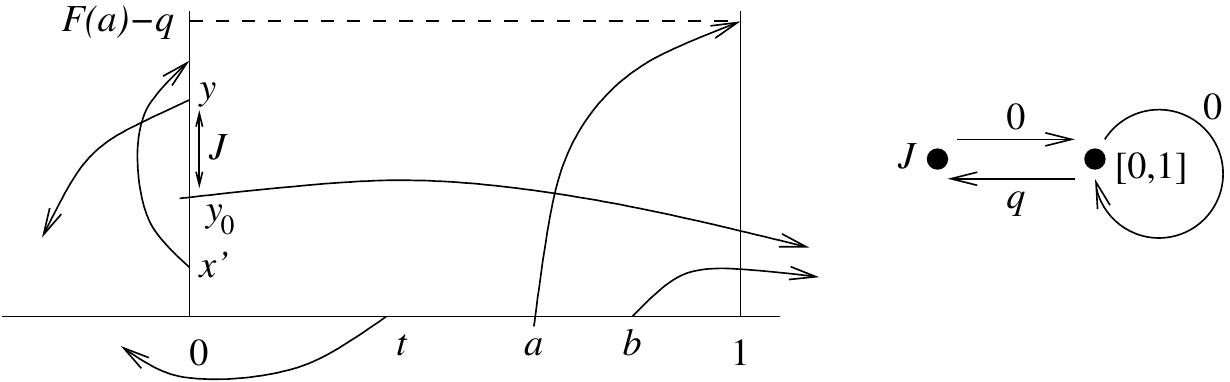}}
\caption{Left side: points $t,a,b$ are in $[0,1]$ but maybe not in
this order;
point $y$ may be below $y_0$ in $B_0$. In all cases, we have the
coverings on the right.}\label{fig:C22}
\end{figure}

From now on, we assume that \eqref{eq:Fy0} does not hold, that is,
$\Re(F(y_0))<1.$
This implies that
$z'\in (0,1)$ and $\Re(F(y_0)) \ge z'$
by \eqref{eq:all-inequalities}
(recall that $\Re(F(y_0))=\Re(F(z))$).
If there exists $t_2\in (\underline M+1)\cap [z',1]$, then
Lemma~\ref{lem:3M} applies (with $z_0=z'$, $t_1=t$ and $t_2$)
and $\Per(F)=\IN$.
So, in the rest of the proof we assume that
\begin{equation}\label{eq:Mz'}
(\underline M+1)\cap [z',1]=\emptyset.
\end{equation}
Lemma~\ref{lem:M&M}(f), applied with $x_0=z'-1$ and $x=t'\in\underline{M}$,
implies that $\Re(F(t'+1))< z'$ (otherwise, there would exist
$t''\in \underline{M}$ such that $z'\le \Re(t'')+1\le \Re(t')+1$,
which would contradict \eqref{eq:Mz'} since $\Re(t')+1\le 1$).
Since $F$ has degree one,
$\Re(F(y)) = \Re(F(t')) < z'-1$ and, hence,
$F(J)\supset [z'-1,z'].$

Now we split the proof of this remaining case into three subcases,
depending on the values of $a$ and~$q$.
\begin{itemize} 
\item If $a\le z'$, we have the situation represented in
Figure~\ref{fig:C23}.
\begin{figure}[htb]
\centerline{\includegraphics[width=\textwidth]{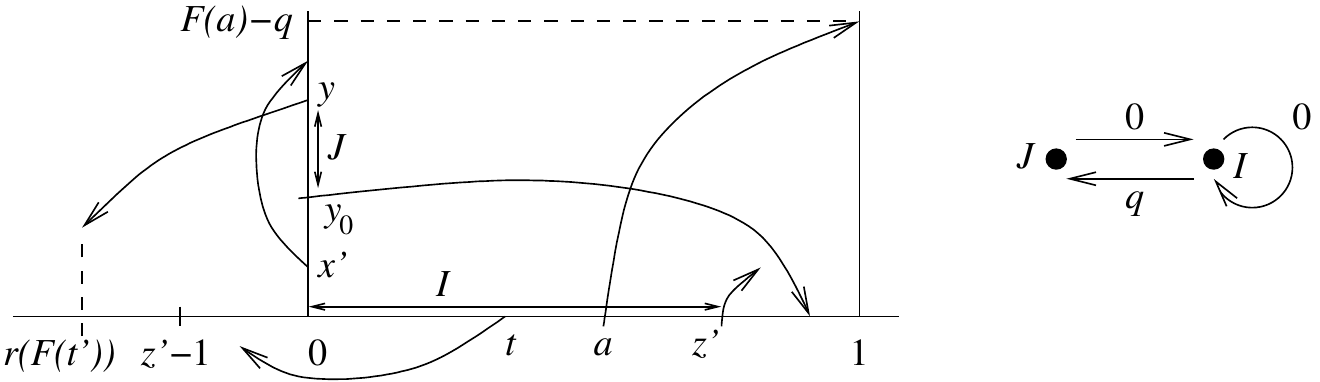}}
\caption{Left side: points $t,a$ are in $[0,z']$ but maybe not in this
order; point $y$ may be below $y_0$ in $B_0$. In all cases, we have the
coverings on the right.}\label{fig:C23}
\end{figure}
We set $I=[0,z']$ and there is no branching point in
$(0,z')$ because $z'\in (0,1)$. The interval $I$ contains
$t,z'$ and $a$, with $\Re(F(t))\le 0$ and $\Re(F(z'))>z'>0$. Either
$F(t)\notin B_q$, or $F(z')\notin B_q$, and thus $F(I)$ contains
$[q,F(a)] \subset B_q$.
Hence $I \arrowto J+q$.
Moreover, $I \arrowto I$ and $J \arrowto I$.
Thus $\Per(F)=\IN$ by Lemma~\ref{lem:SemiHorseshoe-mod1}.

\item Suppose that $a>z'$ and $q\ge 1$.
By Lemma~\ref{lem:M&M}(e), there exists $t''\in \underline M+1$ such that
\begin{equation}\label{eq:t't''}
\Re(F(t'+1))\le \Re(t'')<\Re(t'+1)=1.
\end{equation}
We have $\Re(F(t''))<\Re(t'')$ because $t''\in \underline M+1$.
Moreover, $\Re(t'')<z'$ by \eqref{eq:Mz'}.
We set $\widetilde{t}= \max (\Re(t''),t) \in (0,z')$;
then we have $\Re(F(\widetilde{t})) < \widetilde{t}$
(see Figure~\ref{fig:C24}).
\begin{figure}[htb]
\centerline{\includegraphics[width=\textwidth]{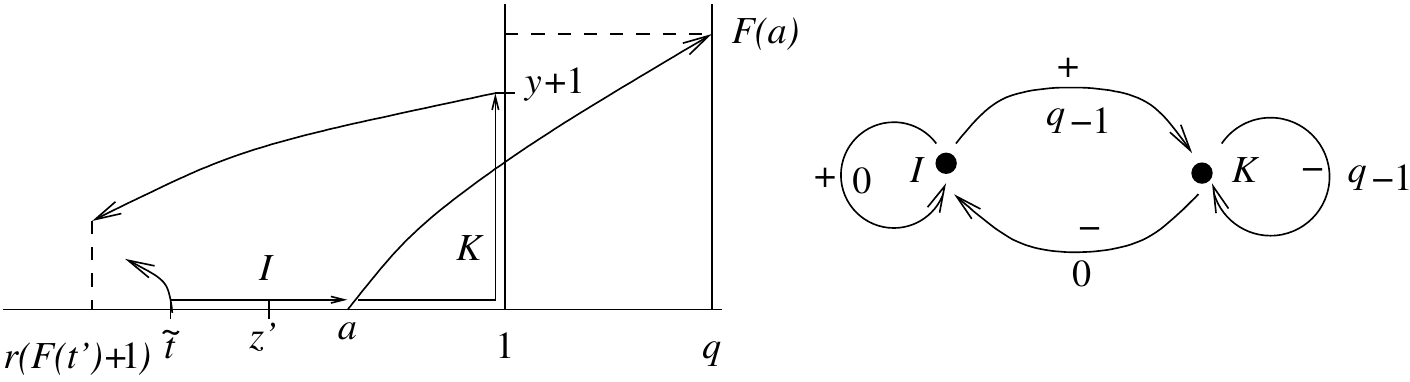}}
\caption{Positions of points and covering graph of
$I,K$.}\label{fig:C24}
\end{figure}
Let $I=[\widetilde{t},a]\subset \IR$ and $K=\chull{a,y+1}$
endowed with the order such that $\min K=a$.
Then $I$ positively covers $I$ and $K+q-1$
(because $F(a)\in B_q$ with $q \ge 1$) and
$K$ negatively covers $I$ and $K+q-1$
(because $q\ge 1$ and $\Re(F(y'))=\Re(F(t'))$
and $\Re(F(t'+1))\le \Re(t'') \le \widetilde{t}$ by \eqref{eq:t't''}).
Moreover, $(I+\IZ)\cap (K+\IZ)=\{a\}+\IZ$, and $F(a)\notin I+\IZ$.
Thus Lemma~\ref{lem:+-loop} applies and gives
$\Per(F)\supset \IN\setminus\{2\}$.

\item Suppose that $a > z'$ and $q \le 0$.
Let $I = \chull{a,b} \subset [0,1)$.
If $0\le b\le t$, then $[b,t]$ and $[t,z']$ form a horseshoe; and
if $t\le b\le a$, then $[t,b]$ and $I$ form a horseshoe
(see Figure~\ref{fig:C25}).
In both cases, Proposition~\ref{prop:SemiHorseshoe} applies
and $\Per(F)=\IN$.
\begin{figure}[htb]
\centerline{\includegraphics{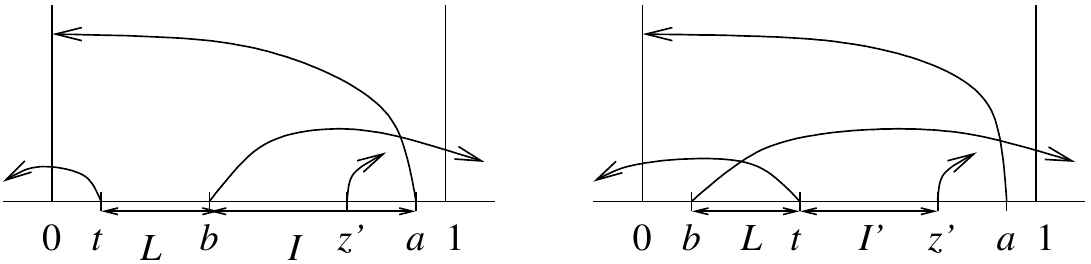}}
\caption{The two possibilities when $b<z'$. In both cases, there is a
horseshoe (either $L,I$ or $L,I'$).}\label{fig:C25}
\end{figure}

It remains to consider the case when $b>a$, which implies that $b> z'$;
see Figure~\ref{fig:C26}.
Then $J$ covers $I-1$ (recall that $\Re(F(y))=\Re(F(t'))\le z'-1$) and $I$
covers
$I$ and $J+q$.
Notice that $I\subset (0,1)$ because $b\ge z'>\Re(z)=0$, which implies that
the sets $I+\IZ$ and $J+\IZ$ are disjoint.
Then $\Per(F)=\IN$ by Lemma~\ref{lem:SemiHorseshoe-mod1}.
\end{itemize} 
We have covered all the possible cases, and thus
Lemma~\ref{lem:caseC2} is proved.
\end{proof}
\begin{figure}[htb]
\centerline{\includegraphics[width=0.95\textwidth]{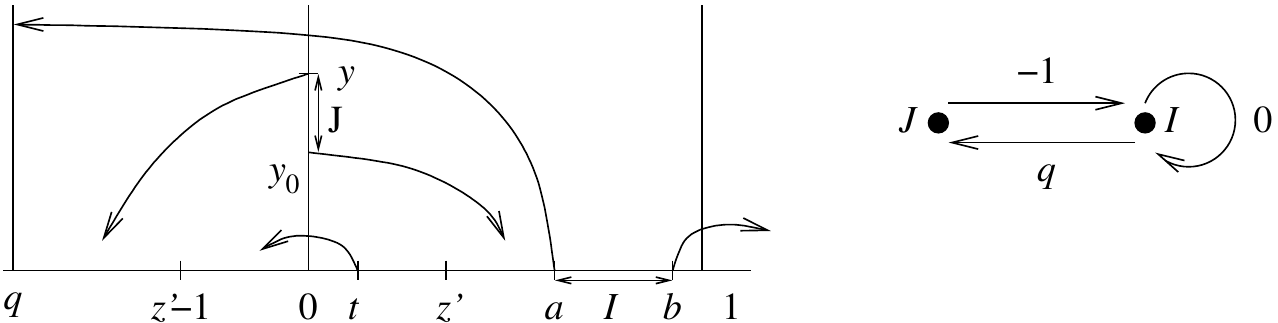}}
\caption{When $b\ge a$.}\label{fig:C26}
\end{figure}

Finally, in the next lemma we study Case~(C3).

\begin{lemma}\label{lem:caseC3}
Suppose that $y_0\in F(\IR)$ and $t'\notin B_0$.
Then, either $\Per(F) \supset\IN\setminus\{1\}$, or
$\Per(F)\supset\IN\setminus\{2\}$.
\end{lemma}

In order to make the proof easier to read, we first deal with a
special
configuration of points.

\begin{lemma}\label{lem:t'+1}
Suppose that $y_0\in F(\IR)$, $t', z'\in\IR$, $\Re(F(0))\le t$ and
$t'+1\le z'\le a<1$. Then $\Per(F)\supset \IN\setminus\{2\}$.
\end{lemma}

\begin{proof}
Let $I:=[t,z']$, $J:=[t',t]$ and $K:=[0,y_0]$. Notice that these three
intervals have disjoint
interiors, and $\Int(J)$ contains the
branching point $0$ (see Figure~\ref{fig:C31}).
\begin{figure}[htb]
\centerline{\includegraphics{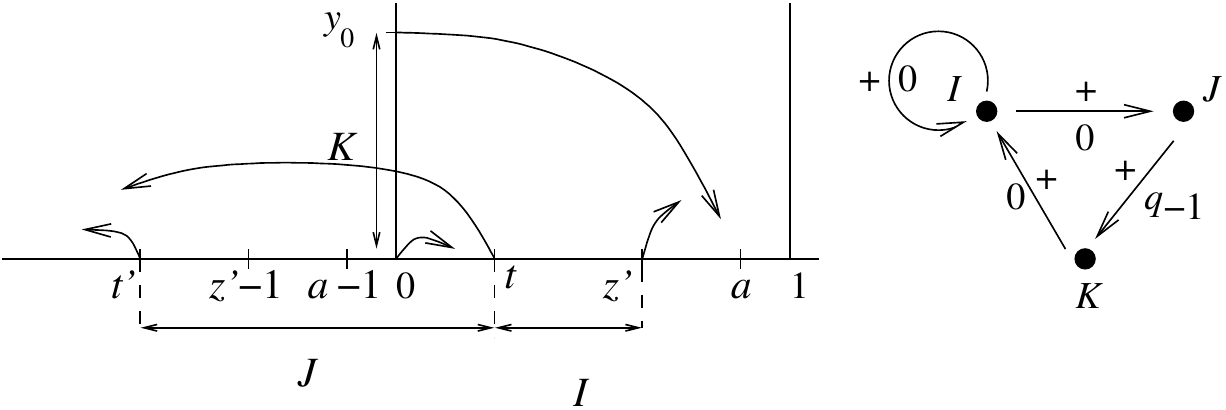}}
\caption{The intervals $I,J,K$ and their covering graph
in Lemma~\ref{lem:t'+1}.}\label{fig:C31}
\end{figure}

It is clear that $I\signedcover{F}I$, $I\signedcover{F}J$
and $K\signedcover{F}I$. By assumption, $t'<z'-1< a-1<0.$
Thus, all these points belong to $J$.
Moreover, either $F(t')\notin B_{q-1}$, or
$F(z'-1)\notin B_{q-1}$ (because $\Re(F(t'))<t'$ and $\Re(F(z'))>z'$).
Hence $J\signedcover{F}K+q-1$. Now, we are going to show that these
coverings imply that $\Per(F)\supset \IN\setminus\{2\}$.
We set
\[
\CC:=I\signedcover{F}I\quad\text{and}\quad
\CC':=I-q+1\signedcover{F}J-q+1\signedcover{F}K\signedcover{F}I.
\]
Proposition~\ref{prop:signedcover}, applied to the loop $\CC$, shows
that there exists a fixed point. We fix $n\ge 3$ and
we consider the chain of coverings
$\CC' \CC^{n-3}$. This gives a loop of length $n$ from $I-q+1$ to $I$.
According to Proposition~\ref{prop:signedcover},
there exists a point $x\in I-q+1$ such that $F^n(x)=x+q-1$,
$F(x)\in J-q+1$, $F^2(x)\in K$ and $F^i(x)\in I$ for all $3\le i\le
n$. It remains to prove that the period {\modi} of $x$ is exactly $n$.
Let $p$ be the period $\modi$ of $x$.
If $p< n$, then  $p\le n-2$ because $p$ divides $n\ge 3$.
Thus $F^2(x)\in K$, $F^{2+p}(x)\in I$ and $F^{2+p}(x)- F^2(x)\in\IZ$.
But this is impossible
because $I\subset (0,1)$, and hence $(I+\IZ)\cap (K+\IZ)=\emptyset$.
This  proves that $p=n$. Therefore, $\Per(F)\supset \IN\setminus\{2\}$.
\end{proof}

\begin{proof}[Proof of Lemma~\ref{lem:caseC3}]
We can assume that  $\Re(F(0))\in (-1,1)$ since, otherwise,
Lemma~\ref{lem:bigF(0)} gives the conclusion.
Then, applying Lemma~\ref{lem:allperiods-1} to
$y_0$ (knowing that $\Re(F(y_0))>0$), we see that,
either $\Per(F)\supset\IN\setminus\{1\}$, or we
are in one of the following cases:
\begin{enumerate}[(I)]
\item $F(0)\in (-1,0)\cup B_0$ and $F(y_0)\in (0,1)$,
\item $F(0)\in (0,1)$ and $F(y_0)\in (0,1)$,
\item $F(0)\in (0,1)$ and $\Re(F(y_0))\in [1,2)$.
\end{enumerate}
Notice that in Cases~(I) and (II), we have $z'\in (0, 1)$ because
$t<\Re(z')\le \Re(F(y_0))<1$.
In addition, we can assume that $\Re(F(t))\ge t-1$, otherwise
Lemma~\ref{lem:F(t)<t-1} gives the result
(using $z'\in\IR$ in Cases~(I) and (II), and $\Re(F(0))\ge 0$ in
Case~(III)).
Recall that $\Re(F(t))\le \Re(t')\le \Re(z)=0$, $t\in (0,1)$ and
$t'\notin B_0$ by assumption. Thus
\[
   -1 < t-1 \le \Re(F(t)) \le \Re(t')) < 0
\]
and both points $t'$ and $F(t)$ belong to $(-1, 0)$.
Now we consider several cases.
\begin{enumerate}
\item If $\Re(F(0))\ge t$, then $\Per(F)=\IN$ by Lemma~\ref{lem:F(0)>t}.

\item Suppose that $a<t$ and $0<F(0)\le t$. If $q\ge 1$, then
we are in the situation depicted in Figure~\ref{fig:C3b1} and we can
apply Proposition~\ref{prop:SemiHorseshoe} to $[a,t]$ and $[t,1]$.
\begin{figure}[htb]
\centerline{\includegraphics{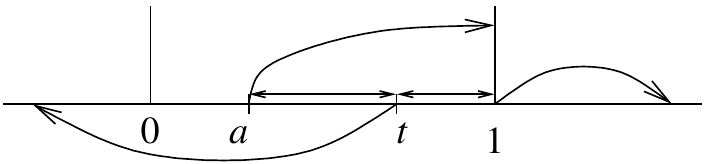}}
\caption{Case (b) with $q\ge 1$ ($q=1$ in the picture): the intervals
$[a,t]$ and $[t,1]$ form a horseshoe.}\label{fig:C3b1}
\end{figure}

Now assume that $q\le 0$, which implies that $a\neq 0$.
Let $I=[a,t]$ and $J=[0,y_0]$. Since $F(1)>1$, there exists $d'\in (t,1)$
such that $F(d')>1$. If  $\Re(z')\ge 1$, we set $d=d'$;
otherwise $z'\in (0,1)$ and  we set $d=z'$.
In both cases, $t<d< 1$ and $\Re(F(d))\ge d$.
Since $\Re(F(t))\le 0<a$, there exists $c\in (t,d)$ such
that $F(c)=a$. Let $K=[c,d]$. Then the three intervals $I,J,K$ contain
no branching point in their interior and they are disjoint $\modi$ (that is,
the sets $I+\IZ, J+\IZ,K+\IZ$ are disjoint). Moreover we have
$F(I)\supset J+q$, $F(J)\supset K$
(because $F(0)\le t$ and $\Re(F(y_0))=\Re(F(z))\ge \Re(z')\ge d$)
and $F(K)\supset I\cup K$ (see Figure~\ref{fig:C3b2}).
\begin{figure}[htb]
\centerline{\includegraphics[width=0.9\textwidth]{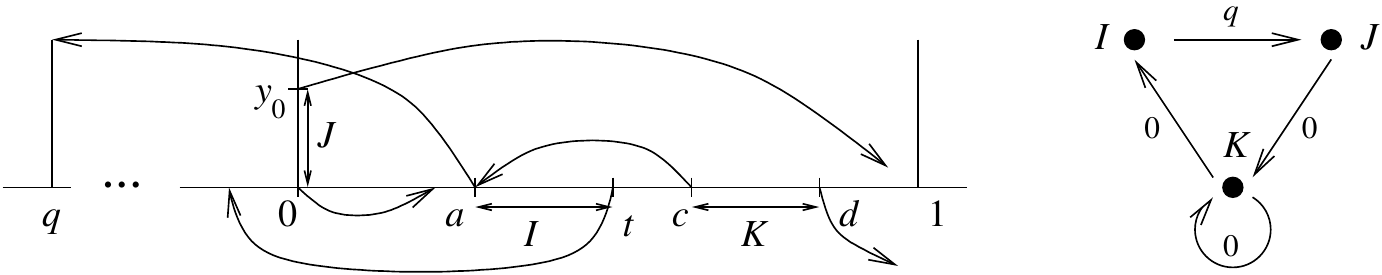}}
\caption{Case (b) with $q\le 0$; on the right: covering graph of
$I,J,K$.}\label{fig:C3b2}
\end{figure}
We define the loops of coverings
\[
  \CC:=K \arrowto K
     \quad \text{and} \quad
  \CC':= K-q \arrowto I-q \arrowto J \arrowto K.
\]
The loop $\CC$ gives a fixed point.
For $n\ge 3$, we consider $\CC' \CC^{n-3}$, which is a loop of length $n$.
According to Proposition~\ref{prop:covering},
there exists a periodic $\modi$ point $x\in K-q$ such that
$F^n(x)=x+q$, $F(x)\in I-q$, $F^2(x)\in J$ and
$F^i(x)\in K$ for all $3\le i\le n$.
It remains to prove that the period {\modi} of $x$ is exactly $n$.
Let $p$ be the period $\modi$ of $x$.
If $p< n$, then  $p\le n-2$ because $p$ divides $n\ge 3$.
Thus $F^2(x)\in J$, $F^{2+p}(x)\in K$ and $F^{2+p}(x)- F^2(x)\in\IZ$.
But this is impossible
because $(J+\IZ)\cap (K+\IZ)=\emptyset$.
This  proves that $p=n$.
Therefore, $\Per(F)\supset\IN\setminus\{2\}$.

\item If $0<F(0)<t\le a$ and $\Re(F(y_0))\ge 1$, we set $I=[t,1]$ and
$J=[0,y_0]\subset B_0$ (see Figure~\ref{fig:C3c}).
We have
$F(I)\supset I$ (because $F(t)<t$ and $F(1)>1$),
$F(I)\supset J=q$ (because $a\in I$ and $F(1)\notin B$),
$F(J)\supset I$ (because $F(0)<t$ and $\Re(F(y_0))\ge 1$ by assumption).
Hence $\Per(F)=\IN$ by Lemma~\ref{lem:SemiHorseshoe-mod1}.
\begin{figure}[htb]
\centerline{\includegraphics[width=0.9\textwidth]{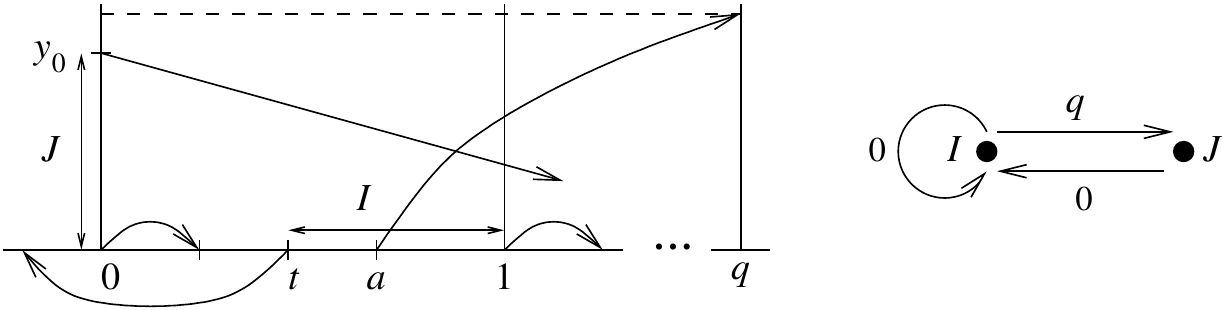}}
\caption{Case (c); on the right: covering graph of
$I,J$.}\label{fig:C3c}
\end{figure}

\item If $z'\in (0,1)$ and $\Re(F(0))\le t\le a\le z'$, we set
$I=[t,z']\subset \IR$ and $J=[0,y_0]\subset B_0$ (see Figure~\ref{fig:C3d}).
Then $F(J) \supset I$ (because $\Re(F(0))\le t$ and $\Re(F(y_0))=\Re(F(z))\ge z'$),
$F(I)\supset I$ (because $\Re(F(t))\le t$ and $\Re(F(z'))\ge z'$)
and $F(I)\supset J+q$ (because $a\in I$ and either $F(t)\notin B_q$
or $F(z')\notin B_q$).
\begin{figure}[htb]
\centerline{\includegraphics{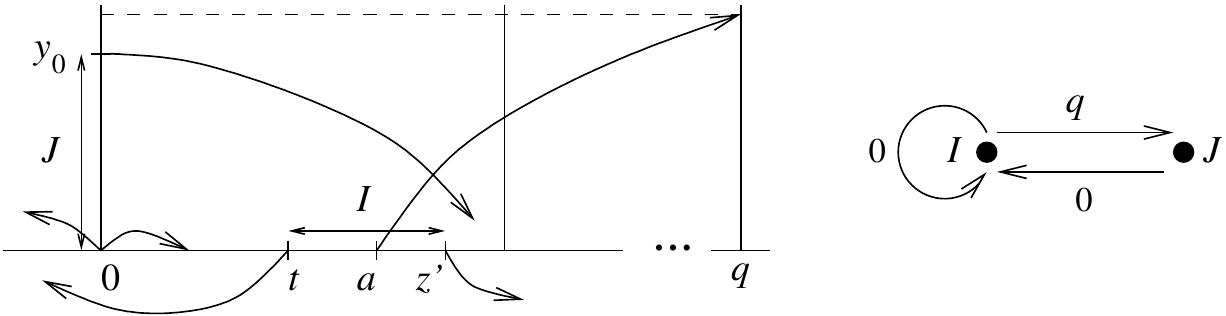}}
\caption{Case (d): the two arrows starting from $0$ mean that it is only
known that $\Re(F(0))\le t$; on the left: covering graphs of
$[0,y_0]$ and $J=[t,z']$.}\label{fig:C3d}
\end{figure}
Hence $\Per(F)=\IN$ by Lemma~\ref{lem:SemiHorseshoe-mod1}.

\item Suppose that $z'\in (0,1)$, $\Re(F(0))\le t$ and $a>z'$.
If $z'\le t'+1$, we apply Lemma~\ref{lem:3M} with $t_1=t$,
$t_2=t'+1$, $z_0=z'$ and we obtain $\Per(F)=\IN$. If $z'\ge t'+1$, we apply
Lemma~\ref{lem:t'+1} and we obtain $\Per(F)\supset\IN\setminus\{2\}$.
\end{enumerate}
Case~(III) is covered by items (a), (b) and (c). Case~(II) is covered
by items (a), (b), (d) and (e), and Case~(I) is covered by items (d)
and (e). This concludes the proof.
\end{proof}

\subsubsection{Conclusion of the proof}

Suppose that $m\in\Int(\RotR(F))$ with $m\in\IZ$. We may assume that
$0\in\Int(\RotR(F))$ by considering $F-m$ instead of $F$, which has the
same set of periods.
Lemmas~\ref{lem:caseC1}, \ref{lem:caseC2} and \ref{lem:caseC3} give
the conclusion in Case~(C).
In a similar but symmetric way Case~(D) holds.
This, together with Lemmas~\ref{lem:largegap} and \ref{lem:sortgapR},
gives at last Theorem~\ref{theo:0inInterior}.

\section{The set of periods of rotation number 0 --- some surprises}
\label{sec:0suprises}

For a lifting of a circle map $F \in \LL_1(\IR)$, the strategy to determine
$\Per(F)$ is to characterize $\Per(p/q,F)$ for every rational rotation
number $p/q$ (see \cite{ALM}).
The situation is different depending whether $p/q$
belongs to the interior of the rotation interval or to its boundary.
Assume that $p,q$ are coprime. If $p/q\in \Int(\Rot(F))$, it is known
that $\Per(p/q,F)=q\IN$. If $p/q\in\Bd(\Rot(F))$, there exists
$s\in\IN\cup\{\tinf\}$ such that $\Per(p/q,F)=q\cdot\Shs(s)$. In
both cases, the strategy is to prove the result for $0$ (i.e.
$p/q=0/1$) and then apply it to $G:=F^q-p$ to obtain the result for
$\Per(p/q,F)$. When one deals with the set of periods of a map
$F\in\Li$, the first, natural idea is to adopt the same strategy and,
first, (try to) characterize $\Per(0,F)$. However, this idea does
not work as expected, neither for $\Per(0,F)$, nor for the step
relating $\Per(p/q,F)$ to what can occur for $0$.

The aim of this section is to show the problems that can arise for
the rotation number $0$.
Recall that Theorem~\ref{theo:0inInterior} states that, if
$0\in\Int(\RotR(F))$, then $\Per(F)$ contains all integers except
maybe $1$ or $2$. Notice that this result
deals with all periods {\modi} and not true periods.
The conditions $p/q\in\Int(\RotR(F))$ and $0\in\Int(\RotR(F^q-p))$
are equivalent; but, whereas it is straightforward to deduce $\Per(p/q,F)$ from
$\Per(0, F^q-p)$, there is no easy way to determine $\Per(F)$ when one knows
$\Per(F^q-p)$.
On the other hand, Theorem~\ref{ConverseEndInteger} deals with
a difficulty arising for rotation numbers $p/q\in \Bd(\RotR(F))$ when
$p/q\notin\IZ$.

In all examples of this section, the map $F\in\Li$ will satisfy
$F(\IR)=S$,
and hence $\RotR(F)=\Rot(F)$ by \cite[Proposition~3.4]{AlsRue2008}.

\subsection{Per(0, F) when 0 is in the interior of the rotation interval}

The general rotation theory for a degree 1 map on an infinite tree
states that, if $0\in\Int(\RotR(F))$, there exists $n$ such that
$\Per(0,F)\supset\set{k\in\IN}{k\ge n}$
\cite[Theorem~3.11]{AlsRue2008}.
Unfortunately, the integer $n$ can be arbitrarily large, even for
the space $S$, as shown by the next example.

\begin{example}\label{ex:Per(0,F)}
\textbf{A map such that $0\in\Int(\RotR(F))$ and
$\Per(0,F)=\{k\in\IN\mid k\ge n\}$.}

We fix $n\ge 3$.
Let $b=\max B_0$ and choose $a\in (-1,0)$.
We define $F\in\Li$ such that $F(0)=-1$, $F(b)=
b+1$, $F(a)=b-n-1$ and $F$ is affine on $B_0$, $[-1,a]$ and $[a,0]$.
The map $F$ is illustrated in Figure~\ref{fig:Per(0,F)}.

\begin{figure}[htb]
\centerline{\includegraphics{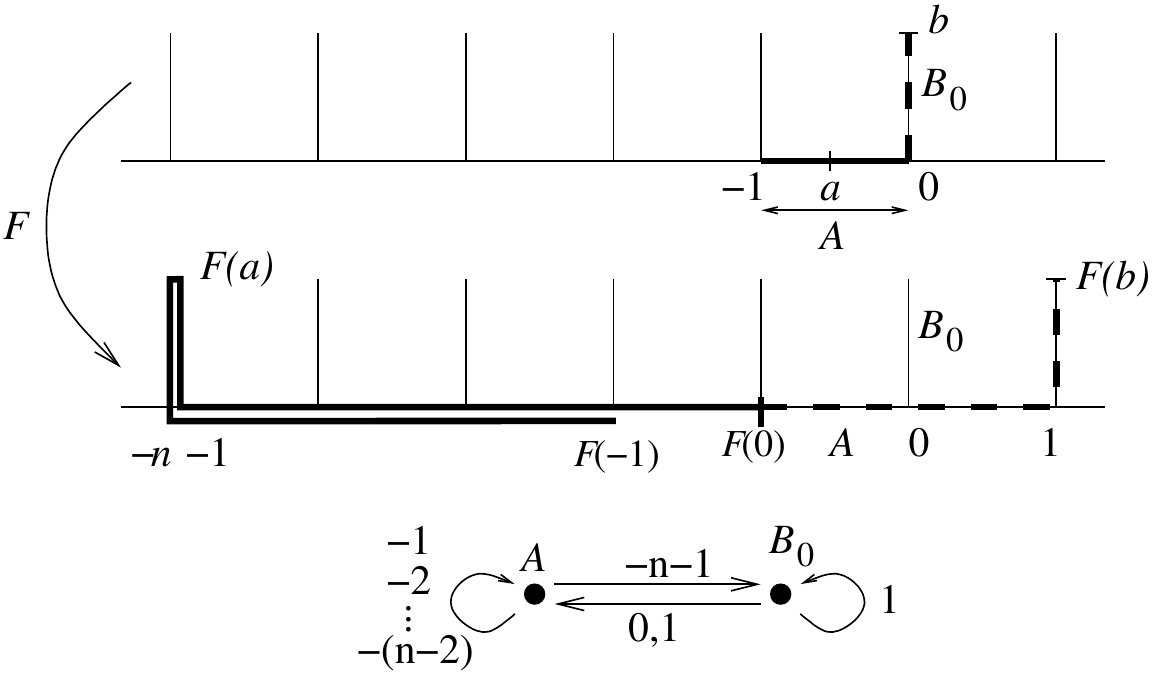}}
\caption{The map $F$ of Example~\ref{ex:Per(0,F)} and the covering
graph of $B_0$ and $A=[-1,0]$. The Markov graph can be easily deduced
from this graph by splitting $A$ into $[-1,a]$ and
$[a,0]$.}\label{fig:Per(0,F)}
\end{figure}

Using the Markov graph of $F$ and the tools from
\cite[Subsection~6.1]{AlsRue2008}, one can compute that
$\RotR(F)=\Rot(F)=[-(n-2),1]$ (which contains $0$ in its interior for
every $n\ge 3$)
and $\Per(0,F)=\{k\in\IN\mid k\ge n\}$.
\end{example}

\subsection{Sets of periods living in complicated trees
can be obtained for rotation number~0}

Although the whole space $S$ is an infinite tree, a periodic orbit
of rotation number $0$ is a true periodic orbit, and thus it
is compact and lives in a finite subtree of $S$. This makes possible
to study $\Per(0,F)$ by using the works on periodic orbits for finite
trees \cite{AGLMM, AJM4}.
In Section~\ref{sec:Y}, we saw that the sets $\Per(0,F)$ can display all
possible sets of periods of maps in $\Cstar$.
In this subsection, we show that the converse is not true: there
exist maps in $\Li$ with $0\in\RotR(F)$ and
such that $\Per(0,F)$ is not the set of periods
of a map in $\Cstar$. We are going to exhibit examples in which
$\Per(0,F)$ can be deduced from the set of periods of a tree map,
where the tree is more complicated than a $3$-star.

Let us introduce some notation.
Let $P$ be a true periodic orbit of $F\in\Li$.
We will denote by $T_P\subset S$ the finite tree defined by
\[
  T_P := \chull{\Re(P)} \cup
         \bigcup_{i \in \chull{r\circ P} \cap \Z} B_i.
\]
Observe that $T_P$ and the closure of $S \setminus T_P$ have at most
two points  in common:
$\min \Re(P) \in \IR$ and $\max \Re(P) \in \IR.$
Moreover, $\min \Re(P)$ and $\max \Re(P)$ are either points
of $P$ or branching points.

We also define the map {\map{F_P}{T_P}} by
$F_P := \ret_{T_P} \circ F\evalat{T_P},$
where $\ret_{T_P}$ is the standard retraction from $T$ to $T_P.$
More precisely, for every $x\in T_P,$
\[
F_P(x) = \begin{cases}
       F(x) & \text{if $F(x)\in T_P$,}\\
       \min \Re(P) & \text{if $\Re(F(x)) < \min \Re(P)$,}\\
       \max \Re(P) & \text{if $\Re(F(x)) > \max \Re(P)$.}
    \end{cases}
\]

Let $x\in T_P$. If $F^n(x)\in T_P$ for all $n\ge 0$, then the orbits
of $x$ under $F$ and $F_P$ coincide. In particular, $x$ is
$F$-periodic of period $k$ if and only if it is $F_P$-periodic of
period $k$. When the orbits of $x$ under $F$ and $F_P$ do not
coincide, it follows that $x$ is eventually mapped by $F_P$ either to
$\min \Re(P)$ or $\max \Re(P)$.
Therefore, these are the only points that may be periodic for  $F_P$
but not for $F$.
This leads to the next lemma, showing that it is worth studying
the set of periods of $F_P$.

\begin{lemma}\label{lem:Fp-F}
There exists $E\subset \IN$ with $\#E\le 2$ such that
$\TPer(F_P)\setminus E\subset \Per(0,F)$.
\end{lemma}

Now we briefly define (in a slightly restricted case)
the notions of  patterns and linear models introduced in \cite{AGLMM}
to study the sets of periods of tree maps. Let $T$ be a (finite) tree,
$P$ a finite subset of $T$ with at least two elements and
$\varphi$ a cyclic permutation of $P$.
The \emph{discrete components} of $P$ are the sets
$\overline{C_i}\cap P, i=1, \ldots, n$, where $C_1,\ldots, C_n$ are
the connected components of $\chull{P}\setminus P$.
If $x,y$ are two distinct elements of the same discrete component,
$\chull{x,y}$ is called a \emph{$P$-basic path}.
If $T'$ (resp. $P$, $\varphi'$) is also a  tree (resp. a finite subset
of $T'$ with at least two elements, a cyclic permutation of $P'$), we
write $(T,P,\varphi)\sim_{pat}(T',P',\varphi')$ if there exists a
bijection {\map{h}{P}[P']} such that $h\circ \varphi=\varphi'\circ h$
and $h$ preserves the discrete components. This gives an equivalence
relation; the equivalence class of $(T,P,\varphi)$ is denoted
$[T,P,\varphi]$ and is called a \emph{periodic pattern}.
If {\map{f}{T}} is a tree map, $P$ a periodic orbit of $f$ and $A$ a
periodic pattern, we say that $f$ \emph{exhibits $A$} over $P$ if
$[T,P,f\evalat{P}]=A$.
The set of periods \emph{forced} by a pattern $A$ is the maximal
subset $E_A\subset \IN$ such that every tree map exhibiting the
pattern $A$ also has periodic orbits of period $n$ for all $n\in E_A$.

The triple $(T,f,P)$ is called an \emph{$A$-linear model} if
\begin{itemize}
\item $f$ exhibits $A$ over $P$,
\item $f$ is monotone on all $P$-basic paths,
\item for every connected component $I$ of $T\setminus (P\cup V(T))$
(where $V(T)$ denotes the  set of vertices of $T$),
$f\evalat{\overline{I}}$ is affine.
\end{itemize}
Notice that the monotonicity on $P$-basic paths implies that the image
of each vertex $v$ is uniquely determined and belongs to $P\cup V(T)$
(consider three $P$-basic paths containing $v$ and their images in
order to find $f(v)$ -- see also \cite[Proposition~4.2]{AGLMM}). Thus
an $A$-linear model is Markov with respect to the partition
generated by $P\cup V(T)$. The $A$-linear model is the analogous of
the ``connect-the-dots'' map associated to a periodic orbit of an
interval map, but the difficulty for tree maps is that the linear
model may live in a different tree than the original one --- some of
the vertices may collapse or explode.

The key results are the following ones.
For every periodic pattern $A$, there exists an $A$-linear model (and
it is unique up to isomorphism) \cite[Theorem~A]{AGLMM}. Moreover, if
a tree map $f$ exhibits the periodic pattern $A$, then the set of
periods of significant periodic points of an $A$-linear model is
included in $\TPer(f)$ \cite[Corollary~B]{AJM4}. A periodic point is
called \emph{significant} if its orbit is not equivalent, by iteration
of the map, to the orbit of a vertex, see e.g. \cite{{AJM4}} for the
precise definition. Significant periodic points essentially correspond
to loops in the Markov graph, therefore the set of periods forced by a
periodic pattern $A$ can be computed using the Markov graph of an
$A$-linear model.

The characterization of the whole set of periods of a tree map uses
the $p$-orderings of Baldwin, where $p$ ranges in a finite set of
integers depending on the tree, in particular on the valences of the
vertices.  When the tree is a $k$-star, one may need the $p$-orderings
$\leso{p}$ for  $2\le p\le k$.

Let us come back to the map $F_P$ coming from a periodic orbit $P$ of
$F\in\Li$.  Although all the vertices of $T_P$ have valence 3, the
linear model of $[T_P,P,F_P\evalat{P}]$ may have vertices of
arbitrarily large valence. In Example~\ref{ex:n-star-model}, we show
that, for all $k\ge 3$, there exist $F\in\Li$ and $P$ a periodic orbit
of $F$ such that the linear model of $[T_P,P,F_P]$ lives in a $k$-star
and the $k$-th partial ordering of Baldwin is needed to express the
set of periods of $F_P$.  More complicated trees than stars can even
be obtained, as shown in Example~\ref{ex:2gluedstars-model}.

\begin{example}\label{ex:n-star-model}
Fix an integer $k\ge 3$. Choose $a\in (0,1)$ and
$b_0,b_1,\ldots, b_{k-1}\in B_0$ such that
$1=b_0>b_1>\cdots>b_{k-1}>0$. We set
$x_i=i+b_i\in B_i$ for all $0\le i\le k-1$ and $x_k=a+k-2\in\IR$.
In addition, we set
\[
 A_i = [b_{i+1}, b_i] \text{ for all } 0\le i\le k-2,\
 A_{k-1} = [0,b_{k-1}], \
 L= [0,a] \text{ and }
 R=[a,1].
\]
We define the map $F\in\Li$ such that $F(x_i)=x_{i+1}$ for all $0\le
i\le k-1$, $F(x_k)=x_0$, $F(1)=0$, $F$ is affine in restriction to
each of the intervals $L$, $R$ and $A_i, 0\le i\le k-1$, and the map
is defined on the rest of $S$ using degree 1. Then
$P=(x_0,x_1,\ldots,x_k)$ is a true periodic orbit of period $k+1$ for
$F$, and $F$ is linear Markov. The map $F$ and its Markov graph are
represented in Figure~\ref{fig:Fk}.

The map $F_P$ is defined on $T_P=B_0\cup\cdots\cup B_{k-1}\cup [0, k-1]$.
If $F_P(x)\neq F(x)$ then, either $F_P(x)=k-1$,
or $F_P(x)=0$. The point $0$ is fixed under $F_P$ and
$F_P^{k-1}(k-1)=0$
Thus $\TPer(F_P)\setminus\{1\}\subset \Per(0,F)$.

The linear model of $F_P$ is supported by a $k$-star; it is
represented in Figure~\ref{fig:model-Fk}.
To prove this fact, the easiest (but not most convincing) way is to
see that the map in Figure~\ref{fig:model-Fk} does exhibit the right
pattern, then the uniqueness of the linear model gives the
conclusion.
We leave to the interested readers the checking that the only way to
realize a linear model of $F_P$ is to collapse the $k-2$ vertices of
$T_P$. This can be done by looking at all basic paths and their
images.

\begin{figure}[!tb]
\centerline{\includegraphics[width=0.95\textwidth]{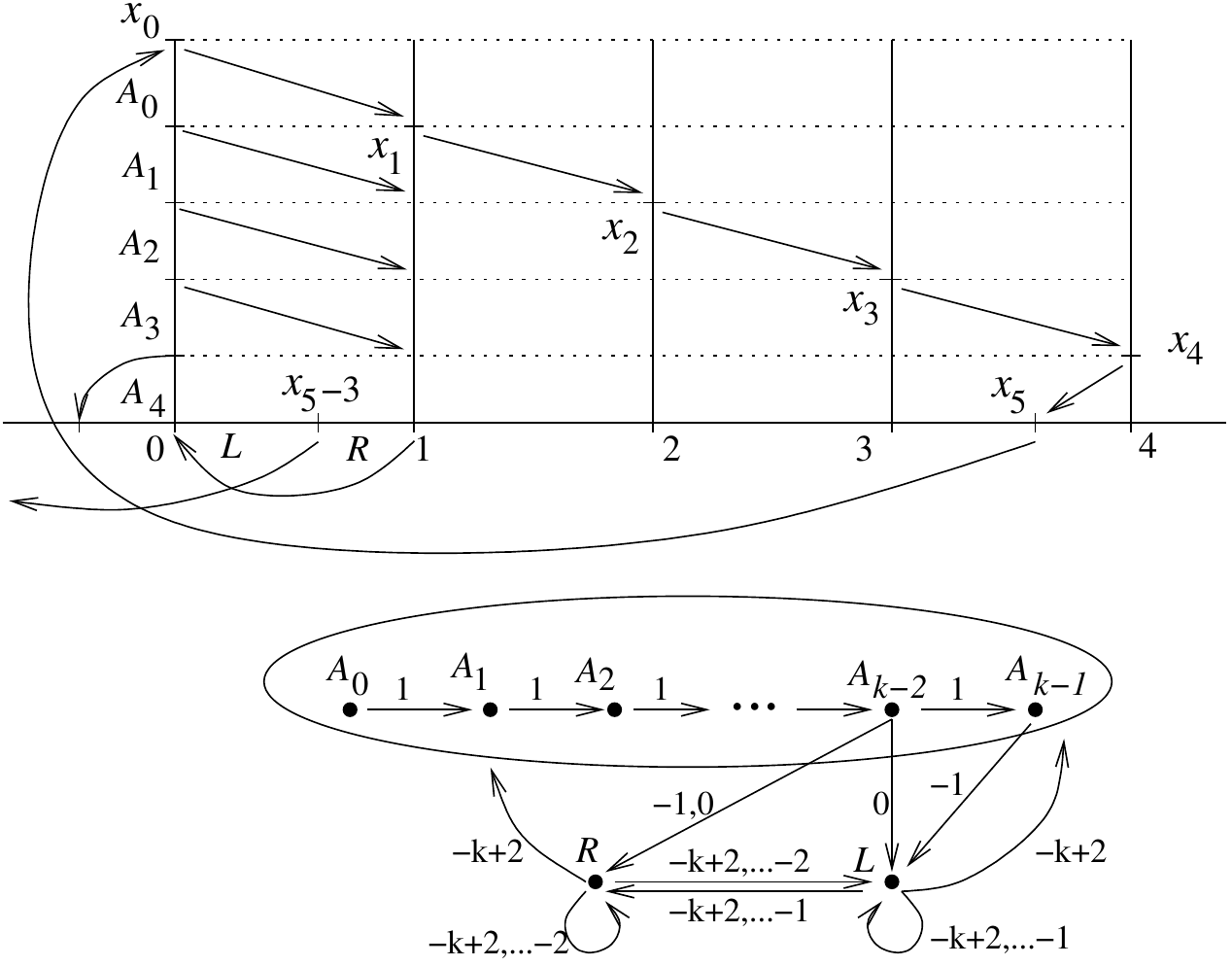}}
\caption{Above: the map $F$ from Example~\ref{ex:n-star-model}, which
is defined by its action on $x_0,\ldots,x_k$ and
$1$, and is piecewise linear on the partition generated by these
points $\modi$; picture is for $k=5$. Below: the Markov graph of $F$;
several integers on the
same arrow, as well as an arrow pointing to the ellipse containing
$A_0,\ldots, A_{k-1}$, are short-cuts indicating several
arrows.}\label{fig:Fk}

\vspace*{5.7ex}

\centerline{\includegraphics[width=0.95\textwidth]{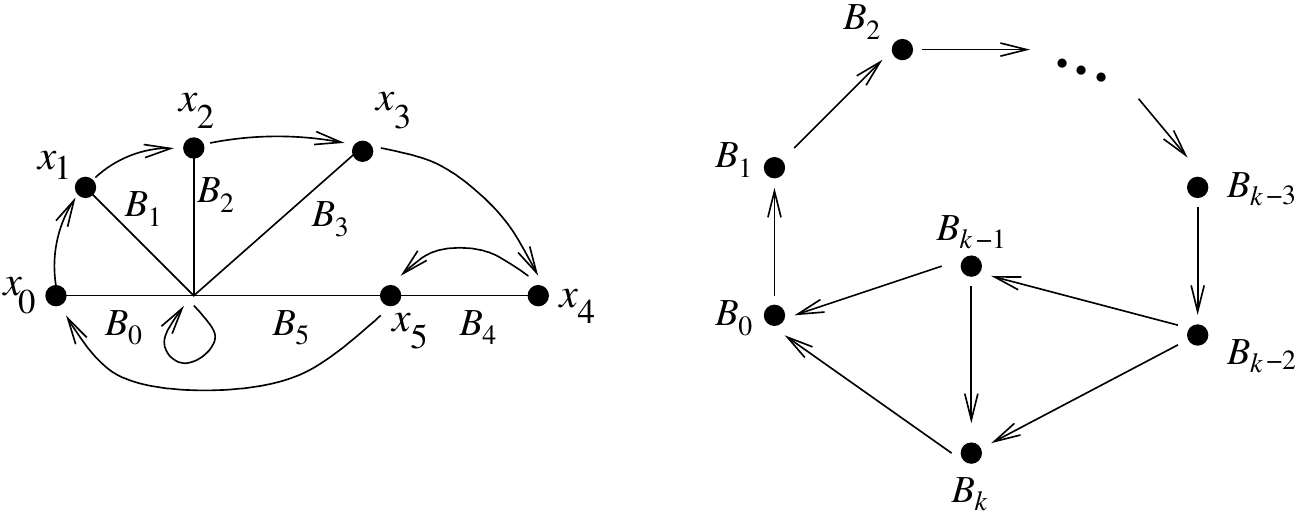}}
\caption{On the right: the linear model of $[T_P,P,F_P\evalat{P}]$,
the map being affine on each of the intervals $B_0,\ldots, B_k$
(picture is for $k=5$).
On the left: its Markov graph.}\label{fig:model-Fk}
\end{figure}

From the linear model, one can show that
the pattern $[T_P,P,F_P\evalat{P}]$ forces all the periods $n$ for
$n\le_k k+1$, where $\le_k$ is the $k$-ordering of Baldwin.
A direct computation from the Markov graph of $F$ gives
$\Rot(F)=[-k+2,0]$ and
\[
    \Per(0,F) = \{k,k+1\} \cup
        \set{ik+j(k+1)}{i,j \ge 1}=
    \set{n\in\IN}{n \le_k k+1} \setminus \{1\}.
\]
Therefore, the inclusions $\set{n\in\IN}{n\le_k k+1}\subset \TPer(F)$
and $\TPer(F)\setminus\{1\}\subset \Per(0,F)$ are equalities.
\end{example}

\begin{example}\label{ex:2gluedstars-model}
Given $p,q\ge 3$, it is possible to build a map $G\in\Li$ with a true
periodic orbit $P$ of period $p+2q-4$ such that the linear model of
$G_P$ lives in a tree consisting in a $p$-star glued to a $q$-star.
To remain readable, we illustrate the construction for $p=6$ and $q=7$
(hence the period of $P$ is $16$)
instead of giving the definition for arbitrary $p,q$. We choose
points $x_0\in (0,1)$ and $x_1,\ldots, x_{15}\in B$ as in
Figure~\ref{fig:2gluedstars}.
\begin{figure}[!tb]
\centerline{\includegraphics[width=0.95\textwidth]{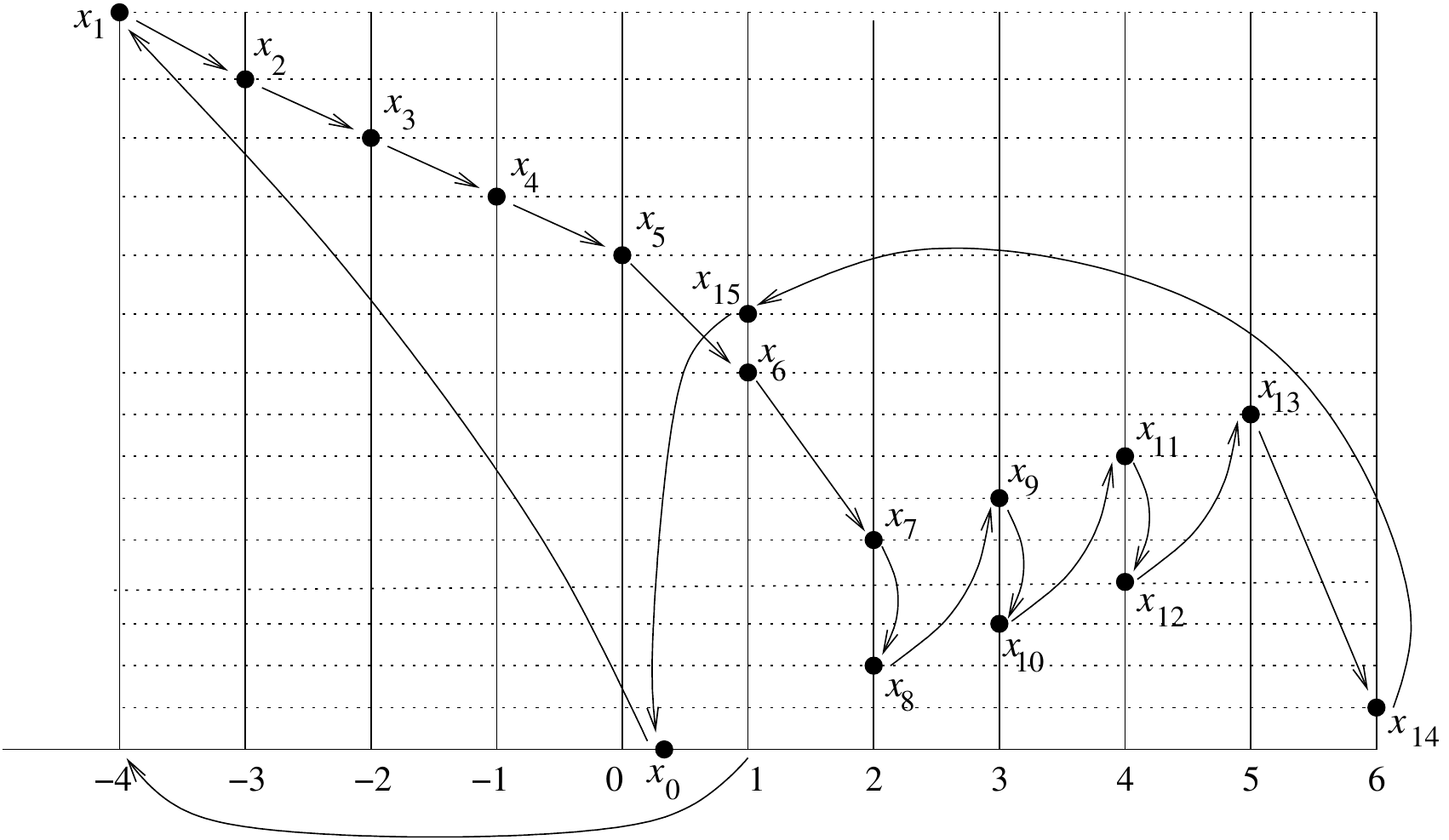}}
\caption{The map $G$ from Example~\ref{ex:2gluedstars-model}
and its periodic orbit $P=\{x_0,\ldots, x_{15}\}$; $G$
is of degree 1 and affine on each interval of the partition generated
by $(P\cup\{0\})+\IZ$.}\label{fig:2gluedstars}
\end{figure}
Then $G$ is defined by $G(x_i)=x_{i+1}$
for all $0\le i\le 15$, $G(x_{15})=x_0$, $G(0)=-5$ and $G$ is of
degree $1$ and affine on each interval of the partition generated by
these points $\modi$.
We do no
draw the Markov graph of $G$, which is rather big, but one may check
that
$\Rot(G)=[-5,1]$ (in the Markov graph, the endpoints of $\Rot(G)$ are
reached by the loops $[0,x_0]\signedcover[-5]{}[0,x_0]$ and, e.g.,
$[x_{7}+2,x_{12}]\signedcover[1]{}[x_{7}+2,x_{12}]$). The tree $T_P$
is
equal to $[-4,6]\cup\bigcup_{-4\le i\le 6}B_i$. The point $-4$
is fixed for $G_P$ and the point $6$ is sent to $-4$ by $G_P^2$.
Therefore, as in
Example~\ref{ex:n-star-model},
$\TPer(G_P)\setminus\{1\}\subset\Per(0,G)$.
The linear model of $G_P$ is represented in
Figure~\ref{fig:2gluedstars-model};
the $p-2$ vertices of $T_P$ less than or equal
to $0$ collapse into a fixed vertex, and the $q-2$ vertices greater
than or equal to $1$ collapse to another fixed vertex.
It is possible to compute
that the set of (significant) periods of the linear model is
$\{1\}\cup\{n\ge 6\}$ and that $\Per(0,F)=\{n\ge 6\}$.
\begin{figure}[!htb]
\centerline{\includegraphics{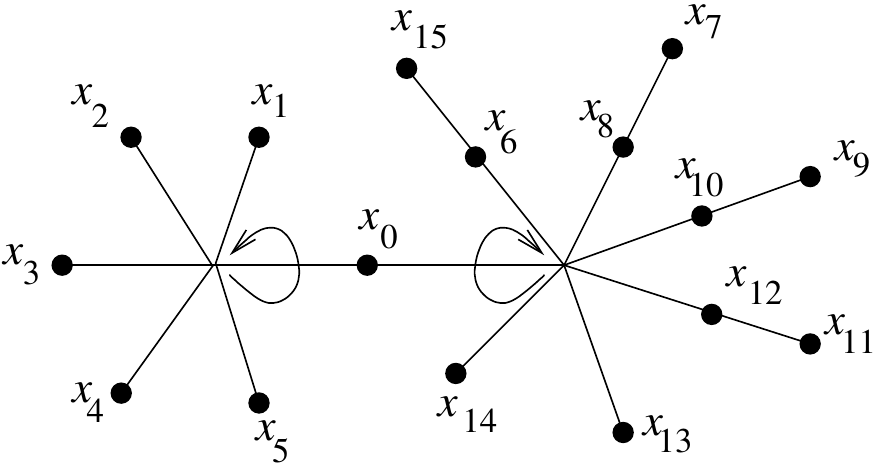}}
\caption{The linear model of $G_P$ (from
Example~\ref{ex:2gluedstars-model}): the points $x_0,\ldots, x_{15}$
are mapped cyclically, the two vertices are fixed and the map is
affine on each interval generated by this
partition.}\label{fig:2gluedstars-model}
\end{figure}
\end{example}

\clearpage
\providecommand{\href}[2]{#2}
\providecommand{\arxiv}[1]{\href{http://arxiv.org/abs/#1}{arXiv:#1}}
\providecommand{\url}[1]{\texttt{#1}}
\providecommand{\urlprefix}{URL }

\end{document}